\documentclass[a4paper, 12pt]{amsart}
 \usepackage[utf8]{inputenc}
\usepackage[all]{xy}
\usepackage{amsmath}
\usepackage{amsfonts}
\usepackage{amssymb}
\usepackage{graphicx}
\usepackage{xy}
\usepackage{latexsym}
 \usepackage{float}
 \makeatletter
\def\@captype{figure}
\def\@captype{figure}
\def\captionof#1#2{{\def\@captype{#1}#2}}
\makeatother

\usepackage{pstricks,pst-node,pst-tree,pstricks-add,pst-all}

\usepackage{multirow}

\usepackage{array}
\date{}

\newtheorem{theo}{Th\'eor\`eme}
\newtheorem{lem}[theo]{Lemme}
\newtheorem{prop}[theo]{Proposition}

\newtheorem{defi}[theo]{Definition}

\newtheorem{rem}[theo]{Remarque}
\newtheorem{nota}[theo]{Notations}

\newcommand{\beq}{\begin{equation}}
\newcommand{\eeq}{\end{equation}}

\newcommand{\h}{{\mathbb{H}}}
\newcommand{\Z}{{\mathbb{Z}}}
\newcommand{\R}{{\mathbb{R}}}

\newcommand{\cut}{\hfil\break}

\newcommand{\del}{\partial}
\newcommand{\tq}{{\ |\ }}
\newcommand{\tens}{{\otimes}}

\newcommand{\type}{\small{ \in}}

\begin{document}

\title[L'anneau de cohomologie \`a coefficients dans $\Z_p$]{ L'anneau de cohomologie  de  toutes les vari\' et\' es de Seifert}

\footnotetext[1]{\noindent{ MSC-class:  13D03, 57S25 (Primary),
 55N45 (Secondary)}}

\author{A. Bauval,  C. Hayat}


\begin{abstract}

\noindent

\emph{The cohomology ring with coefficients in $\Z_p$, where $p$ is a prime integer, of a Seifert manifold $M$, orientable or not orientable is obtained from a $\Delta$-simplicial decomposition of $M$. Many choices must be made before applying 
Alexander-Whitney formula to get the cup-products. The most difficult choices are those of  liftings from the cellular complex to the $\Delta$-simplicial complex when we add the condition to be cocycles.}

-----

\emph{L'anneau de cohomologie à coefficients dans $\Z_p$, où $p$ est un nombre premier,  d'une variété de Seifert $M$, orientable ou non-orientable est obtenu à partir d'une décomposition $\Delta$-simpliciale de $M$. Plusieurs choix sont à faire avant d'appliquer la formule d'Alexander-Whitney qui permet d'obtenir les cup-produits. Les choix les plus délicats sont ceux de relevés du complexe cellulaire dans le complexe simpliciale auxquels on impose d'être des cocyles $\Delta$-simpliciaux.}
\end{abstract}

\maketitle


\section{ Introduction, Notations} \label{section:IN}

\subsection{Inroduction} Dans cet article, on détermine   l'anneau de  cohomologie à valeurs dans $\Z_p$, où $p$ est un entier premier, pour toutes les variétés de Seifert orientables ou non en  utilisant un point de vue $\Delta$-simplicial qui, bien que long et très combinatoire, se montre  très efficace. Cette efficacité a \'et\'e tr\`es bien illustr\'ee par A. Hatcher \cite{hatcher}, voir aussi \cite{jpb}.  Dans tout le reste de l'article, on omettra le symbole $\Delta$ devant le mot "sympliciale" ou "simplexe". L'anneau de cohomologie des variétés de Seifert orientables a déjà été étudié dans plusieurs articles \cite{bhzz1}, \cite{bhzz2}, \cite{bz}, mais pas celui des  variétés de Seifert non-orientables.
Notre attention pour le calcul des cup-produits a été attirée par les études portant sur une  extension du  Th\'eor\`eme de Borsuk-Ulam pour les vari\'et\'es de dimension 3 \cite{ghz}. En effet si $\tau $ est une involution sur une telle vari\'et\'e $M$, alors toute application continue $f$ de $M$ dans $\R^3$ admet un point $x\in M$ tel que $f(\tau(x))=x$ si et seulement si la puissance trois pour le cup-produit de la classe de cohomologie à coefficients dans $\Z_2$ associ\' ee \`a $\tau$  est non nulle.

Décrivons les étapes de la  m\'ethode que nous avons choisie.

Nous commen\c cons par construire une d\'ecomposition cellulaire de la vari\'et\'e $M$, (Section \ref{section:decomsimp}), en pr\'ecisant dans \ref{sub:wab} les mots qui permettent de paver les 2-cellules, bords du voisinage tubulaire des fibres singuli\`eres et de la  derni\`ere 3-sph\`ere.
Le complexe cellulaire $(C_*)_{cell}$ ainsi obtenu est subdivisé en un complexe simplicial $(C_*)_{simp}$.

Notons $T\colon (C_*)_{cell}\to (C_*)_{simp}$ le morphisme associé à cette subdivision, et $T^t\colon C^*_{simp}\to C^*_{cell}$ le morphisme   transpos\'e   (Sous-section \ref{sub:mor}). Nous
choisissons, pour chaque  g\'en\'erateur $\xi$ des cochaînes cellulaires,  un relev\'e $R(\xi)$ de $T^t(\xi)$ dans les cochaînes simpliciales.
Ce choix est fait de telle fa\c con que $R(\xi)$ soit un cocycle et pas seulement une cocha\^ine, (Section \ref{section:relev}).

Le cup-produit de deux cocha\^\i nes simpliciales est calcul\'e par la formule d'Alexander-Whitney.  Pour les cup-produits d'\'el\'ements de $H^1\tens H^1$, on d\'ecrit un calcul qui permet d'\'eviter l'évaluation sur les (nombreux~!) 2-simplexes, (Sous-sections \ref{sub:coef}, \ref{sub:11=2,2} et \ref{sub:11=2,p}). Pour les cup-produits d'\'el\'ements de $H^1\tens H^2$, on applique la formule d'Alexander-Whitney de fa\c con plus classique puis le  quasi-isomorphisme $T^t$, (Sous-sections \ref{sub:12=3,2} et \ref{sub:12=3,p}).

\vspace{1cm}

Le plan de cet article est comme suit. Apr\`es cette section d'introduction, de notations sur les vari\'et\'es de Seifert et de quelques invariants associ\'es, la section \ref{section:gps} d\'ecrit une d\'ecomposition cellulaire d'une vari\'et\'e de Seifert quelconque $M$ et donne une pr\'esentation des groupes de cohomologie $H^*(M,\Z_p)$. Dans la section \ref{section:theprin}, le th\'eor\`eme principal  pr\'esente tous les cup-produits sous forme de tableaux. La preuve de ces r\'esultats constitue le reste de l'article. Dans les sections et sous-sections \ref{section:decomsimp}, \ref{sub:mor}, \ref{section:relev} sont d\'ecrits les choix faits pour obtenir une d\'ecomposition simpliciale, le quasi-isomorphisme $T$, et les relev\'es des cocycles cellulaires en cocycles simpliciaux.

Dans la section \ref{section:p=2}, on  applique, via la formule d'Alexander-Whitney,  tous ces choix pour le calcul des cup-produits lorsque les coefficients de la cohomologie sont \'egaux \`a 2, et dans la section \ref{section:p>2} lorsque les coefficients de la cohomologie sont \'egaux \`a un entier $p$ premier, $p>2$.

La section \ref{section:fig} est faite de figures symbolisant les d\'ecompositions cellulaires et simpliciales.

\subsection{Notations}\label{intro:Nota}

En suivant les notations d'Orlik \cite{s}, \cite{st}, \cite{orlik}, une variété de Seifert $M$  est décrite par une liste d'invariants de Seifert $$\{e; (\type, g); (a_1,b_1),\ldots ,(a_m,b_m)\}.$$ Ici $e$ est un entier, le type  $\type$ sera décrit plus bas, $g$ est le genre de la surface de base (l'espace des orbites obtenues en identifiant chaque fibre $S^1$ de $M$ à un point),  et pour chaque $k$, les entiers $a_k,b_k$ sont premiers entre eux avec  $a_k\ne 0$ (si $b_k=0$ alors $a_k=\pm 1).$

Comme Orlik dans \cite{orlik}, p.74 ( et aussi comme chez d'autres auteurs), nous introduisons une fibre supplémentaire, non-exceptionnelle $a_0=1,b_0=e$ et utilisons la présentation suivante du groupe fondamental de $M$ :
\begin{eqnarray}
\pi_1(M)=\left<{\begin{matrix}q_0,\ldots , q_m\\v_1,\ldots  ,v_{g'}\\h\end{matrix}\left|
\begin{matrix}
 [q_k,h]\quad\text{and}\quad q_k^{a_k}h^{b_k},&0\le k\le m\\
v_jhv_j^{-1}h^{-\varepsilon_j},&1\le j \le g'\\
q_0\ldots q_mV&\end{matrix}\right.}\right>, \label{BB}
\end{eqnarray}
où les générateurs et $g', V$ sont décrits ci-dessous.

\begin{itemize}
\item Le type $\type$ de $M$ est égal à:
	\begin{itemize}
	\item[$o_1$] si la surface de base et l'espace total sont orientables (forcément tous les $\varepsilon_j$ sont égaux à  $1$);
	\item[$o_2$] si la surface de base est orientable et l'espace total  non-orientable, 
alors $g\ge 1$ (forcément tous les $\varepsilon_j$ sont égaux à  $-1$);
	\item[$n_1$] si la surface de base et l'espace total sont non-orientables 
alors $g\ge 1$ et de plus tous les $\varepsilon_j$ sont égaux à $1$;
	\item[$n_2$] si la surface de base est  non-orientable (alors $g\ge 1$) et l'espace total est orientable (forcément tous les  $\varepsilon_j$ sont égaux à $-1$);
	\item[$n_3$] si la surface de base et l'espace total sont non-orientables avec de plus, 
tous les $\varepsilon_j$ égaux à $-1$ sauf $\varepsilon_1=1$, et $g\ge 2$;
	\item[$n_4$] si la surface de base et l'espace total sont non-orientables avec de plus, 
tous les $\varepsilon_j$ égaux à $-1$ sauf $\varepsilon_1=\varepsilon_2=1$, et $g\ge 3$.
	\end{itemize}
\item L'orientabilité de la surface de base  et le genre $g$ déterminent le nombre $g'$ 
de générateurs $v_j$ et le mot $V$ dans la longue relation de $\pi_1(M)$ de la façon suivante:
	\begin{itemize}
	\item quand la surface de base est orientable, i.e. $\type=o_i$, $g'=2g$ et  $V=[v_1,v_2]\ldots[v_{2g-1},v_{2g}]$;
	\item quand la surface de base est non-orientable,  i.e. $\type=n_i$, $g'=g$ et $V=v_1^2\ldots v_g^2$.
	\end{itemize}
\item Le générateur  $h$ correspond à la fibre générique régulière.
\item Les générateurs $q_k$ pour $0\le k\le m$ correspondent aux (possibles) fibres exceptionnelles.
\end{itemize}

Dans ce papier nous utiliserons les notations suivantes.

\begin{nota}\label{nota:cd} Soit $M$ une variété de  Seifert décrite par une liste d'invariants de  Seifert 
  $$\{e; (\type, g); (a_1,b_1),\ldots ,(a_m,b_m)\},$$
et soit $p$ un entier premier.
\begin{itemize}
\item  Notons $a$ le plus petit commum multiple des $a_k$, alors
$$c=\sum_{k=0}^m b_k(a/a_k).$$
\item Le nombre de $a_k$ divisibles par $p$ sera noté $n$.
\begin{itemize}
\item Quand $n=0$, on suppose que $b_k$ est divisible par $p$ si et seulement si $0\le k\le r$;
\item quand $n>0$, on suppose que $a_k$ est divisible par $p$ si et seulement si $0\le k\le n-1$, les indices $k$ sont réordonnés par  p-valuation  décroissante $\nu_p(a_k)$. 
\end{itemize}
\item On distingue trois cas:
	\begin{itemize}
	\item Cas 1, \ \ $n=0$ et $c$ est divisible par $p$;
	\item Cas 2, \ \ $n=0$ et $c$ n'est pas divisible par $p$;
	\item Cas 3, \ \ $n>0$.
	\end{itemize} 
\item Le symbole $*$ est égal à $4g$ pour  $\type =o_i$,  et à $2g$ pour  $\type =n_i.$
\end{itemize}
\end{nota}


\section{Les groupes de cohomologie}\label{section:gps}

\subsection{ Le complexe cellulaire}\label{sub:cell}

\bigskip

La vari\'et\'e de Seifert $M$ admet une d\'ecomposition cellulaire en cellules de dimension de $0$ \` a $3$ qui est d\'ecrite ci-dessous. Voir Figures 2, 3, 4, 5, Section \ref{section:fig} : \\
$\bullet$ une 0-cellule $\sigma$;\\
$\bullet$ des 1-cellules (d'origine et d'extr\' emit\' e $\sigma$) $t_j, q_k,h$;\\
$\bullet$ des 2-cellules :
\begin{itemize}
\item[-] $\delta$ de "bord":

 $\prod[t_{2i-1},t_{2i}]\prod q_k$ pour les Types $o_i$; $\prod t_j^2\prod q_k$ pour les Types $n_i;$
\item[-] $\rho_k$ tores de bords $[h,q_k]$;
\item[-]  $\nu_j$ tores si $\varepsilon_j=1$ ou bouteilles de Klein si $\varepsilon_j=-1$, de bords $ht_jh^{-\varepsilon_j}t_j^{-1}$;
\item[-]  $\mu_k$ disques de bords $w_{a_k,b_k}(q_k,h)$, qui est un mot en $q_k,h$ comportant $a_k$ fois la lettre $q_k$ et $b_k$ fois la lettre $h$, mais dans un ordre tr\`es particulier qui sera pr\' ecis\' e plus loin (Sous-section \ref{sub:wab});
\end{itemize}
$\bullet$ des 3-cellules :
\begin{itemize}
\item[-] $\epsilon$ dont le bord  (en tant que complexe cellulaire) est  pav\' e par deux exemplaires de $\delta$ et de chaque $\nu_j$ et un exemplaire de chaque $\rho_k$,
\item[-] $\zeta_k$ dont le bord est  pav\' e par deux exemplaires de $\mu_k$ et un exemplaire de $\rho_k$. Ce pavage, assez d\' elicat, est li\' e \`a une propri\' et\' e essentielle du mot $w_{a_k,b_k}(q_k,h)$, qui sera expliqu\' e plus loin dans le th\'eor\`eme (\ref{theo:ab}).
\end{itemize}

\subsection{Groupes et g\' en\' erateurs de $H^*(M,\Z_p)$}

\bigskip

La cocha\^ine duale d'une cellule $x$ est not\'ee $\hat x\in C^*(M;\Z)$. Sauf pr\'ecision, les indices $j$ sont des entiers v\'erifiant $1\le j\le g'$ et sont absents si $g=0$.

\begin{theo} \label{theo:gps}{\bf  Pour $p=2$.}\\
 Les groupes de cohomologie \` a coeffcients dans $\Z_2$ sont :
 \begin{itemize}
\item   $H^0(M,\Z_2)=\Z_2$ et $H^3(M,\Z_2)=\Z_2\{\gamma\}.$
\end{itemize}
 $H^1(M,\Z_2)$ et $H^2(M,\Z_2)$ d\'ependent des Cas 1,2,3 et non des Types:
 \noindent -- Cas 1 
\begin{itemize}
\item  $H^1(M,\Z_2)=\Z_2^{g'+1}=\Z_2\{\theta_j, 1\le j\le g';\alpha\} $ o\`u $\theta_j=[\hat t_j]$ et  $\alpha=[\hat h-\sum b_ka_k^{-1}\hat q_k] ;$
\item $H^2(M,\Z_2)=\Z_2^{g'+1}=\Z_2\{\varphi_j, 1\le j\le g';\beta\}$ o\`u $\varphi_j=[\hat\nu_j]$ et $\beta=[\hat\delta]$.
\end{itemize}
-- Cas 2 
\begin{itemize}
\item  $H^1(M,\Z_2)=\Z_2^{g'}=\Z_2\{\theta_j\}$ o\`u $\theta_j=[\hat t_j] ;$ 
\item   $H^2(M,\Z_2)=\Z_2^{g'}=\Z_2\{\varphi_j\}$ o\`u $\varphi_j= [\hat\nu_j]$.
\end{itemize}
-- Cas 3 
\begin{itemize}
\item  $H^1(M,\Z_2)=\Z_2^{g'+n-1}=\Z_2\{\theta_j,1\le j\le g'; \alpha_k, 0<k\le n-1\}$ o\`u $\theta_j=[\hat t_j]$ et  $\alpha_k=[\hat q_k-\hat q_0]$;  \\
\item  $H^2(M,\Z_2)=\Z_2^{g'+n-1}=\Z_2\{\varphi_j,1\le j\le g'; \beta_k, 0<k\le n-1\}$ o\`u $\varphi_j=[\hat\nu_j]$ et  $\beta_k=[\hat\mu_k]$.
\end{itemize}

{\bf Pour $p>2$.}\\
Les groupes de cohomologie \` a coefficients dans $\Z_p$ sont :
\begin{itemize}
\item $H^0(M,\Z_p)=\Z_p$, et $H^3(M,\Z_p)=\Z_p\{\gamma\}$ pour $o_1$ et $n_2$ tandis que $H^3(M,\Z_p)=0$ pour $o_2$,$n_1$,$n_3$,$n_4$.
\end{itemize}
$H^ 1(M,\Z_p)$ et $H^ 2(M,\Z_p)$ d\'ependent des  Cas 1,2,3 et du Type :\\
 {\bf Type $o_1$} Les r\'esultats sont les m\^ emes que lorsque $p=2$ avec $g'=2g$ \'eventuellement nul.\\
 {\bf Type $o_2$}\\
 -- Cas 1,2 
\begin{itemize}
\item   $H^ 1(M,\Z_p)=\Z_p^{2g}=\Z_p\{\theta_j, 1\leq j\leq 2g\}$ o\`u $\theta_j=[\hat t_j];$ 
\item $H^ 2(M,\Z_p)=\Z_p^{2g-1}=\Z_p\{\varphi_j , 2<j\leq 2g; \beta\}$ o\`u $\varphi_j=[\hat\nu_j+(-1)^j\hat\nu_1]$  et $\beta=[\hat\delta]$.
\end{itemize}
-- Cas 3 
\begin{itemize}
\item $H^ 1(M,\Z_p)=\Z_p^{2g+n-1}=\Z_p\{\theta_j, 1\leq j\leq 2g; \alpha_k, 0< k\leq n-1\}$ o\`u $\theta_j=[\hat t_j]$ et  $\alpha_k=[\hat q_k-\hat q_0]$;
\item $H^ 2(M,\Z_p)=\Z_p^{2g+n-1}=\Z_p\{\varphi_j, 2< j\leq 2g; \beta_k, 0\leq k\leq n-1\}$ o\`u
 $\varphi_j=[\hat\nu_j+(-1)^j\hat\nu_1]$  et  $\beta_k=[\hat\mu_k]$.
\end{itemize}
 {\bf  Type $n_1$} \\
-- Cas 1,2 
\begin{itemize}
\item $H^1(M,\Z_p)=\Z_p^g=\Z_p\{\theta_j, 1< j\leq g;\alpha\}$ o\`u  $\theta_j=[\hat t_j-\hat t_1]$ et $\alpha=[{c\over 2a}\hat t_1+\hat h-\sum b_ka_k^{-1}\hat q_k]$, la constante $c$ \'etant \'egale \`a $0$ dans le Cas 1 ; 
 \item $H^2=\Z_p^{g-1}=\Z_p\{\varphi_j, 1< j\leq g\}$ o\`u $\varphi_j=[\hat\nu_j-\hat\nu_1]$. 
 \end{itemize}
 -- Cas 3  
 \begin{itemize}
\item $H^ 1(M,\Z_p)=\Z_p^{g+n-1}=\Z_p\{\theta_j,1< j\leq g;\alpha_k,0\leq k\leq n-1\}$ o\`u $\theta_j=[\hat t_j-\hat t_1]$ et  $\alpha_k=[\hat q_k-\frac{1}{2}\hat t_g]$ ;
\item $H^ 2(M,\Z_p)=\Z_p^{g+n-2}=\Z_p\{\varphi_j, 1< j\leq g;\beta_k, 0< k\leq n-1\}$ o\`u $\varphi_j=[\hat\nu_j-\hat\nu_1]$ et  $\beta_k=[\hat\mu_k]$ .
\end{itemize}
 {\bf Type $n_2$}\\
 -- Cas 1,2,3  
\begin{itemize}
\item $H^1(M,\Z_p)=\Z_p^{g-1+n}=\Z_p\{\theta_j, 1< j\leq g;\alpha_k, 0\leq k\leq n-1\}$ o\`u $\theta_j=[\hat t_j-\hat t_1]$ et   $\alpha_k=[\hat q_k-{1\over 2}\hat t_g]$ ;
\item $H^2(M,\Z_p)=\Z_p^{n+g-1}=\Z_p\{\varphi_j,1< j\leq g;\beta_k,0\leq k\leq n-1\} $  o\`u $\varphi_j=[\hat\nu_j]$ et  $\beta_k=[\hat\mu_k]$.
\end{itemize}
 {\bf Type $n_3$} \\
  -- Cas 1,2,3  
\begin{itemize}
\item $H^ 1(M,\Z_p)$ a les m\^ emes g\'en\'erateurs que pour le Type $n_2 ;$ 
\item $H^2(M,\Z_p)=\Z_p^{g-2+n}=\Z_p\{\varphi_j, 2< j\leq g ;\beta_k, 0\leq k\leq n-1\}$  o\`u $\varphi_j=[\hat\nu_j]$ et  $\beta_k=[\hat\mu_k]$.
\end{itemize}
 {\bf Type $n_4$} \\
 -- Cas 1,2,3  
\begin{itemize}
\item $H^ 1(M,\Z_p)$ a les m\^ emes g\'en\'erateurs que pour le Type $n_2 ;$ \\
\item $H^2(M,\Z_p)=\Z_p^{g-2+n}=\Z_p\{\varphi_3; \varphi_j,3< j\leq g; \beta_k, 0\leq k\leq n-1\}$  o\`u $\varphi_3=[\hat\nu_2-\hat\nu_1]$,  $\varphi_j=[\hat\nu_j]$ et  $\beta_k=[\hat\mu_k]$.
\end{itemize}

\end{theo}

Les quatre lemmes suivants constituent la preuve de ce th\'eor\`eme. Ils  d\'etaillent  les bords des cha\^ ines, les bords des cocha\^ ines et les expressions de $H^ 1$ et $H^ 2$ pour un $p$ premier quelconque.\\

De la d\' ecomposition cellulaire (et des figures qui la pr\' ecisent), on d\' eduit la description suivante des bords des cha\^ ines :
\begin{lem}{ Bord des cha\^\i nes cellulaires}\\
$\del\sigma=0$, $\del t_j=\del q_k=\del h=0$, $\del\rho_k=0$,\cut
$\del\delta=\sum q_k$ pour $o_i$, $\del\delta=2\sum t_j+\sum q_k$ pour $n_i$,\cut
$\del\nu_j=0$ lorsque $\varepsilon_j=1$, $\del\nu_j=2h$ lorsque $\varepsilon_j=-1$,\cut
$\del\mu_k=a_kq_k+b_kh$,\cut
$\del\epsilon=\sum\rho_k$ pour $o_1$, $\del\epsilon=\sum\rho_k+2\sum(-1)^j\nu_j$ pour $o_2$, $\del\epsilon=\sum\rho_k+2\sum_{\epsilon_j=1}\nu_j$ pour $n_i$,\cut
$\del\zeta_k=-\rho_k$.
\end{lem}

Par dualit\'e, on obtient les bords des cocha\^ ines.
\begin{lem}{ Bord des cocha\^\i nes cellulaires}\\
$\del\hat\sigma=0$, $\del\hat\delta=\del\hat\mu_k=0$, $\del\hat\epsilon=\del\hat\zeta_k=0$,\cut
$\del\hat t_j=0$ pour $o_i$, $\del\hat t_j=2\hat\delta$ pour $n_i$,\cut
$\del\hat q_k=\hat\delta+a_k\hat\mu_k$,\cut
$\del\hat h=\sum b_k\hat\mu_k+2\sum_{\varepsilon_j=-1}\hat\nu_j$,\cut
$\del\hat\nu_j=0$ pour $o_1$,$n_2$, et pour $n_1$,$n_3$,$n_4$ lorsque $\varepsilon_j=-1$,\cut
$\del\hat\nu_j=2(-1)^j\hat\epsilon$ pour $o_2$, $\del\hat\nu_j=2\hat\epsilon$ pour $n_1$,$n_3$,$n_4$ lorsque $\varepsilon_j=1$,\cut
$\del\hat\rho_k=\hat\epsilon-\hat\zeta_k$.
\end{lem}

Quel que soit l'anneau de coefficients $A$, $H^0=A$ est engendr\' e par $1:=[\hat\sigma]$.
Le groupe  $H^3$ est \'egal \` a $A$ pour $o_1$ et $n_2$, et \` a  $A/2A$ pour $o_2$,$n_1$,$n_3$,$n_4$. Il est  engendr\' e par $\gamma:=[\hat\epsilon]=[\hat\zeta_k]$. 

\begin{lem}{ Pr\'esentation du groupe $H^1(M,\Z_p)$}\\
1) Quelque soit l'anneau de coefficients $A$, on a\\
$H^1=\{x\hat h+\sum y_j\hat t_j+\sum z_k\hat q_k\tq x,y_j,z_k\in A, (*)\}$ o\`u la condition $(*)$ de cocycle est
$$\forall k, a_kz_k+b_kx=0,$$
avec en plus,
pour $n_i$, $\sum z_k=0$ ; 
pour $o_i$, $\sum z_k=-2\sum y_j$ ;
et  pour $o_2$,$n_2$,$n_3$,$n_4$, $2x=0$.\\
2) Si l'anneau  $A=\Z$, ou $\Z_p$ avec $p$ premier $>2$, ceci se simplifie  en:\cut
pour $o_2$, $H^1=A^{2g}\times\{\sum z_k\hat q_k\tq z_k\in A, \forall k, a_kz_k=0, \sum z_k=0\}$ ;\cut
pour $n_2$,$n_3$,$n_4$, $H^1=\{\sum y_j\hat t_j+\sum z_k\hat q_k\tq y_j,z_k\in A, \forall k, a_kz_k=0, 2\sum y_j+\sum z_k=0\}$.\\
3) Si  l'anneau $A=\Z_2$, ceci se simplifie pour tous les Types en\cut
$H^1=\{x\hat h+\sum y_j\hat t_j+\sum z_k\hat q_k\tq x,y_j,z_k\in A, \forall k, a_kz_k+b_kx=0, \sum z_k=0\}$.
\end{lem}

\begin{lem}{ Pr\'esentation du groupe $H^2(M,\Z_p)$}\\
$H^2=\{x\hat\delta+\sum y_j\hat\nu_j+\sum z_k\hat\mu_k\tq x,y_j,z_k\in A, (*)\}/Im(\del)$ o\`u la condition $(*)$ est\cut
vide pour $o_1$, $n_2$ ; 
 pour $o_2$, $2\sum(-1)^jy_j=0$ ;
pour $n_1$,$n_3$,$n_4$, $2\sum_{\varepsilon_j=1}y_j=0$,\cut
et $Im(\del)$ est engendr\' e par les $\hat\delta+a_k\hat\mu_k$, $\sum b_k\hat\mu_k+2\sum_{\varepsilon_j=-1}\hat\nu_j$, et de plus $2\hat\delta$ pour $n_i$.
\end{lem}

\vspace{0.5cm}


 \section{Le th\'eor\` eme principal}\label{section:theprin}
 
 \subsection{Tableau des dimensions}\label{nota:wab}

\begin{center}
\begin{scriptsize}
\begin{tabular}{|c|c|c|c|} 
\hline
Type& Cas& $H^1$&$H^2$\\
\hline
$o_1$&$1$&$(2g)+(1)$&$(2g)+(1)$ \\ 
&$2$&$(2g)+(0)$&$(2g)+(0)$\\
&$3$&$(2g)+(n-1)$&$(2g)+(n-1 )$\\
\hline
$o_2$&$1,2$&$(2g)+(0)$&$(2g-2)+(1)$\\
&$3$&$(2g)+(n-1)$&$(2g-2)+(n)$\\
\hline
$n_1$&$1,2$&$(g-1)+(1)$&$(g-1)+(0)$\\
&$3$&$(g-1)+(n)$&$(g-1)+(n-1)$\\
\hline
$n_2$&&$(g-1)+(n)$&$(g-1)+(n)$\\
\hline
$n_3,n_4$&&$(g-1)+(n)$&$(g-2)+(n)$\\
\hline
\end{tabular}
\end{scriptsize}
\end{center}
\captionof{table}{\caption {Dimensions de $H^1$ et $H^2$ pour $p>2$. Pour $p=2$ elles sont, pour {\bf tous} les Types, comme le Type $o_1$ pour $p>2$}} \label{capt:dim}


 Le tableau \ref{capt:dim} donne les dimensions de $H^i(M;\Z/p\Z)$. Ces dimensions sont exprim\' ees comme somme de deux parenth\`eses, l'une li\' ee au genre $g'$,  l'autre au nombre $n$ de $a_k$ qui sont divisibles par $p$.
  Lorsque la seconde parenth\`ese est (1), le g\' en\' erateur suppl\'ementaire de $H^1$  est not\' e 
$\alpha=[(c/2a)\hat t_1+ \hat h-\sum b_k a_k^{-1} \hat q_k]$ avec la convention que $c/2a=0$ lorsque  $p=2$ ; celui de $H^2$ est $\beta=[\hat \delta]$.

\subsection{Ennonc\'e du th\'eor\` eme principal}

\bigskip 

Le th\'eor\`eme principal est donn\'e sous forme de tableaux.
\vspace{0.5cm}

\centerline{\Large Lorsque  $p=2$}
\vspace{0.5cm}
\centerline{ Cas 1} 
\vspace{0.5cm}
\captionof{table}{\caption{Les g\'en\'erateurs de $H^ 1$ sont $\theta_j=[\hat t_j], 1\le j\le g'; \alpha=[\hat h-\sum_0^ mb_ka_k^{-1}\hat q_k]$. Les g\'en\'erateurs de $H^ 2$ sont $\varphi_j=[\hat \nu_j], 1\le j\le g', \beta=[\hat \delta].$}}

\vspace{0.5cm}

\newcolumntype{M}[1]{>{\raggedright}m{#1}}

\begin{center}
\begin{tabular}{|c||M{4cm}||c||c|}
\hline

&$\theta_i\cup\theta_j$&$\theta_i\cup\alpha$&{$\alpha\cup\alpha$}\tabularnewline
\hline
\hline
$o_i$&$\theta_{2u}\cup\theta_{2u-1}=\beta$  ; $0$ sinon& \multirow{2}{*}{$\varphi_i$}&\multirow{2}{*}{$\frac{c}{2}\beta+\sum_{\varepsilon_j=-1}\varphi_j$}   \tabularnewline    
\cline{1-2}
 $n_i$&$\beta$ si $i=j$; $0$ sinon& &\tabularnewline
\hline

\end{tabular}
\end{center}

\vspace{0.5cm}

\newcolumntype{M}[1]{>{\raggedright}m{#1}}

\begin{center}
\begin{tabular}{|c||M{5.5cm}||c||c||c|}
\hline

&$\theta_i\cup\varphi_j$&$\theta_i\cup\beta$&$\alpha\cup\varphi_j$  &$\alpha\cup\beta$\tabularnewline
\hline
\hline
$o_1$&\multirow{2}{5.5cm}{$\theta_{2u}\cup\varphi_{2u-1}=\theta_{2u-1}\cup\varphi_{2u}=\gamma$; $0$ sinon}&\multirow{6}{*}{$0$}&$0$&\multirow{6}{*}{$\gamma$} \tabularnewline
\cline{1-1} \cline{4-4}
$o_2$&&&$\gamma$& \tabularnewline
\cline{1-1}\cline{2-2}\cline{4-4}
$n_1$&\multirow{4}{5cm}{$\gamma$ si $i=j$; $0$ sinon}&&$0$&\tabularnewline
\cline{1-1}\cline{4-4}
$n_2$& &&$\gamma$&\tabularnewline
\cline{1-1}\cline{4-4}
$n_3$&&&$\gamma$ si $j\neq 1$; $0$ sinon& \tabularnewline
\cline{1-1}\cline{4-4}
$n_4$&&&$\gamma$ si $j\neq 1,2$; $0$ sinon& \tabularnewline
\hline
\end{tabular}
\end{center}

\vfill\eject

\centerline{  Cas 2} 
\vspace{0.5cm}
\captionof{table}{\caption{Les g\'en\'erateurs de $H^ 1$ sont $\theta_j=[\hat t_j], 1\le j\le g'$. Les g\'en\'erateurs de $H^ 2$ sont $\varphi_j=[\hat \nu_j], 1\le j\le g'.$}}

\vspace{0.5cm}

\newcolumntype{M}[1]{>{\raggedright}m{#1}}

\begin{center}
\begin{tabular}{|c||c|}
\hline

&$\theta_i\cup\theta_j$\tabularnewline
\hline
\hline
$o_i,n_i$&0\tabularnewline
\hline

\end{tabular}
\end{center}

\vspace{0.5cm}

\newcolumntype{M}[1]{>{\raggedright}m{#1}}

\begin{center}
\begin{tabular}{|c||M{3cm}|}
\hline

&$\theta_i\cup\varphi_j$\tabularnewline
\hline
\hline
$o_i$&$\theta_{2u}\cup\varphi_{2u-1}=\theta_{2u-1}\cup\varphi_{2u}=\gamma$; $0$ sinon \tabularnewline
\hline
$n_i$&$\gamma$ si $i=j$; $0$ sinon\tabularnewline
\hline
\end{tabular}
\end{center}

\vspace{1cm}

\centerline{ Cas 3} 
\vspace{0.5cm}
\captionof{table}{\caption{Les g\'en\'erateurs de $H^ 1$ sont $\theta_j=[\hat t_j], 1\le j\le g'; \alpha_k=[\hat q_k-\hat q_0], 0< k\le n-1$. Les g\'en\'erateurs de $H^ 2$ sont $\varphi_j=[\hat \nu_j], 1\le j\le g', \beta_k=[\hat \mu_k],1\leq k\le n-1.$}}

\vspace{0.5cm}

\newcolumntype{M}[1]{>{\raggedright}m{#1}}
\begin{center}
\begin{tabular}{|c||c||c||c|}
\hline

&$\theta_i\cup\theta_j$&$\theta_i\cup\alpha_k$&$\alpha_k\cup\alpha_i$\tabularnewline
\hline
\hline
$o_i,n_i$& $0$ &$0$&$\frac{a_0}{2}\sum_{0<\ell \leq n-1}\beta_\ell+\delta_{k,i}\frac{a_k}{2}\beta_k$\tabularnewline
\hline
\end{tabular}
\end{center}

\vspace{0.5cm}

\newcolumntype{M}[1]{>{\raggedright}m{#1}}

\begin{center}
\begin{tabular}{|c||M{3cm}||c||c||M{3.5cm}|}
\hline
&$\theta_i\cup\varphi_j$&$\theta_i\cup\beta_k$ &$\alpha_k\cup\varphi_j$ &$\alpha_k\cup\beta_i$\tabularnewline
\hline
\hline
$o_i$& $\theta_{2u}\cup\varphi_{2u-1}=\theta_{2u-1}\cup\varphi_{2u}=\gamma$; $0$ sinon&\multirow{2}{*}{$0$}&\multirow{2}{*}{$ 0$} &\multirow{3}{*}{$\gamma$ si $i=k$; $0$ sinon} \tabularnewline
\cline{1-2}
$n_i$& $\gamma$ si $i=j$; $0$ sinon&&  & \tabularnewline
\hline
\end{tabular}
\end{center}

\vfill\eject

\centerline{\Large Lorsque $p>2$}
\vspace{0.5cm}

\centerline{ Cas 1} 
\vspace{0.5cm}
\noindent{\Huge  $o_1$}\\
\captionof{table}{\caption{Les g\'en\'erateurs de $H^ 1$ sont $\theta_j=[\hat t_j], 1\le j\le 2g; \alpha=[\hat h-\sum_0^{n-1}b_ka_k^{-1}\hat q_k]$. Les g\'en\'erateurs de $H^ 2$ sont $\varphi_j=[\hat \nu_j], 1\le j\le 2g, \beta=[\hat \delta].$}}

\vspace{0.5cm}

\newcolumntype{M}[1]{>{\raggedright}m{#1}}

\begin{center}
\begin{tabular}{|M{3cm}||c||c|||M{3cm}||c||c||c|}
\hline
$\theta_i\cup\theta_j$&$\theta_j\cup\alpha$&$\alpha\cup\alpha$&
$\theta_i\cup\varphi_j$&$\theta_i\cup\beta$ &$\alpha\cup\varphi_j$ &$\alpha\cup\beta$\tabularnewline
\hline
\hline
$\theta_{2u-1}\cup\theta_{2u}=\beta$ $0$ sinon&$\varphi_j$&$0$&$\theta_{2u}\cup\varphi_{2u-1}=-\gamma;$ $\theta_{2u-1}\cup\varphi_{2u}=\gamma;$ $0$ sinon &$0$&$0$&$\gamma$\tabularnewline
\hline
\end{tabular}
\end{center}

\vspace{0.5cm}

\noindent{\Huge  $o_2$}\\
\captionof{table}{\caption{Les g\'en\'erateurs de $H^ 1$ sont $\theta_j=[\hat t_j], 1\le j\le 2g.$ Les g\'en\'erateurs de $H^ 2$ sont $\varphi_j=[\hat \nu_j+(-1)^j\hat\nu_1], j>2, \beta=[\hat \delta].$}}

\vspace{0.5cm}

\newcolumntype{M}[1]{>{\raggedright}m{#1}}

\begin{center}
\begin{tabular}{|M{3cm}|||M{4cm}|}
\hline
$\theta_i\cup\theta_j$&\tabularnewline
\hline
\hline
$\theta_{2u-1}\cup\theta_{2u}=\beta$; $0$ sinon &Tous les autres cup-produits sont nuls car $H^ 3(M;\Z_p)=0$\tabularnewline
\hline
\end{tabular}
\end{center}

\vspace{1cm}

\noindent{\Huge  $n_1$}\\
\captionof{table}{\caption{Les g\'en\'erateurs de $H^ 1$ sont $\theta_j=[\hat t_j-\hat t_1], 1< j;$ $\alpha=[\hat h-\sum_0^{n-1}b_ka_k^{-1}\hat q_k]$. Les g\'en\'erateurs de $H^ 2$ sont $\varphi_j=[\hat \nu_j-\hat\nu_1], 1< j.$}}

\vspace{0.5cm}

\newcolumntype{M}[1]{>{\raggedright}m{#1}}

\begin{center}
\begin{tabular}{|c||c||c|||M{4cm}|}
\hline
$\theta_i\cup\theta_j$&$\theta_j\cup\alpha$&$\alpha\cup\alpha$&\tabularnewline
\hline
\hline
$0$& $\varphi_j$&$0$& Tous les autres cup-produits sont nuls car $H^ 3(M;\Z_p)=0$\tabularnewline
\hline
\end{tabular}
\end{center}

\vspace{0.5cm}

\noindent{\Huge  $n_2$}\\
\captionof{table}{\caption{Les g\'en\'erateurs de $H^1$ sont $\theta_j=[\hat t_j-\hat t_1], j>1.$  Les g\'en\'erateurs de $H^ 2$ sont $\varphi_j=[\hat \nu_j], j>1.$}}

\vspace{0.5cm}

\newcolumntype{M}[1]{>{\raggedright}m{#1}}

\begin{center}
\begin{tabular}{|c|||M{3cm}|}
\hline
$\theta_i\cup\theta_j$&
$\theta_i\cup\varphi_j$\tabularnewline
\hline
\hline
$0$ &$\theta_j\cup\varphi_j=\gamma;$ $0$ si $i\neq j$\tabularnewline
\hline
\end{tabular}
\end{center}

\vspace{1cm}

\noindent{\Huge  $n_3$}\\
\captionof{table}{\caption{Les g\'en\'erateurs de $H^1$ sont $\theta_j=[\hat t_j-\hat t_1], j>1.$  Les g\'en\'erateurs de $H^ 2$ sont $\varphi_j=[\hat \nu_j], j>2.$}}

\vspace{0.5cm}

\newcolumntype{M}[1]{>{\raggedright}m{#1}}

\begin{center}
\begin{tabular}{|c|||M{4cm}|}
\hline
$\theta_i\cup\theta_j$&
\tabularnewline
\hline
\hline
$0$ &Tous les autres cup-produits sont nuls car $H^ 3(M;\Z_p)=0$\tabularnewline
\hline
\end{tabular}
\end{center}

\vspace{1cm}

\noindent{\Huge  $n_4$}\\
\captionof{table}{\caption{Les g\'en\'erateurs de $H^1$ sont $\theta_j=[\hat t_j-\hat t_1], j>1.$  Les g\'en\'erateurs de $H^ 2$ sont $\varphi_j=[\hat \nu_j], j>3; [\hat\nu_2-\hat\nu_1]=\varphi_3.$}}

\vspace{0.5cm}

\newcolumntype{M}[1]{>{\raggedright}m{#1}}

\begin{center}
\begin{tabular}{|c|||M{4cm}|}
\hline
$\theta_i\cup\theta_j$&
\tabularnewline
\hline
\hline
$0$ &Tous les autres cup-produits sont nuls car $H^ 3(M;\Z_p)=0$\tabularnewline
\hline
\end{tabular}
\end{center}

\vfill\eject

\centerline{  Cas 2} 
\vspace{0.5cm}
\noindent{\Huge  $o_1$}\\
\captionof{table}{\caption{Les g\'en\'erateurs de $H^ 1$ sont $\theta_j=[\hat t_j], 1\le j\le 2g.$  Les g\'en\'erateurs de $H^ 2$ sont $\varphi_j=[\hat \nu_j], 1\le j\le 2g.$}}

\vspace{0.5cm}

\newcolumntype{M}[1]{>{\raggedright}m{#1}}

\begin{center}
\begin{tabular}{|c|||M{3cm}|}
\hline
$\theta_i\cup\theta_j$&
$\theta_i\cup\varphi_j$ \tabularnewline
\hline
\hline
$0$&$\theta_{2u}\cup\varphi_{2u-1}=-\gamma;$ $\theta_{2u-1}\cup\varphi_{2u}=\gamma$; $0$ sinon\tabularnewline
\hline
\end{tabular}
\end{center}

\vspace{0.5cm}

\noindent{\Huge  $o_2$}\\
\captionof{table}{\caption{Les g\'en\'erateurs de $H^ 1$ sont $\theta_j=[\hat t_j], 1\le j\le 2g.$ Les g\'en\'erateurs de $H^ 2$ sont $\varphi_j=[\hat \nu_j+\hat\nu_1], j>2, \beta=[\hat \delta].$}}

\vspace{0.5cm}

\newcolumntype{M}[1]{>{\raggedright}m{#1}}

\begin{center}
\begin{tabular}{|M{3cm}|||M{4cm}|}
\hline
$\theta_i\cup\theta_j$&\tabularnewline
\hline
\hline
$\theta_{2u-1}\cup\theta_{3u}=\beta$; $0$ sinon &Tous les autres cup-produits sont nuls car $H^ 3(M;\Z_p)=0$\tabularnewline
\hline
\end{tabular}
\end{center}

\vspace{1cm}

\noindent{\Huge  $n_1$}\\
\captionof{table}{\caption{Les g\'en\'erateurs de $H^ 1$ sont $\theta_j=[\hat t_j-\hat t_1], 1< j; \alpha=[\frac{c}{2a}\hat t_1+\hat h-\sum_0^{n-1}b_ka_k^{-1}\hat q_k]$. Les g\'en\'erateurs de $H^ 2$ sont $\varphi_j=[\hat \nu_j-\hat\nu_1], j>1.$}}

\vspace{0.5cm}

\vspace{0.5cm}

\newcolumntype{M}[1]{>{\raggedright}m{#1}}

\begin{center}
\begin{tabular}{|c||c||c|||M{4cm}|}
\hline
$\theta_i\cup\theta_j$&$\theta_j\cup\alpha$&$\alpha\cup\alpha$&\tabularnewline
\hline
\hline
$0$& $\varphi_j$&$0$& Tous les autres cup-produits sont nuls car $H^ 3(M;\Z_p)=0$\tabularnewline
\hline
\end{tabular}
\end{center}

\vspace{1cm}

\noindent{\Huge  $n_2$}\\
\captionof{table}{\caption{Les g\'en\'erateurs de $H^1$ sont $\theta_j=[\hat t_j-\hat t_1], j>1.$  Les g\'en\'erateurs de $H^ 2$ sont $\varphi_j=[\hat \nu_j], j>1.$}}

\vspace{0.5cm}

\newcolumntype{M}[1]{>{\raggedright}m{#1}}

\begin{center}
\begin{tabular}{|c|||M{2.7cm}|}
\hline
$\theta_i\cup\theta_j$&
$\theta_i\cup\varphi_j$\tabularnewline
\hline
\hline
$0$ &$\theta_i\cup\varphi_i=\gamma ;$; $0$ si $i\neq j$ \tabularnewline
\hline
\end{tabular}
\end{center}

\vspace{1cm}

\noindent{\Huge  $n_3$}\\
\captionof{table}{\caption{Les g\'en\'erateurs de $H^1$ sont $\theta_j=[\hat t_j-\hat t_1], j>1.$  Les g\'en\'erateurs de 
$H^ 2$ sont $\varphi_j=[\hat \nu_j], j>2.$}}

\vspace{0.5cm}

\newcolumntype{M}[1]{>{\raggedright}m{#1}}

\begin{center}
\begin{tabular}{|c|||M{4cm}|}
\hline
$\theta_i\cup\theta_j$&
\tabularnewline
\hline
\hline
$0$ &Tous les autres cup-produits sont nuls car $H^ 3(M;\Z_p)=0$\tabularnewline
\hline
\end{tabular}
\end{center}

\vspace{1cm}

\noindent{\Huge  $n_4$}\\
\captionof{table}{\caption{Les g\'en\'erateurs de $H^1$ sont $\theta_j=[\hat t_j-\hat t_1], j>1.$  Les g\'en\'erateurs de $H^ 2$ sont $\varphi_j=[\hat \nu_j], j>3; [\hat\nu_2-\hat\nu_1].$}}

\vspace{0.5cm}

\newcolumntype{M}[1]{>{\raggedright}m{#1}}

\begin{center}
\begin{tabular}{|c|||M{4cm}|}
\hline
$\theta_i\cup\theta_j$&
\tabularnewline
\hline
\hline
$0$ &Tous les autres cup-produits sont nuls car $H^ 3(M;\Z_p)=0$\tabularnewline
\hline
\end{tabular}
\end{center}

\vfill\eject

\centerline{  Cas 3} 

\vspace {0.5cm}

\noindent{\Huge  $o_1$}\\
\captionof{table}{\caption{Les g\'en\'erateurs de $H^ 1$ sont $\theta_j=[\hat t_j], 1\le j\le 2g$ et $\alpha_k=[\hat q_k-\hat q_0], 1\leq k\leq n-1$. Les g\'en\'erateurs de $H^ 2$ sont $\varphi_j=[\hat \nu_j], 1\le j\le 2g, \beta_k=[\hat \mu_k], 0\leq k\leq n-1.$}}

\vspace{0.5cm}

\newcolumntype{M}[1]{>{\raggedright}m{#1}}

\begin{center}
\begin{tabular}{|c||c||c|||M{2.4cm}||c||M{1.2cm}||M{1.2cm}|}
\hline
$\theta_i\cup\theta_j$&$\theta_j\cup\alpha_k$&$\alpha_k\cup\alpha_j$&
$\theta_i\cup\varphi_j$&$\theta_i\cup\beta_k$ &$\alpha_k\cup\varphi_j$ &$\alpha_k\cup\beta_j$\tabularnewline
\hline
\hline
 $0$ &$0$&$0$&$\theta_{2u}\cup\varphi_{2u-1}=-\gamma;$ $\theta_{2u-1}\cup\varphi_{2u}=\gamma$; $0$ sinon&$0$&$\alpha_k\cup\varphi_g=-1/2\gamma$ pour tout $k\geq 1$; $0$ sinon&$b_k^{-1}\gamma$, si $j=k$; $0$ sinon\tabularnewline
\hline
\end{tabular}
\end{center}

\vspace{1cm}

\noindent{\Huge  $o_2$}\\
\captionof{table}{\caption{Les g\'en\'erateurs de $H^ 1$ sont $\theta_j=[\hat t_j], 1\le j\le 2g ;\alpha_k=[\hat q_k-\hat q_0], 1\leq k\leq n-1$. Les g\'en\'erateurs de $H^ 2$ sont $\varphi_j=[\hat \nu_j-(-1)^j\hat\nu_1], j>2, \beta_k=[\hat \mu_k], 0\leq k\leq n-1.$}}

\vspace{0.5cm}

\newcolumntype{M}[1]{>{\raggedright}m{#1}}

\begin{center}
\begin{tabular}{|c||c||c|||M{3cm}|}
\hline
$\theta_i\cup\theta_j$&$\theta_j\cup\alpha_k$&$\alpha_k\cup\alpha_j$&
\tabularnewline
\hline
\hline
$0$ &$0$&$0$&Tous les autres cup-produits sont nuls car $H^ 3(M;\Z_p)=0$\tabularnewline
\hline
\end{tabular}
\end{center}

\vfill\eject

\noindent{\Huge  $n_1$}\\
\captionof{table}{\caption{Les g\'en\'erateurs de $H^ 1$ sont $\theta_j=[\hat t_j-\hat t_1], 1< j\leq g; \alpha_k=[\hat q_k-\frac{1}{2}\hat  t_g], 1\leq k\leq n-1 $. Les g\'en\'erateurs de $H^ 2$ sont $\varphi_j=[\hat \nu_j-\hat\nu_1], j>1$ et $\beta_k=[\hat \mu_k], 1\leq k\leq n-1.$}}

\vspace{0.5cm} 

\newcolumntype{M}[1]{>{\raggedright}m{#1}}

\begin{center}
\begin{tabular}{|c||c||c|||M{3cm}|}
\hline
$\theta_i\cup\theta_j$&$\theta_j\cup\alpha_k$&$\alpha_k\cup\alpha_j$&
\tabularnewline
\hline
\hline
$0$ &$0$&$0$&Tous les autres cup-produits sont nuls car $H^ 3(M;\Z_p)=0$\tabularnewline
\hline
\end{tabular}
\end{center}

\vspace{1cm}

\noindent{\Huge  $n_2$}\\
\captionof{table}{\caption{Les g\'en\'erateurs de $H^1$ sont $\theta_j=[\hat t_j-\hat t_g], j>1; \alpha_k=[\hat q_k- \frac{1}{2}\hat t_1], 0\leq k\leq n-1$. Les g\'en\'erateurs de $H^ 2$ sont $\varphi_j=[\hat \nu_j], j>1, \beta_k=[\hat \mu_k], 0\leq k\leq n-1.$}}

\vspace{0.5cm}

\newcolumntype{M}[1]{>{\raggedright}m{#1}}

\begin{center}
\begin{tabular}{|c||c||c|||M{2.4cm}||c||M{1.2cm}||M{1.2cm}|}
\hline
$\theta_i\cup\theta_j$&$\theta_j\cup\alpha_k$&$\alpha_k\cup\alpha_j$&
$\theta_i\cup\varphi_j$&$\theta_i\cup\beta_k$ &$\alpha_k\cup\varphi_j$ &$\alpha_k\cup\beta_j$\tabularnewline
\hline
\hline
 $0$ &$0$&$0$&$ \theta_i\cup\varphi_i=\gamma$; $0$ si $i\neq j$ &$0$&$\alpha_k\cup\varphi_g=-1/2\gamma$ pour tout $k\geq 1$; $0$ sinon&$b_k^{-1}\gamma$, si $j=k$; $0$ sinon\tabularnewline
\hline
\end{tabular}
\end{center}

\vfill\eject

\noindent{\Huge  $n_3$}\\
\captionof{table}{\caption{Les g\'en\'erateurs de $H^1$ sont $\theta_j=[\hat t_j-\hat t_1], j>1; \alpha_k=[\hat q_k- \frac{1}{2}\hat t_1], 0\leq k\leq n-1$. Les g\'en\'erateurs de $H^ 2$ sont $\varphi_j=[\hat \nu_j], j>2, \beta_k=[\hat \mu_k], 0\leq k\leq n-1.$}}

\vspace{0.5cm}

\newcolumntype{M}[1]{>{\raggedright}m{#1}}

\begin{center}
\begin{tabular}{|c||c||c|||M{3cm}|}
\hline
$\theta_i\cup\theta_j$&$\theta_j\cup\alpha_k$&$\alpha_k\cup\alpha_j$&
\tabularnewline
\hline
\hline
$0$ &$0$&$0$&Tous les autres cup-produits sont nuls car $H^ 3(M;\Z_p)=0$\tabularnewline
\hline
\end{tabular}
\end{center}

\vspace{1cm}

\noindent{\Huge  $n_4$}\\
\captionof{table}{\caption{Les g\'en\'erateurs de $H^1$ sont $\theta_j=[\hat t_j-\hat t_1], j>1; \alpha_k=[\hat q_k- \frac{1}{2}\hat t_1], 0\leq k\leq n-1$. Les g\'en\'erateurs de $H^ 2$ sont $\varphi_j=[\hat \nu_j], j>3; [\hat\nu_2-\hat\nu_1], \beta_k=[\hat \mu_k], 0\leq k\leq n-1.$}}

\vspace{0.5cm}

\newcolumntype{M}[1]{>{\raggedright}m{#1}}

\begin{center}
\begin{tabular}{|c||c||c|||M{3cm}|}
\hline
$\theta_i\cup\theta_j$&$\theta_j\cup\alpha_k$&$\alpha_k\cup\alpha_j$&
\tabularnewline
\hline
\hline
$0$ &$0$&$0$&Tous les autres cup-produits sont nuls car $H^ 3(M;\Z_p)=0$\tabularnewline
\hline
\end{tabular}
\end{center}

La suite de cet article est la preuve de ce th\'eor\` eme.

\section{D\'ecomposition simpliciale}\label{section:decomsimp}

\bigskip

Avant de d\'ecrire le d\'ecoupage simpliciale, nous donnons la d\'efinition et les propri\'et\'es du mot $w_{a_k,b_k}(q_k,h)$ qui permet de paver la sph\`ere bordant $\zeta_k$ comme d\'ecrit dans \ref{sub:cell}.

\subsection{ D\' efinition et propri\' et\' es de $w_{\alpha,\beta}$}\label{sub:wab}

Le cas $a_k=1,b_k\leq 0$ sera tr\`es simple, mais dans le cas g\' en\' eral $a_k,b_k>0$, pour pouvoir paver comme \' evoqu\' e 	dans \ref{sub:cell} la sph\`ere bordant $\zeta_k$, le bord $w_{a_k,b_k}(q_k,h)$ de $\mu_k$ doit \^etre un mot tel qu'en effectuant sur ce mot une certaine permutation circulaire et en rempla\c cant un certain $hq_k$ par $q_kh$, on retombe sur le mot de d\' epart. C'est cette propri\' et\' e qui permet le pavage de la sph\`ere par deux exemplaires de $\mu_k$ et un exemplaire de $\rho_k$ pour former le bord de la 3-cellule $\zeta_k$.

\vfill\eject

\begin{defi} Le mot
 $w_{\alpha,\beta}$ (pour $\alpha,\beta$ premiers entre eux) est d\' efini r\' ecursivement par~:\\
$w_{1,0}(a,t)=a$, $w_{0,1}(a,t)=t$, $w_{\alpha+\beta,\beta}(a,t)=w_{\alpha,\beta}(a,at)$, $w_{\alpha,\alpha+\beta}(a,t)=w_{\alpha,\beta}(at,t)$,\\
si bien que le mot $w_{\alpha,\beta}(a,t)$ contient $\alpha$ fois la lettre $a$ et $\beta$ fois la lettre $t$.
\end{defi}

\begin{theo}\label{theo:ab}
Soient $\alpha,\beta,u,v$ entiers tels que $$\alpha u-\beta v=1,\qquad 0<u\leq\beta,\qquad 0\leq v<\alpha.$$ Alors
$$w_{\alpha,\beta}(a,t)=w_{\alpha-v,\beta-u}(a,t)w_{v,u}(a,t)=(w_{v,u}(a,t)t^{-1})at(a^{-1}w_{\alpha-v,\beta-u}(a,t)).$$
\end{theo}

\begin{proof}

 La preuve est par induction sur $\min(\alpha,\beta)$ qui, vues les hypoth\`eses, est supérieur où égal à 1.\\

Si $\alpha=1$ alors $(u,v)=(1,0)$, $w_{\alpha,\beta}(a,t)=at^\beta$, $w_{\alpha-v,\beta-u}(a,t)=at^{\beta-1}$, $w_{v,u}(a,t)=t$, et on a bien $(at^{\beta-1})t=at^\beta$ et $(tt^{-1})at(a^{-1}at^{\beta-1})=1.at.t^{\beta-1}=at^\beta$.

Si $\beta=1$ alors $(u,v)=(1,\alpha-1)$, $w_{\alpha,\beta}(a,t)=a^\alpha t$, $w_{\alpha-v,\beta-u}(a,t)=a$, $w_{v,u}(a,t)=a^{\alpha-1}t$, et on a bien $a(a^{\alpha-1}t)=a^\alpha t$ et $(a^{\alpha-1}tt^{-1})at(a^{-1}a)=a^{\alpha-1}.at.1=a^\alpha t$.

Si $\alpha,\beta$ sont tous deux strictement supérieur à 1, alors (comme ils sont premiers entre eux) ils sont distincts, donc deux cas se pr\' esentent~: $1<\beta<\alpha$ ou $1<\alpha<\beta$.

Si $1<\beta<\alpha$, soient $q$ le quotient et $r$ le reste de la division euclidienne de $\alpha$ par $\beta$. Alors $w_{\alpha,\beta}(a,t)=w_{r+q\beta,\beta}(a,t)=w_{r,\beta}(a,a^qt)$. L'hypoth\`ese d'induction appliqu\' ee au couple $(r,\beta)$ et aux entiers $u',v'$ tels que $ru'-\beta v'=1,0<u'\leq\beta,0\leq v'<r$ donne alors~:\cut
$w_{\alpha,\beta}(a,t)=w_{r,\beta}(a,a^qt)$ est \' egal d'une part \`a $w_{r-v',\beta-u'}(a,a^qt)w_{v',u'}(a,a^qt)$, d'autre part \`a\cut $(w_{v',u'}(a,a^qt)(a^qt)^{-1})(a.a^qt)(a^{-1}w_{r-v',\beta-u'}(a,a^qt))=$\cut$(w_{v',u'}(a,a^qt)t^{-1})at(a^{-1}w_{r-v',\beta-u'}(a,a^qt))$.\cut Par ailleurs, le couple $(u,v)$ associ\' e \`a $(\alpha,\beta)$ se d\' eduit du couple $(u',v')$ associ\' e \`a $(r,\beta)$ par $u=u',v=qu'+v'$. Il ne reste plus qu'\`a remarquer que $w_{r-v',\beta-u'}(a,a^qt)=w_{\alpha-v,\beta-u}(a,t)$ et que $w_{v',u'}(a,a^qt)=w_{v,u}(a,t)$.

Le cas $1<\alpha<\beta$ est absolument similaire.
\end{proof}

\begin{nota}\label{nota:w} Dans la suite, nous appliquerons ce th\' eor\`eme \`a $\alpha=a_k,\beta=b_k$ et noterons $u_k,v_k$ les entiers $u,v$ correspondants. En notant $w_{a_k,b_k}(q_k,h)$ sous la forme $x_{k,1}\ldots x_{k,z_k}$ avec {\rm les $x_{k,i}$ \' egaux \`a $q_k$ (pour $a_k$ d'entre eux dont le premier) ou $h$ (pour $b_k$ d'entre eux dont le dernier) (donc $z_k=a_k+b_k$),} le th\' eor\`eme exprime que pour $w_k=z_k-u_k-v_k+1$, le mot $w_{a_k,b_k}(q_k,h)$ est aussi \' egal \`a $x_{k,w_k}\ldots x_{k,z_k-1}q_khx_{k,2}\ldots x_{k,w_k-1}$, et que de plus, le morceau $x_{k,w_k}\ldots x_{k,z_k}$ de ce mot contient $v_k$ fois $q_k$ et $u_k$ fois $h$.

Dans le cas $b_k\leq 0$ (donc $a_k=1$), nous poserons $u_k=1,v_k=0,w_k=z_k=1+|b_k|$, et $x_{k,1}=q_k$, $x_{k,\ell}=h$ pour $2\leq\ell\leq z_k$.
\end{nota}

Une approche alternative aux mots $w_{\alpha,\beta}$ peut être trouvée dans les références \cite{oz}, \cite{ltz}, \cite{gr}, \cite{hmz}.

\subsection{ D\' ecoupage simplicial}\label{coup}

Transformons ce complexe cellulaire en complexe simplicial en rajoutant~:\\
$\bullet$ un centre et des rayons aux 2-cellules $\delta$ et $\mu_k$, pour remplacer chacune par une juxtaposition de triangles ;\\
$\bullet$ une ``diagonale'' aux $\nu_j,\rho_k$, pour remplacer chacun par deux triangles ;\\
$\bullet$ pour chacune des 3-cellules $\epsilon,\zeta_k$, dont le bord est une sph\`ere pav\' ee par les 2-simplexes d\' ej\`a construits : un centre, des rayons joignant ce centre aux sommets marqu\' es sur la sph\`ere ; des triangles joignant ce centre aux ar\^etes marqu\' ees sur la sph\`ere, de mani\`ere \`a remplacer chaque 3-cellule par une juxtaposition de t\' etra\`edres.

Plus pr\' ecis\' ement, on remplace la d\' ecomposition cellulaire ci-dessus par la d\' ecompo\-sition simpliciale suivante~:

\vspace {0.5cm}

\noindent {\bf 1)} {\bf Le 0-simplexe $\sigma$ et les 1-simplexes $t_j,q_k,h,$}. 

\vspace {0.5cm}

\noindent {\bf 2)} {\bf D\'ecoupage des $\rho_k$, Figure 6} : 

\vspace {0.5cm} 

\noindent  $\bullet$ des 1-simplexes $g_k$ (d'origine et d'extr\' emit\' e $\sigma$) et des 2-simplexes $\rho_{k,1},\rho_{k,2}$, (de faces respectives $(h,g_k,q_k),(q_k,g_k,h)$).

\vspace {0.5cm}

\noindent {\bf 3)} {\bf D\' ecoupage  des $\nu_j$, Figures 7, 8, 9} :

\vspace {0.5cm} 

\noindent  $\bullet$ des 1-simplexes $f_j$ (d'origine et d'extr\' emit\' e $\sigma$) et des 2-simplexes $\nu_{j,2}$ (de faces $(t_j,f_j,h)$) et $\nu_{j,1}$ (de faces $(h,f_j,t_j)$ si $\varepsilon_j=1$,
$(h,t_j,f_j)$ si $\varepsilon_j=-1$).

\vspace {0.5cm} 

\noindent {\bf 4)}{\bf  D\' ecoupage de $\delta$, Figures 10, 11}. Dans le d\' ecoupage de $\delta,\epsilon$ il faudra distinguer les types $o_1$, $o_2$, $n_1$ \`a $n_4$ ;  

\vspace {0.5cm}

\noindent $\bullet$ Un 0-simplexe $a$, des 1-simplexes $e_0,\ldots,e_{*+m}$ (d'origine $a$ et d'extr\' emit\' e $\sigma$) ;

\vspace {0.5cm}

\noindent  $\bullet$ des 2-simplexes $\delta_0,\ldots,\delta_{*+m}$, plus pr\' ecis\' ement~:

\vspace {0.5cm}

 $-$ pour les Types $o_1$, $o_2$~: $\delta_i, $ de faces, respectivement \\
$(t_1,e_1,e_0),(t_2,e_2,e_1),(t_1,e_2,e_3),(t_2,e_3,e_4), \ldots,(q_0,e_{4g+1},e_{4g}),\ldots 
(q_m,e_0,e_{4g+m}) ;$

\vspace {0.5cm}

$-$ pour les Types $n_1$ \`a $n_4$~: \\
$(t_1,e_1,e_0),(t_1,e_2,e_1),\ldots,(q_0,e_{2g+1},e_{2g}),\ldots
(q_m,e_0,e_{2g+m}).$

\vspace {0.5cm}

\noindent{\bf 5)} {\bf D\' ecoupage de chaque $\mu_k$, Figures 12, 13}. Dans le d\' ecoupage de $\mu_k,\zeta_k$ il faudra distinguer le cas particulier $b_k\leq 0$ (et $a_k=1$) du cas g\' en\' eral : 

\vspace {0.5cm}

\noindent $\bullet$ Un 0-simplexe $c_k$, des 1-simplexes $p_{k,1},\ldots,p_{k,z_k}$ (d'origine $c_k$ et d'extr\' e\-mit\' e $\sigma$) ;
 
\vspace {0.5cm}

 \noindent $\bullet$ des 2-simplexes $\mu_{k,1},\ldots,\mu_{k,z_k}$, de faces~: 
\begin{itemize}
\item[-] si $b_k>0$~: $(x_{k,1},p_{k,2},p_{k,1}),\ldots,(x_{k,z_k},p_{k,1},p_{k,z_k})$
\item[-] si $b_k<0$~: $(q_k,p_{k,2},p_{k,1}),(h,p_{k,2},p_{k,3}),\ldots,(h,p_{k,z_k},p_{k,1})$
\item[-] si $b_k=0$~: $(q_k,p_{k,1},p_{k,1})$.
\end{itemize}

\vspace{0.5cm}

\noindent{\bf 6)}{\bf  D\' ecoupage de $\epsilon$, Figures 14, \`a 18}:

\vspace {0.5cm}

\noindent $\bullet$ un 0-simplexe $b$;

\vspace {0.5cm}

\noindent $\bullet$ des 1-simplexes $A^+,A^-$ (d'origine $b$ et d'extr\' emit\' e $a$);

\vspace {0.5cm}

\noindent $\bullet$ des 1-simplexes $S_0^+,\ldots,S_{*+m}^+,S_0^-,\ldots,S_{*+m}^-$ (d'origine $b$ et d'extr\' emit\' e $\sigma$),\\
 et des 2-simplexes $E_0^+,\ldots,E_{*+m}^+,E_0^-,\ldots,E_{*+m}^-$ avec $E_\ell^\pm$ de faces $(e_\ell,S_\ell^\pm,A^\pm)$;
 
 \vspace {0.5cm}
 
\noindent $\bullet$ des 2-simplexes $T_0^+,\ldots,T_{*+m}^+,T_0^-,\ldots T_{*+m}^-$,
plus pr\' ecis\' ement~:

\vspace {0.5cm}

- pour les Types $o_1$, $o_2$~: $T_0^\pm,\ldots,T_{4g+m}^\pm$ de faces\\
$(t_1,S_1^\pm,S_0^\pm),(t_2,S_2^\pm,S_1^\pm),(t_1,S_2^\pm,S_3^\pm),(t_2,S_3^\pm,S_4^\pm),\ldots,(q_m,S_0^\pm,S_{4g+m}^\pm)$;

\vspace {0.5cm}

- pour les Types $n_1$ \`a $n_4$~: $T_0^\pm,\ldots,T_{2g+m}^\pm$ de faces\\
$(t_1,S_1^\pm,S_0^\pm),(t_1,S_2^\pm,S_1^\pm),\ldots,(q_m,S_0^\pm,S_{2g+m}^\pm)$.

\vspace {0.5cm}

\begin{nota} Ici la notation $T_0^\pm$ de faces $(t_1, S_1^\pm,S_0^\pm)$ signifie que les faces de $T_0^+$ sont 
$(t_1, S_1^+,S_0^+)$ et celles de $T_0^-$ sont $(t_1, S_1^-,S_0^-)$ etc.
\end{nota}

\vspace {0.5cm}

\noindent $\bullet$ des 2-simplexes $H_0,\ldots,H_{*+m}$ :\\
les faces de $H_\ell$ \' etant en g\' en\' eral $(h,S_\ell^+,S_\ell^-)$, mais \' etant $(h,S_\ell^-,S_\ell^+)$ si $\ell=2j-1<*$ avec $\varepsilon_j=-1$, i.e. dans les Types suivants~: Type $o_2$, $\ell$ impair $<4g$~; Type $n_2$, $\ell$ impair $<2g$~; Type $n_3$, $\ell$ impair, $3\leq\ell<2g$~; Type $n_4$, $\ell$ impair, $5\leq\ell<2g$;

\vfill\eject

\noindent $\bullet$ des 2-simplexes $F_0,\ldots,F_{*+m}$, de faces:

\vspace {0.5cm}

-  pour le  Type $o_1$~: $(f_1,S_1^+,S_0^-),(f_2,S_2^+,S_1^-),(f_1,S_2^+,S_3^-),$\\ 
$(f_2,S_3^+,S_4^-),\ldots,(g_m,S_0^+,S_{4g+m}^-)$;

\vspace {0.5cm}

- pour le Type $o_2$~: $(f_1,S_1^+,S_0^-),(f_2,S_2^-,S_1^+),(f_1,S_2^-,S_3^+),$\\
$(f_2,S_3^+,S_4^-),\ldots,(g_m,S_0^+,S_{4g+m}^-)$;

\vspace {0.5cm}

- pour le Type $n_1$~: $(f_1,S_1^+,S_0^-),(f_1,S_2^+,S_1^-),$ 
$\ldots,(g_m,S_0^+,S_{2g+m}^-)$;

\vspace {0.5cm}

- pour le Type $n_2$~: $(f_1,S_1^+,S_0^-),(f_1,S_2^-,S_1^+),$ 
$\ldots,(g_m,S_0^+,S_{2g+m}^-)$;

\vspace {0.5cm}

- pour le  Type $n_3$~: $(f_1,S_1^+,S_0^-),(f_1,S_2^+,S_1^-),(f_2,S_3^+,S_2^-),$\\
$(f_2,S_4^-,S_3^+), \ldots,(g_m,S_0^+,S_{2g+m}^-)$;

\vspace {0.5cm}

- pour le  Type $n_4$~: $(f_1,S_1^+,S_0^-),(f_1,S_2^+,S_1^-),(f_2,S_3^+,S_2^-),(f_2,S_4^+,S_3^-),$\\
$(f_3,S_5^+,S_4^-),(f_3,S_6^-,S_5^+),\ldots,(g_m,S_0^+,S_{2g+m}^-)$;

\vspace {0.5cm}

\noindent $\bullet$ des 3-simplexes $D_0^+,\ldots,D_{*+m}^+,D_0^-,\ldots,D_{*+m}^-$, plus pr\' ecis\' ement~:

\vspace {0.5cm}

- pour les Types $o_1$, $o_2$~: $D_0^\pm,\ldots,D_{4g+m}^\pm$ de faces\\ 
$(\delta_0,T_0^\pm,E_1^\pm,E_0^\pm),(\delta_1,T_1^\pm,E_2^\pm,E_1^\pm),(\delta_2,T_2^\pm,E_2^\pm,E_3^\pm),(\delta_3,T_3^\pm,E_3^\pm,E_4^\pm),$\\
$\ldots,(\delta_{4g+m},T_{4g+m}^\pm,E_0^\pm,E_{4g+m}^\pm)$;

\vspace {0.5cm}

- pour les Types $n_1$ \`a $n_4$~: $D_0^\pm,\ldots,D_{2g+m}^\pm$ de faces \\
$(\delta_0,T_0^\pm,E_1^\pm,E_0^\pm),\ldots,(\delta_{2g+m},T_{2g+m}^\pm,E_0^\pm,E_{2g+m}^\pm)$;

\vspace {0.5cm}

\noindent $\bullet$ des 3-simplexes $N_{j,1},N'_{j,1},N_{j,2},N'_{j,2}$ de faces, 

\vspace {0.5cm}

- pour le Type $o_1$~:\\
si $j$ impair, $(\nu_{j,1},H_{2j-1},F_{2j-2},T_{2j-2}^-),(\nu_{j,1},H_{2j},F_{2j},T_{2j}^-),$\\
$(\nu_{j,2},T^ +_{2j-2},F_{2j-2},H_{2j-2})(\nu_{j,2},T^+_{2j},F_{2j},H_{2j+1})$\\
 si $j$ pair,
$(\nu_{j,1},H_{2j-2},F_{2j-3},T_{2j-3}^-),(\nu_{j,1},H_{2j-1},F_{2j-1},T_{2j-1}^-),$\\
$(\nu_{j,2},T_{2j-3}^+,F_{2j-3},H_{2j-3}) (\nu_{j,2},T_{2j-1}^+,F_{2j-1},H_{2j})$;

\vspace {0.5cm}

- pour le Type $o_2$~:\\
si $j$ impair, $(\nu_{j,1},H_{2j-1},T_{2j-2}^-,F_{2j-2}),(\nu_{j,1},H_{2j},T_{2j}^+,F_{2j}),$\\
$(\nu_{j,2},T_{2j-2}^+,F_{2j-2},H_{2j-2}) (\nu_{j,2},T_{2j}^-,F_{2j},H_{2j+1})$,\\ 
si $j$ pair,
$(\nu_{j,1},H_{2j-2},T_{2j-3}^+,F_{2j-3}),(\nu_{j,1},H_{2j-1},T_{2j-1}^-,F_{2j-1}),$\\
$(\nu_{j,2},T_{2j-3}^-,F_{2j-3},H_{2j-3}),(\nu_{j,2},T_{2j-1}^+,F_{2j-1},H_{2j})$;

\vfill\eject

- pour les  Types $n_1$ \`a $n_4$~:\\
 si $\varepsilon_j=1$,
$(\nu_{j,1},H_{2j-1},F_{2j-2},T_{2j-2}^-),(\nu_{j,1},H_{2j},F_{2j-1},T_{2j-1}^-),$\\
$(\nu_{j,2},T_{2j-2}^+,F_{2j-2},H_{2j-2}),(\nu_{j,2},T_{2j-1}^+,F_{2j-1},H_{2j-1})$,\\ 
si $\varepsilon_j=-1$,
$(\nu_{j,1},H_{2j-1},T_{2j-2}^-,F_{2j-2}),(\nu_{j,1},H_{2j},T_{2j-1}^+,F_{2j-1}),$\\
$(\nu_{j,2},T_{2j-2}^+,F_{2j-2},H_{2j-2}),(\nu_{j,2},T_{2j-1}^-,F_{2j-1},H_{2j-1})$;

\vspace {0.5cm}

\noindent $\bullet$ des 3-simplexes $R_{k,1},R_{k,2}$,\\ 
de faces
$(\rho_{k,1},H_{*+k+1},F_{*+k},T_{*+k}^-),(\rho_{k,2},T_{*+k}^+,F_{*+k},H_{*+k})$\\ 
(avec par convention $H_{*+m+1}=H_0$).

\vspace {0.5cm}

\noindent{\bf 7)} {\bf D\' ecoupage de chaque $\zeta_k$, Figures 19, 20, 21} :

\vspace {0.5cm}

\noindent $\bullet$ un 0-simplexe $d_k$;

\vspace {0.5cm}

\noindent $\bullet$ des 1-simplexes $C_k^+,C_k^-$ (d'origine $d_k$ et d'extr\' emit\' e $c_k$) et $S_{k,0},\ldots,S_{k,z_k}$ (d'origine $d_k$ et d'extr\' emit\' e $\sigma$);

\vspace {0.5cm}

\noindent $\bullet$ des 2-simplexes $P_{k,1}^-,\ldots,P_{k,z_k}^-$ de faces $(p_{k,1},S_{k,1},C_k^-),\ldots,(p_{k,z_k},S_{k,z_k},C_k^-)$;

\vspace {0.5cm}

\noindent $\bullet$ des 2-simplexes $P_{k,1}^+,\ldots,P_{k,z_k}^+$ de faces
\begin{itemize} 
\item[-] si $b_k>0$,
$(p_{k,1},S_{k,w_k},C_k^+),\ldots,(p_{k,z_k-w_k+1},S_{k,z_k},C_k^+),$\\
$(p_{k,z_k-w_k+2},S_{k,0},C_k^+),(p_{k,z_k-w_k+3},S_{k,2},C_k^+),\ldots,(p_{k,z_k},S_{k,w_k-1},C_k^+)$
\item[-] si $b_k<0$,
$(p_{k,1},S_{k,0},C_k^+),(p_{k,2},S_{k,3},C_k^+)\ldots,(p_{k,z_k},S_{k,1},C_k^+)$
\item[-] si $b_k=0$, $(p_{k,1},S_{k,0},C_k^+)$;
\end{itemize} 

\vspace {0.5cm}

\noindent $\bullet$ des 2-simplexes $X_{k,1},\ldots,X_{k,z_k}$, de faces
\begin{itemize} 
\item[-] si $b_k>0$, $(x_{k,1},S_{k,2},S_{k,1}),\ldots,(x_{k,z_k},S_{k,1},S_{k,z_k})$,
\item[-] si $b_k<0$, $(q_k,S_{k,2},S_{k,1}),(h,S_{k,2},S_{k,3})\ldots,(h,S_{k,z_k},S_{k,1})$,
\item[-] si $b_k=0$, $(q_k,S_{k,1},S_{k,1})$;
\end{itemize}

\vspace {0.5cm}

\noindent $\bullet$ des 2-simplexes $Q_k,H'_k,G_k$ de faces
\begin{itemize} 
\item[-] si $b_k>0$, $(q_k,S_{k,0},S_{k,z_k}),(h,S_{k,2},S_{k,0}),(g_k,S_{k,2},S_{k,z_k})$,
\item[-] si $b_k<0$, $(q_k,S_{k,3},S_{k,0}),(h,S_{k,1},S_{k,0}),(g_k,S_{k,2},S_{k,0})$,
\item[-] si $b_k=0$, $(q_k,S_{k,0},S_{k,0}),(h,S_{k,1},S_{k,0}),(g_k,S_{k,1},S_{k,0})$;
\end{itemize}

\vspace {0.5cm}

\noindent $\bullet$ des 3-simplexes $M_{k,1}^-,\ldots,M_{k,z_k}^-$, de faces
\begin{itemize} 
\item[-] si $b_k>0$, $(\mu_{k,1},X_{k,1},P_{k,2}^-,P_{k,1}^-),\ldots,(\mu_{k,z_k},X_{k,z_k},P_{k,1}^-,P_{k,z_k}^-)$,
\item[-] si $b_k<0$, $(\mu_{k,1},X_{k,1},P_{k,2}^-,P_{k,1}^-),(\mu_{k,2},X_{k,2},P_{k,2}^-,P_{k,3}^-)\ldots$,\\
$(\mu_{k,z_k},X_{k,z_k},P_{k,z_k}^-,P_{k,1}^-)$,
\item[-] si $b_k=0$, $(\mu_{k,1},X_{k,1},P_{k,1}^-,P_{k,1}^-)$;
\end{itemize}

\vfill\eject

\noindent $\bullet$ des 3-simplexes $M_{k,1}^+,\ldots,M_{k,z_k}^+$ de faces
\begin{itemize} 
\item[-] si $b_k>0$, $(\mu_{k,1},X_{k,w_k},P_{k,2}^+,P_{k,1}^+),\ldots,(\mu_{k,z_k-w_k+1},Q_k,P_{k,z_k-w_k+2}^+,P_{k,z_k-w_k+1}^+)$,\\
$(\mu_{k,z_k-w_k+2},H'_k,P_{k,z_k-w_k+3}^+,P_{k,z_k-w_k+2}^+),\ldots,(\mu_{k,z_k},X_{k,w_k-1},P_{k,1}^+,P_{k,z_k}^+)$,
\item si $b_k<0$, $(\mu_{k,1},Q_k,P_{k,2}^+,P_{k,1}^+),(\mu_{k,2},X_{k,3},P_{k,2}^+,P_{k,3}^+),\ldots,$\\
$(\mu_{k,z_k-1},X_{k,z_k},P_{k,z_k-1}^+,P_{k,z_k}^+),(\mu_{k,z_k},H'_k,P_{k,z_k}^+,P_{k,1}^+)$,
\item[-] si $b_k=0$, $(\mu_{k,1},Q_k,P_{k,1}^+,P_{k,1}^+)$;
\end{itemize}

\vspace {0.5cm}

\noindent $\bullet$ des 3-simplexes $R'_{k,1},R'_{k,2}$ de faces
\begin{itemize} 
\item[-] si $b_k>0$, $(\rho_{k,1},H'_k,G_k,Q_k),(\rho_{k,2},X_{k,1},G_k,X_{k,z_k})$,
\item[-] si $b_k<0$, $(\rho_{k,1},X_{k,2},G_k,Q_k),(\rho_{k,2},X_{k,1},G_k,H'_k)$,
\item[-] si $b_k=0$, $(\rho_{k,1},H'_k,G_k,Q_k),(\rho_{k,2},X_{k,1},G_k,H'_k)$.
\end{itemize}

\section{ Morphisme des cellules vers les simplexes}\label{sub:mor}

\begin{defi}\label{defi:mor} On note $T$ l'application d\'efinie sur les g\'en\'erateurs du complexe des cha\^\i nes cellulaires vers ceux des cha\^\i nes simpliciales  par~:\\
$\bullet$ $T$ envoie les $0$- et $1$-cellules $\sigma$ et $t_j,q_k,h$ sur les $0$- et $1$-simplexes du m\^eme nom,\\
$\bullet$ $T(\rho_k)=\rho_{k,1}-\rho_{k,2}$,\\
$\bullet$ $T(\nu_j)=\nu_{j,1}-\varepsilon_j\nu_{j,2}$,\\
$\bullet$ $T(\mu_k), 1\le \ell \le z_k$ (voir Notation \ref{nota:wab}),
\begin{itemize}
\item[-] si $b_k\geq 0$~: $T(\mu_k)=\sum\mu_{k,\ell}$,
\item[-] si $b_k\leq 0$~: $T(\mu_k)=\mu_{k,1}-\sum_{\ell>1}\mu_{k,\ell}$,
\end{itemize}
$\bullet$ $T(\delta)$
\begin{itemize}
\item[-] pour les Types $o_i$~: $T(\delta)=\sum_{i=0}^{g-1}(\delta_{4i}+\delta_{4i+1}-\delta_{4i+2}-\delta_{4i+3})+\sum_{\ell=4g}^{4g+m}\delta_\ell$,
\item[-] pour les Types $n_i$~: $T(\delta)=\sum\delta_\ell,$
\end{itemize}
$\bullet$ $T(\epsilon)$, en notant $D'_\ell:=D^+_\ell-D^-_\ell$,
\begin{itemize}
\item[-] pour le Type $o_1$~:
$T(\epsilon)=\sum_{k=0}^m (R_{k,1}-R_{k,2})+\sum_{\ell=*}^{*+m}D'_\ell+\sum_{i=0}^{g-1}(D'_{4i}+D'_{4i+1}-D'_{4i+2}-D'_{4i+3})+\sum(N_{j,1}-N_{j,2}-N'_{j,1}+N'_{j,2})$,
\item[-] pour le Type $o_2$~:
$T(\epsilon)=\sum_{k=0}^m (R_{k,1}-R_{k,2})+\sum_{\ell=*}^{*+m}D'_\ell+\sum_{i=0}^{g-1}(D'_{4i}+D'_{4i+1}-D'_{4i+2}-D'_{4i+3})+\sum(-1)^j(N_{j,1}+N_{j,2}+N'_{j,1}+N'_{j,2})$,
\item[-] pour les Types $n_i$~:
$T(\epsilon)=\sum_{k=0}^m (R_{k,1}-R_{k,2})+\sum_{\ell=*}^{*+m}D'_\ell+\sum_{i=0}^{2g-1}D'_j+\sum_{\varepsilon_j=1}(N_{j,1}-N_{j,2}+N'_{j,1}-N'_{j,2})+\sum_{\varepsilon_j=-1}(-N_{j,1}-N_{j,2}+N'_{j,1}+N'_{j,2}),$
\end{itemize}
$\bullet$ $T(\zeta_k)$, en notant $M'_{k,\ell}:=M_{k,\ell}^+- M_{k,\ell}^-$,
\begin{itemize}
\item[-] si $b_k\geq 0$~: $T(\zeta_k)=-R'_{k,1}+R'_{k,2}+\sum M'_{k,\ell}$,
\item[-] si $b_k\leq 0$~: $T(\zeta_k)=-R'_{k,1}+R'_{k,2}+M'_{k,1}-\sum_{\ell>1}M'_{k,\ell}$.
\end{itemize}

\end{defi}

\vfill\eject

\begin{prop} L'application $T$ d\'efinie ci-dessus est un quasi-isomorphisme  du complexe des cha\^\i nes cellulaires vers celui des cha\^\i nes simpliciales.\\
On en d\' eduit un (quasi-iso-) morphisme $T^t$, du complexe des cocha\^\i nes simpliciales vers celui des cocha\^\i nes cellulaires~: $(T^t(f))(s):=f(T(s))$.
\end{prop}

\section{Relev\'e des cocycles cellulaires en cocycles simpliciaux}\label{section:relev}

\bigskip

\subsection{Bord du complexe des cocha\^\i nes simpliciales}

 Pour chaque g\'en\'erateur $\xi=[\hat \xi]$, le but  des deux sections suivantes est de {\sl choisir}  un relev\'e de $T^t$, noté $R(\hat \xi)$, i.e. tel que $T^tR(\hat \xi)=\hat \xi$ avec la propri\'et\'e suppl\'ementaire d'\^etre un cocycle (et pas seulement une  cocha\^\i ne).

Pour all\' eger la pr\' esentation, on n'\' ecrira (en commen\c cant \`a regrouper) que les bords qui seront utiles dans la partie suivante pour expliciter des repr\' esentants des g\' en\' erateurs. 

Le bord est nul sur toutes les 3-cocha\^\i nes, et le bord de la 0-cocha\^\i ne $\hat\sigma+\hat a+\hat b+\sum\hat c_k+\sum\hat d_k$ est nul. 

Posons $U_\ell=\hat F_\ell+\hat\delta_\ell+\hat T_\ell^\pm.$ Ici et dans la suite cette notation est \`a comprendre de la fa\c con suivante : $U_\ell= \hat F_\ell+\hat\delta_\ell+\hat T_\ell^++\hat T_\ell^-.$

\begin{prop} Les  bords des 1-cocha\^\i nes simpliciales sont :\\
\noindent{\bf Types $o_1$, $o_2$}\\
--- $\del\hat t_j=\varepsilon_j\hat\nu_{j,1}+\hat\nu_{j,2}+\hat T_{2j-2}^\pm+\hat T_{2j}^\pm+\hat\delta_{2j-2}+\hat\delta_{2j}$ si $j$ impair, et $\varepsilon_j\hat\nu_{j,1}+\hat\nu_{j,2}+\hat T_{2j-3}^\pm+\hat T_{2j-1}^\pm+\hat\delta_{2j-3}+\hat\delta_{2j-1}$ si $j$ pair;\\
--- $\del\hat f_j=-\varepsilon_j\hat\nu_{j,1}-\hat\nu_{j,2}+$
$\hat F_{2j-2}+\hat F_{2j}$ si $j$ impair, $\hat F_{2j-3}+\hat F_{2j-1}$ si $j$ pair;\\
--- $\del(\hat t_j+\hat f_j)=U_\ell+U_{\ell-2}$ avec si $j$ impair, $\ell=2j$ et si $j$ pair, $\ell=2j-1 ;$

\vspace {0.5cm}

\noindent  {\bf  Type $n_1$ \`a $n_4$}\\
--- $\del\hat t_j=\varepsilon_j\hat\nu_{j,1}+\hat\nu_{j,2}+\hat T_{2j-2}^\pm+\hat T_{2j-1}^\pm+\hat\delta_{2j-2}+\hat\delta_{2j-1};$\\
--- $\del\hat f_j=-\varepsilon_j\hat\nu_{j,1}-\hat\nu_{j,2}+\hat F_{2j-2}+\hat F_{2j-1};$

\vspace {0.5cm}

\noindent  {\bf  Type $o_1$, $o_2$ et $n_1$ \`a $n_4$}\\
--- $\del(\hat h+\sum\hat f_j+\sum\hat q_k)=\sum(\hat H_\ell+\hat F_\ell)+\sum_k(\hat H'_k+\hat G_k+\sum_{x_{k,i}=h}(\hat\mu_{k,i}+\hat X_{k,i}));$\\
--- $\del\hat q_k=\hat\rho_{k,1}+\hat\rho_{k,2}+\hat\delta_{*+k}+\hat T^\pm_{*+k}+\sum_{x_{k,i}=q_k}(\hat\mu_{k,i}+\hat X_{k,i})+\hat Q_k;$\\
--- $\del\hat g_k=-\hat\rho_{k,1}-\hat\rho_{k,2}+\hat F_{*+k}+\hat G_k;$\\
--- $\del\hat p_{k,\ell}=\hat P_{k,\ell}^\pm+\hat\mu_{k,\ell}-\hat\mu_{k,\ell-1}$ si $\ell>1$ et si $b_k>0;$\\
--- $\del\hat p_{k,\ell}=\hat P_{k,\ell}^\pm-\hat\mu_{k,\ell}+\hat\mu_{k,\ell-1}$ si $\ell>2$ et si $b_k<0;$\\
--- $\del\hat p_{k,2}=\hat P_{k,2}^\pm-\hat\mu_{k,2}-\hat\mu_{k,1};$\\
--- $\del(\hat S_0^\pm+\hat e_0)=(\hat T_0^\pm+\hat H_0+\hat F_0)-(\hat T_{*+m}^\pm+\hat H_{*+m}+\hat F_{*+m});$

Le symbole $*$ est d\'efini dans  la Notation \ref{nota:cd} 
\vspace {0.5cm}

\noindent  {\bf - Pour $o_i$ lorsque $1<\ell\leq 4g$ et $\ell=$ 2 ou 3 mod 4,}\\
--- $\del(\hat S_\ell^\pm+\hat S_{\ell-1}^\pm+\hat e_\ell+\hat e_{\ell-1})=-U_\ell-U_{\ell-2};$
 
 \vspace {0.5cm}
 
\noindent {\bf - Pour $n_i$ lorsque $1<\ell\leq 2g$, et pour $o_i$ et $n_i$  lorsque $*<\ell\leq *+m$,} \\
--- $\del(\hat S_\ell^\pm+\hat e_\ell)=U_\ell-U_{\ell-1};$

\vspace {0.5cm}

\noindent {\bf - Pour $o_1$ et $n_1$ (et pour tous les Types si $p=2$) }\\
--- $\del(\hat A^++\sum\hat S_\ell^+)=-\sum(\hat H_\ell+\hat F_\ell);$\\
--- $\del\hat C_k^+=\sum\hat P_{k,\ell}^+.$\\
--- Bord des $\hat S_{k,-}$ :\\
- si $b_k>0$
\begin{itemize}
\item[-- ] $\del\hat S_{k,0}=-\hat Q_k+\hat H'_k-\hat P_{k,z_k-w_k+2}^+$,
\item[-- ] $\del\hat S_{k,1}=-\hat X_{k,z_k}+\hat X_{k,1}-\hat P_{k,1}^-$,
\item[-- ] $\del\hat S_{k,2}=-\hat H_k'-\hat G_k-\hat X_{k,1}+\hat X_{k,2}-\hat P_{k,2}^--\hat P_{k,z_k-w_k+3}^+$,
\item[-- ] $\del\hat S_{k,\ell}=-\hat X_{k,\ell-1}+\hat X_{k,\ell}-\hat P_{k,\ell}^--\hat P_{k,z_k-w_k+\ell+1}^+$ si $2<\ell<w_k$,
\item[-- ] $\del\hat S_{k,\ell}=-\hat X_{k,\ell-1}+\hat X_{k,\ell}-\hat P_{k,\ell}^--\hat P_{k,\ell-w_k+1}^+$ si $w_k\leq\ell<z_k$,
\item[-- ] $\del\hat S_{k,z_k}=\hat Q_k+\hat G_k-\hat X_{k,z_k-1}+\hat X_{k,z_k}-\hat P_{k,z_k}^--\hat P_{k,z_k-w_k+1}^+,$
\end{itemize}
- si $b_k<0$
\begin{itemize}
\item[-- ] $\del\hat S_{k,0}=\hat Q_k+\hat H'_k+\hat G_k-\hat P_{k,1}^+$,
\item[-- ] $\del\hat S_{k,1}=-\hat H'_k+\hat X_{k,1}+\hat X_{k,z_k}-\hat P_{k,1}^--\hat P_{k,z}^+$,
\item[-- ] $\del\hat S_{k,2}=-\hat G_k-\hat X_{k,1}-\hat X_{k,2}-\hat P_{k,2}^-$,
\item[-- ] $\del\hat S_{k,3}=-\hat Q_k+\hat X_{k,2}-\hat X_{k,3}-\hat P_{k,3}^--\hat P_{k,2}^+$,
\item[-- ] $\del\hat S_{k,\ell}=\hat X_{k,\ell-1}-\hat X_{k,\ell}-\hat P_{k,\ell}^--\hat P_{k,\ell-1}^+$ si $\ell>3;$
\end{itemize}
- si $b_k=0$
\begin{itemize}
\item[-- ] $\del\hat S_{k,0}=\hat H'_k+\hat G_k-\hat P_{k,1}^+$,
\item[-- ] $\del\hat S_{k,1}=-\hat H'_k-\hat G_k-\hat P_{k,1}^-.$
\end{itemize}
\end{prop}

Le bord de la 1-cellule $Z_k$ est donn\'e dans le lemme suivant.


\begin{lem}\label{lem:YZV} D\' efinissons $Y_k,Z_k,V_k$ par~:
\begin{itemize}
\item si $b_k>0$,
\begin{itemize}
\item $Y_k=\hat H'_k+\hat G_k+\hat X_{k,1}+\hat\mu_{k,1}-\sum_{2\leq\ell\leq z_k-w_k+2}\hat P_{k,\ell}^+$,
\item $Z_k=\hat q_k+\hat g_k-v_k\hat C_k^++\hat S_{k,0}-\sum_{\ell\geq 2}(\hat S_{k,\ell}+\hat p_{k,\ell})\sharp\{i\geq\ell\tq x_{k,i}=q_k\}$,
\item $V_k=u_k\hat C_k^++\hat S_{k,0}+\sum_{\ell\geq 2}(\hat S_{k,\ell}+\hat p_{k,\ell})\sharp\{i\geq\ell\tq x_{k,i}=h\}$,
\end{itemize}
\item si $b_k\leq 0$ (donc $a_k=u_k=1$ et $v_k=0$)
\begin{itemize}
\item $Y_k=\hat Q_k+\hat G_k+\hat X_{k,1}+\hat\mu_{k,1}$,
\item $Z_k=\hat q_k+\hat g_k$,
\item $V_k=\hat C_k^++\hat S_{k,0}-\sum_{\ell\geq 2}(\hat S_{k,\ell}+\hat p_{k,\ell})(z_k-\ell+1).$
\end{itemize}
\end{itemize}
 Alors $\del Z_k=U_{*+k}+a_kY_k$ et
$\sum_{x_{k,i}=h}(\hat X_{k,i}+\hat\mu_{k,i})=b_kY_k-\hat H'_k-\hat G_k+\del V_k$.
\end{lem}

\begin{proof}

Le cas $b_k\leq 0$ est facile. Dans le cas $b_k>0$, d\' etaillons la preuve.\\
$\del(\hat q_k+\hat g_k)=U_{*+k}+\hat G_k+\hat Q_k+\sum_{x_{k,i}=q_k}(\hat\mu_{k,i}+\hat X_{k,i})$, or\cut
$\hat X_{k,i}+\hat\mu_{k,i}-\hat X_{k,1}-\hat\mu_{k,1}-\hat H'_k-\hat G_k-\del\sum_{2\leq\ell\leq i}(\hat S_{k,\ell}+\hat p_{k,\ell})+\sum_{2\leq\ell\leq i}\hat P_{k,\ell}^+$ est \' egal~:
\\ - si $2\leq i<w_k$, \`a $\sum_{2\leq\ell\leq i}\hat P_{k,z_k-w_k+\ell+1}^+$,
\\ - si $w_k\leq i<z_k$, \`a $\sum_{2\leq\ell<w_k}\hat P_{k,z_k-w_k+\ell+1}^++\sum_{w_k\leq\ell\leq i}\hat P_{k,\ell-w_k+1}^+$, d'o\`u\\
\begin{tiny}
$$\sum_{x_{k,i}=q_k}(\hat X_{k,i}+\hat\mu_{k,i})-a_k(\hat X_{k,1}+\hat\mu_{k,1})-(a_k-1)(\hat H'_k+\hat G_k)
-\del\sum_{\ell\geq 2}(\hat S_{k,\ell}+\hat p_{k,\ell})\sharp\{i\geq\ell\tq x_{k,i}=q_k\}$$
$\begin{array}{rl}
=&-\sum_{\ell\geq 2}\hat P_{k,\ell}^+\sharp\{i\geq\ell\tq x_{k,i}=q_k\}
+\sum_{\ell<z_k-w_k+2}\hat P_{k,\ell}^+\sharp\{i\geq\ell+w_k-1\tq x_{k,i}=q_k\}\\
&+\sum_{\ell>z_k-w_k+2}\hat P_{k,\ell}^+\sharp\{i\geq\ell+w_k-1-z_k\tq x_{k,i}=q_k\}.
\end{array}$
\end{tiny}

Dans cette expression, le coefficient de $\hat P_{k,\ell}^+$ vaut (compte tenu des propri\' et\' es de $w_{\alpha,\beta}$ d\' etaill\' ees dans \ref{sub:wab}\\
 - si $\ell=1$~: $\sharp\{i\geq w_k\tq x_{k,i}=q_k\}=v_k$,\\
 - si $\ell=z_k-w_k+2$~: \\
 $-\sharp\{i\geq z_k-w_k+2\tq x_{k,i}=q_k\}=1-\sharp\{i<w_k\tq x_{k,i}=q_k\}=1-(a_k-v_k)$,
\\ 
- si $2\leq\ell\leq z_k-w_k+2$~: \\
$\sharp\{i\geq \ell+w_k-1\tq x_{k,i}=q_k\}-\sharp\{i\geq\ell\tq x_{k,i}=q_k\}=$\\
$\sharp\{i\tq \ell\leq i<z_k-w_k+1\ et\ x_{k,i}=q_k\}-\sharp\{i\geq\ell\tq x_{k,i}=q_k\}=\\
-\sharp\{i\geq z_k-w_k+1\tq x_{k,i}=q_k\}=-\sharp\{i<w_k\tq x_{k,i}=q_k\}=-(a_k-v_k)$,
\\ 
- si $\ell>z_k-w_k+2$~: \\$\sharp\{i\geq\ell+w_k-1-z_k\tq x_{k,i}=q_k\}-\sharp\{i\geq\ell\tq x_{k,i}=q_k\}=$\\
$\sharp\{i\geq w_k\tq x_{k,i}=q_k\}+\sharp\{i\tq \ell+w_k-1-z_k\leq i<w_k\ et\ x_{k,i}=q_k\}\\
-\sharp\{i\geq\ell\tq x_{k,i}=q_k\}=v_k.$

On en d\' eduit donc
$\sum_{x_{k,i}=q_k}(\hat X_{k,i}+\hat\mu_{k,i})=a_kY_k-\hat Q_k-\hat G_k+\del[v_k\hat C_k^+-\hat S_{k,0}+\sum_{\ell\geq 2}(\hat S_{k,\ell}+\hat p_{k,\ell})\sharp\{i\geq\ell\tq x_{k,i}=q_k\}]$, si bien que $\del Z_k$ est \' egal au r\' esultat annonc\' e.

On a d\' ej\`a vu calcul\' e $\sum_{x_{k,i}=q_k}(\hat X_{k,i}+\hat\mu_{k,i})$. On va en d\' eduire $\sum_{x_{k,i}=h}(\hat X_{k,i}+\hat\mu_{k,i})$ par diff\' erence, en calculant
$\sum_i(\hat X_{k,i}+\hat\mu_{k,i})$. Rappelons que pour $1\leq i<z_k$,\\
$\hat X_{k,i}+\hat\mu_{k,i}-\hat X_{k,1}-\hat\mu_{k,1}-\hat H'_k-\hat G_k-\del\sum_{2\leq\ell\leq i}(\hat S_{k,\ell}+\hat p_{k,\ell})+\sum_{2\leq\ell\leq i}\hat P_{k,\ell}^+$ \' etait \' egal~:
\\- si $2\leq i<w_k$, \`a $\sum_{2\leq\ell\leq i}\hat P_{k,z_k-w_k+\ell+1}^+$
\\ - si $w_k\leq i<z_k$, \`a $\sum_{2\leq\ell<w_k}\hat P_{k,z_k-w_k+\ell+1}^++\sum_{w_k\leq\ell\leq i}\hat P_{k,\ell-w_k+1}^+$.\\
De plus, pour $i=z_k$, on a presque la m\^eme formule que pour $w_k\leq i<z_k$, mais en rempla\c cant $-\hat G_k$ par $+\hat Q_k$.

D'o\`u \\
$\sum_i(\hat X_{k,i}+\hat\mu_{k,i})-z_k(\hat X_{k,1}+\hat\mu_{k,1})-(z_k-1)\hat H'_k-(z_k-2)\hat G_k+\hat Q_k-\del\sum_{\ell\geq 2}(\hat S_{k,\ell}+\hat p_{k,\ell})\sharp\{i\geq\ell\}=-\sum_{\ell\geq 2}\hat P_{k,\ell}^+(z_k-\ell+1)+\sum_{\ell<z_k-w_k+2}\hat P_{k,\ell}^+(z_k-\ell+2-w_k)
+\sum_{\ell>z_k-w_k+2}\hat P_{k,\ell}^+(2z_k-\ell+2-w_k)=$\cut
$(u_k+v_k)\del\hat C_k^+-(a_k+b_k)\sum_{2\leq\ell\leq z_k-w_k+2}\hat P_{k,\ell}^+$, si bien que par diff\' erence,
$\sum_{x_{k,i}=h}(\hat X_{k,i}+\hat\mu_{k,i})$ est \' egal au r\' esultat annonc\' e.

\end{proof} 

\vfill\eject

\begin{prop}{ Bord des 2-cocha\^\i nes simpliciales} :\\
\noindent {\bf Pour tous les Types,} on a d'abord \\
--- $\del(\hat\delta_0+\hat T_0^\pm+\hat F_0)=0$,\\
---  $\del(\hat\mu_{k,1}+\hat X_{k,1}+\hat G_k-\sum_{2\leq\ell\leq z_k-w_k+1}\hat P_{k,\ell}^+)=0.$

\vspace {0.5cm}

\noindent {\bf Pour les Types $o_i$ (avec $\varepsilon=1$ pour $o_1$ et $=-1$ pour $o_2$)}\\
- si $j$ impair\\
--- $\del(\hat\nu_{j,1}+\hat H_{2j-1}+\hat F_{2j-1})=\hat N'_{j,1}+\varepsilon\hat N_{j+1,1}$,\\
--- $\del(\hat\nu_{j,1}+\varepsilon\hat H_{2j-1}+\varepsilon\hat F_{2j-1}+\hat H_{2j})=(1-\varepsilon)\hat N_{j,1}$ et\\
--- $\del(\hat F_{2j-2}+\hat H_{2j-2}+\hat F_{2j-3})=\varepsilon(\hat N_{j,1}+\hat N'_{j-1,1})$;\\
- si $j$ pair \\
--- $\del(\hat\nu_{j,1}+\hat H_{2j-1}+\hat F_{2j-2})=\hat N_{j,1}+\varepsilon\hat N'_{j-1,1}$ et\\
---  $\del(\hat\nu_{j,1}+\varepsilon\hat H_{2j-1}+\varepsilon\hat F_{2j-2}+\hat H_{2j-2})=(1-\varepsilon)\hat N'_{j,1}$,\\
--- $\del(\hat\mu_{k,1}+\hat X_{k,1}+\hat G_k-\sum_{2\leq\ell\leq z_k-w_k+1}\hat P_{k,\ell}^+)=0.$

\vspace {0.5cm}

\noindent {\bf Pour les Types $n_i$,}\\
--- $\del(\hat\mu_{k,1}+\hat X_{k,1}+\hat G_k-\sum_{2\leq\ell\leq z_k-w_k+1}\hat P_{k,\ell}^+)=0,$\\
--- $\del(\hat\nu_{j,1}+\hat H_{2j-1}+\hat F_{2j-1})=(1+\varepsilon_j)\hat N'_{j,1},$\\
--- $\del(\hat\nu_{j,2}+\hat H_{2j-1}+\varepsilon_j\hat F_{2j-2})=(1+\varepsilon_j)\hat N_{j,2},$\\
--- $\del\hat H_{2j-2}=-\hat N'_{j-1,1}-\hat N_{j,2}.$
\end{prop}


\subsection{Relev\'e des 0 et 1-cocycles cellulaires. }

\bigskip

Ayant d\'ecrit les cobords, on est pr\^ ets \` a {\it choisir} des relev\'es qui soient des {\it cocycles} relevant mod $p$ un repr\'esentant des divers g\'en\'erateurs.\\

\noindent{\bf Le 0-cocycle cellulaire } $\hat \sigma$ se rel\` eve en $1=\hat\sigma+\hat a+\hat b+\sum \hat d_k$.

Gr\^ ace au lemme \ref{lem:YZV}, les 1-cocycles sont   relev\'es comme suit :

\begin{defi}\label{defi:R1}{ Relev\' es des 1-cocycles}

 \noindent  {\bf 1)} {\bf  Pour le g\'en\'erateur} $\boldsymbol{\theta_j}$:\\
{\bf  pour $o_i$ et pour tout $p$}, $\theta_j=[\hat t_j]$. On peut relever $\hat t_j$ par\\
 $R\hat t_j=\hat t_j+\hat f_j+\hat e_\ell+\hat S_\ell^\pm+\hat e_{\ell-1}+\hat S_{\ell-1}^\pm;$ 
 
  \vspace{0.5cm}
 
\noindent{\bf pour $n_i$,} il faut distinguer selon $p$ :\\
--- si $p=2$, alors on a $\theta_j=[\hat t_j]$. 
On peut relever $\hat t_j$ par \\
 $R\hat t_j=\hat t_j+\hat f_j+
\hat e_{\ell-1}+\hat S_{\ell-1}^\pm;$\\
--- si $p>2$ alors on a $\theta_j=[\hat t_j-\hat t_1]$ qui se rel\` eve par \\
 $R(\hat t_j-\hat t_1)= \hat t_j+\hat f_j-(\hat t_1+\hat f_1)-2\sum_{u=2}^ {2j-2}(\hat e_u+\hat S_u^\pm)-(\hat e_1+\hat S_1^\pm)-(\hat e_{2j-1}+\hat S_{2j-1}^\pm).$\\
 Dans les deux situations, on a pris $\ell=2j$ si $j$ est impair et $\ell=2j-1$ si $j$ est pair.
 
 \vspace{1cm}
 
\noindent{\bf 2)} {\bf Pour le g\'en\'erateur} $\boldsymbol{\alpha_k}$ :\\
{\bf pour tous les Types et pour tout $p$,} si $0\leq k< n-1$, on commence par relever
$\hat q_k-\hat q_{k-1}$ par\\
 $R(\hat q_k-\hat q_{k-1})=Z_k-Z_{k-1}-\hat S_{2g'+k}^\pm-\hat e_{2g'+k}$, avec $Z_k=\hat q_k+\hat g_k-v_k\hat C_k^++\hat S_{k,0}-\sum_{\ell\geq 2}(\hat S_{k,\ell}+\hat p_{k,\ell})\sharp\{i\geq\ell\tq x_{k,i}=q_k\}$;
 
  \vspace{0.5cm}
  
\noindent {\bf pour les Types $o_i,n_i$, quand $p=2$ et pour les Types $o_i$ quand $p>2$,} le g\'en\'erateur $\alpha_k$ est  $\alpha_k=[\hat q_k-\hat q_0].$ En additionnant, on trouve le relev\'e \\
$R(\hat q_k-\hat q_0)=Z_k-Z_0-\sum_{i=1}^k(\hat S_{2g'+i}^\pm+\hat e_{2g'+i})$ , o\`u $Z_u$ est d\'efini dans le Lemme \ref{lem:YZV} ;
 
 \vspace{1cm}
  
\noindent {\bf  pour tous les Types $n_i$, quand $p>2$}, $\alpha_k=[\hat q_k-{1\over 2}\hat t_g].$ Il faut rajouter \` a la somme pr\'ec\'edente le 
relev\'e de $\alpha_0=[\hat q_0-{1\over 2}\hat t_g]$  qui est choisi \'egal \` a\\
 $R(\hat q_0-{1\over 2}\hat t_g)=Z_0-{1\over 2}(\hat t_g+\hat f_g)-(\hat e_{2g}+\hat S_{2g}^\pm)-{1\over 2}(\hat e_{2g-1}+\hat S_{2g-1}^\pm).$ On obtient donc :\\
 $R(\hat q_k-\frac{1}{2}\hat t_g)= Z_k-\frac{1}{2}(\hat t_g+\hat f_g)-\sum_{i=-1}^k(\hat S_{2g'+i}^\pm+\hat e_{2g'+i}).$
 
 \vspace{1cm}
 
\noindent{\bf 3)} {\bf Pour le g\'en\'erateur} $\boldsymbol{\alpha}=
[{c\over 2a}\hat t_1+\hat h-\sum b_ka_k^{-1}\hat q_k]$. Avec la notation $c_k=b_ka/a_k$, on a  :

 \vspace{0.5cm}
 
\noindent{\bf pour les Types $o_1$ et $n_1$, dans le Cas 1 (on a donc $c=0$), lorsque $p>2$ et pour tous les Types, dans le Cas 1, lorsque $p=2$, } on peut relever $\hat h-\sum b_ka_k^{-1}\hat q_k$ par\\
$R(\hat h-\sum b_ka_k^{-1}\hat q_k)=\hat h+\sum\hat f_j+\sum\hat g_k+\hat A^++\sum\hat S_\ell^+-\sum b_ka_k^{-1}Z_k-\sum V_k-{1\over a}[c_0(\hat e_{*+1}+\hat S_{*+1}^\pm)+(c_0+c_1)(\hat e_{*+2}+\hat S_{*+2}^\pm)+\ldots+(c_0+\ldots+c_{m-1})(\hat e_{*+m}+\hat S_{*+m}^\pm)]$;

 \vspace{0.5cm}
 
\noindent{\bf  pour le Type $n_1$, dans le Cas 2, }  on peut relever ${c\over 2a}\hat t_1+\hat h-\sum b_ka_k^{-1}\hat q_k$ par \\
$R({c\over 2a}\hat t_1+\hat h-\sum b_ka_k^{-1}\hat q_k)=\hat h+\sum\hat f_j+\sum\hat q_k+\hat A^++\sum\hat S_\ell^+-\sum b_ka_k^{-1}Z_k-\sum V_k-{1\over a}[c_0(\hat e_{*+1}+\hat S_{*+1}^\pm)+(c_0+c_1)(\hat e_{*+2}+\hat S_{*+2}^\pm)+\ldots+(c_0+\ldots+c_{m-1})(\hat e_{*+m}+\hat S_{*+m}^\pm)]+
{c\over 2a}[(\hat t_1+\hat f_1)-(\hat S_1^\pm+\hat e_1)-2(\hat S_0^\pm+\hat e_0)]$
avec
\begin{itemize}
\item si $b_k>0$,\\
-- $Z_k=\hat q_k+\hat g_k-v_k\hat C_k^++\hat S_{k,0}-\sum_{\ell\geq 2}(\hat S_{k,\ell}+\hat p_{k,\ell})\sharp\{i\geq\ell\tq x_{k,i}=q_k\}$,\\
-- $V_k=u_k\hat C_k^++\hat S_{k,0}+\sum_{\ell\geq 2}(\hat S_{k,\ell}+\hat p_{k,\ell})\sharp\{i\geq\ell\tq x_{k,i}=h\}$ ;

\item si $b_k\leq 0$,\\
-- $Z_k=\hat q_k+\hat g_k$,\\
-- $V_k=\hat C_k^++\hat S_{k,0}-\sum_{\ell\geq 2}(\hat S_{k,\ell}+\hat p_{k,\ell})(z_k-\ell+1)$.
\end{itemize} 
\end{defi}


\vspace{0.5cm}

{\bf Pour justifier} (lorsque cela n'est pas \'evident) ces choix de relev\'es, on fait le remarques suivantes :

\vspace{0.5cm}

On rappelle la notation $U_\ell=\hat F_\ell+\hat\delta_\ell+\hat T_\ell^\pm$.\\

\noindent  {\bf 1)}   Pour le g\'en\'erateur $\theta_j$:\\
Pour $o_i$, pour tout $ p$, $\partial(\hat t_j+\hat f_j)=U_\ell+U_{\ell-2}$ avec :\\
-- si $j$ impair : $\ell=2j$ et $$\partial(\hat e_\ell+\hat S_\ell^\pm)=-U_\ell-U_{\ell-1},\qquad\partial(\hat e_{\ell-1}+\hat S_{\ell-1}^\pm)=U_{\ell-1}-U_{\ell-2},$$
-- si $j$ pair : $\ell=2j-1$ et $$\partial(\hat e_\ell+\hat S_\ell^\pm)=-U_\ell+U_{\ell-1},\qquad \partial(\hat e_{\ell-1}+\hat S_{\ell-1}^\pm)=-U_{\ell-1}-U_{\ell-2},$$
-- donc quelle que soit la parit\'e de $j$,
$$\partial(\hat t_j+\hat f_j)=-\partial(\hat e_\ell+\hat S_\ell^\pm+\hat e_{\ell-1}+\hat S_{\ell-1}^\pm).$$
Pour tout $ p$, le relev\'e du  cocycle $\theta_j=\hat t_j$ est  :\\
$R\hat t_j=\hat t_j+\hat f_j+\hat e_\ell+\hat S_\ell^\pm+\hat e_{\ell-1}+\hat S_{\ell-1}^\pm.$

Pour $n_i$, pour tout $p$, $\partial(\hat t_j+\hat f_j)=U_{2j-1}+U_{2j-2}$ et $\partial(\hat e_\ell+\hat S_\ell^\pm)=U_\ell-U_{\ell-1}$, donc\\
-- si $p=2$, on peut relever $\theta_j=\hat t_j$ par \\
$R\hat t_j=\hat t_j+\hat f_j+\hat e_{\ell-1}+\hat S_{\ell-1}^\pm,$\\
-- si $p>2$, on peut (pour relever $\theta_j=\hat t_j-\hat t_1$) relever $\hat t_j-\hat t_{j-1}$ par
$(\hat t_j+\hat f_j)-(\hat t_{j-1}+\hat f_{j-1})-(\hat e_{2j-1}+\hat S_{2j-1}^\pm)-(\hat e_{2j-2}+\hat S_{2j-2}^\pm)$.

\vspace{0.5cm}

\noindent{\bf 2)} Pour les g\'en\'erateures $\alpha_k$.

On va se servir des $Z_k$ du lemme, puisqu'ils contiennent $\hat q_k$ plus des cocha\^\i nes simpliciales dont l'image cellulaire (par $T^t$) est nulle. D'apr\`es le lemme, pour tout $k\in[0,n[$ on a, modulo $p$ : $\partial Z_k=U_{2g'+k}$, or pour $2g'<\ell\le 2g'+m$ on a : $\partial(\hat e_\ell+\hat S_\ell^\pm)=U_\ell-U_{\ell-1}$. Ceci permet ($\forall p$, $\forall o_i, n_i$) de relever $\hat q_k-\hat q_{k-1}$ (pour $0<k<n$) par $Z_k-Z_{k-1}-(\hat e_{2g'+k}+\hat S_{2g'+k}^\pm)$.

Pour les Types $n_i$ et $p>2$, il faut relever en plus $\alpha_0=\hat q_0-{1\over 2}\hat t_g$. On se sert donc de $$\partial[Z_0-{1\over 2}(\hat t_g+\hat f_g)]=U_{2g}-{1\over 2}(U_{2g-1}+U_{2g-2})=(U_{2g}-U_{2g-1})+{1\over 2}(U_{2g-1}-U_{2g-2}),$$ ce qui, puisqu'ici $U_\ell-U_{\ell-1}=\partial(\hat e_\ell+\hat S_\ell^\pm)$, permet de relever $\alpha_0$ par : \\
$Z_0-{1\over 2}(\hat t_g+\hat f_g)-(\hat e_{2g}+\hat S_{2g}^\pm)-{1\over 2}(\hat e_{2g-1}+\hat S_{2g-1}^\pm).$

\vspace{0.5cm}

\noindent{\bf3)} Pour le g\'en\'erateur $\alpha$.\\

- Pour ${c\over 2a}\hat t_1+\hat h-\sum b_ka_k^{-1}\hat q_k.$\\
On a $\del(\hat h+\sum\hat f_j+\sum\hat q_k+\hat A^++\sum\hat S_\ell^+)=\sum_k(\hat H'_k+\hat G_k+\sum_{x_{k,i}=h}(\hat\mu_{k,i}+\hat X_{k,i}))$ or d'apr\`es le lemme, pour $k$ fix\' e,

$\hat H'_k+\hat G_k+\sum_{x_{k,i}=h}(\hat\mu_{k,i}+\hat X_{k,i})=b_kY_k+\del V_k=b_ka_k^{-1}(\del Z_k-U_{*+k})+\del V_k.$

Donc on a $\del(\hat h+\sum\hat f_j+\sum\hat q_k+\hat A^++\sum\hat S_\ell^+-\sum b_ka_k^{-1}Z_k-\sum V_k)=-\sum b_ka_k^{-1}U_{*+k}={-1\over a}\sum c_kU_{*+k}$ avec $\sum c_k=c$ et $U_{*+k}-U_{*+k-1}=\del(\hat e_{*+k}+\hat S_{*+k}^\pm)$, d'o\`u

$-\sum c_kU_{*+k}=\del[c_0(\hat e_{*+1}+\hat S_{*+1}^\pm)+(c_0+c_1)(\hat e_{*+2}+\hat S_{*+2}^\pm)+\ldots+(c_0+\ldots+c_{m-1})(\hat e_{*+m}+\hat S_{*+m}^\pm)]-cU_{*+m}$.

Dans les Cas 1 (de $o_1$ et $n_1$), $c=0$ mod $p$ donc on peut relever $\hat h-\sum b_ka_k^{-1}\hat q_k$ par
$\hat h+\sum\hat f_j+\sum\hat q_k+\hat A^++\sum\hat S_\ell^+-\sum b_ka_k^{-1}Z_k-\sum V_k-{1\over a}[c_0(\hat e_{*+1}+\hat S_{*+1}^\pm)+(c_0+c_1)(\hat e_{*+2}+\hat S_{*+2}^\pm)+\ldots+(c_0+\ldots+c_{m-1})(\hat e_{*+m}+\hat S_{*+m}^\pm)]$.

Dans le Cas 2 de $n_1$, comme $\del(\hat S_0^\pm+\hat e_0)=U_0-U_{*+m}$ et $\del(\hat S_1^\pm+\hat e_1)=U_1-U_0$ et $\del(\hat t_1+\hat f_1)=U_0+U_1$, on peut relever ${c\over 2a}\hat t_1+\hat h-\sum b_ka_k^{-1}\hat q_k$ par la m\^eme expression que dans les Cas 1, \`a laquelle on ajoute ${c\over 2a}[(\hat t_1+\hat f_1)-(\hat S_1^\pm+\hat e_1)-2(\hat S_0^\pm+\hat e_0)]$.


\vspace{1cm}

\subsection{Relev\'e des 2-cocycles cellulaires}

\begin{defi}\label{defi:R2}{ Relev\' es des 2-cocycles}

\noindent{\bf 4)}  {\bf Pour le g\'en\'erateur  $\boldsymbol{\beta}=[\hat\delta]$, valable dans tous les Cas et pour tous les Types.} On rel\` eve $\hat\delta$ par ~:\\
 $R\hat\delta=U_0=\hat\delta_0+\hat T_0^\pm+\hat F_0$.

\vspace{0.5cm}
 
\noindent{\bf 5)}  {\bf Relev\' e de $\boldsymbol{\beta_k}=[\hat\mu_k]$, valable dans tous les Cas, que $p$ divise $a_k$ ou pas et que $b_k$ soit positif ou pas.}  On rel\` eve $\hat\mu_k$  par ~: \\
$R\hat\mu_k=\hat\mu_{k,1}+\hat X_{k,1}+\hat G_k-\sum_{2\leq\ell\leq z_k-w_k+1}\hat P_{k,\ell}^+$.

\vspace{0.5cm}

{\bf 6)}  {\bf Pour le g\'en\'erateur $\boldsymbol{\varphi_j}$ pour $p\leq 2$}. \\
{\bf pour $o_i$, si $\varepsilon=1$,} on  rel\`eve $\hat\nu_j$ par :\\
$R\hat\nu_j=\hat\nu_{j,1}+\hat H_{2j-1}+\hat F_{2j-1}+\hat H_{2j}$ si $j$ impair \\
$R\hat\nu_j=\hat\nu_{j,1}+\hat H_{2j-1}+\hat F_{2j-2}+\hat H_{2j-2}$ si $j$ pair.

\vspace{0.5cm}

\noindent{\bf pour $n_i$,  si $\varepsilon_j=-1$, } on rel\`eve $\hat\nu_j$ par\\ 
$R\hat\nu_j=\hat\nu_{j,1}+\hat H_{2j-1}+\hat F_{2j-1}$.

\end{defi}

Il n'est pas nécessaire de définir le relev\'e de $\varphi_j$ pour les autres Types car pour $p>2$, les cup-produits $H^1\tens H^2 \to H^3$ ne seront à calculer que pour les Types $o_1$ et $ n_2$.
 
 \vfill\eject
 
\section{Calcul des cup-produits pour $p=2$}\label{section:p=2}

\bigskip

\subsection{ Formules d'Alexander-Whitney. Méthode des coefficients }\label{sub:coef}

D'apr\`es la formule d'Alexander-Whitney \cite {hatcher}, \cite{jpb}, le cup-produit de deux cocha\^\i nes simpliciales $f$ de degr\' e $p$ et $g$ de degr\' e $q$ est d\' efini sur tout $p+q$-simplexe par $$(f\cup g)(v_0,\ldots,v_{p+q})=f(v_0,\ldots,v_p)g(v_p,\ldots,v_{p+q}).$$ 
On en d\' eduit imm\' ediatement que le g\' en\' erateur $1$ du $H^0$ est neutre pour $\cup$.

Si $\varphi=f\cup g$ avec $f,g$ deux $1$-cocha\^\i nes simpliciales, on obtient donc, pour tout $2$-simplexe $s=(s_0,s_1,s_2)$, de sommets $(v_0,v_1,v_2)$ et de faces $s_0=(v_1,v_2)$, $s_1=(v_0,v_2)$, $s_2=(v_0,v_1)$, $\varphi(s)=f(s_2)g(s_0)$.

Heureusement, si $f,g$ sont des $1$-cocycles, pour conna\^\i tre la classe du $2$-cocycle $\varphi$, il ne sera pas n\' ecessaire de l'\' evaluer sur les (nombreux~!) 2-simplexes. En effet, soit $\varphi'=T^t(\varphi)$ son image dans le complexe cellulaire (voir \ref{sub:mor}), 
$$\varphi'=x\hat\delta+\sum_{j=1}^{g'} y_j\hat\nu_j+\sum z_k\hat\rho_k+\sum_{k=0}^m r_k\hat\mu_k.$$
 Comme $0=\del\varphi'$, les $z_k$ sont nuls. De plus, la classe de cohomologie de $\varphi'$ (donc de $\varphi$) est
\begin{itemize}
\item dans le Cas 1~: $(x+\sum_{k=0}^m r_k)\beta+\sum_{j=1}^{g'}  y_j\varphi_j$,
\item dans le Cas 2~: $\sum_{j=1}^{g'} y_j\varphi_j$,
\item dans le Cas 3~: $\sum_{k=0}^{n-1} r_k\beta_k+\sum_{j=1}^{g'}  y_j\varphi_j$, en posant $\beta_0=-\sum_{k=1}^{ n-1}\beta_k$.
\end{itemize}

\begin{rem}\label{rem:coeff}
Pour calculer la classe de cohomologie $[\varphi]$, il suffira donc d'\' evaluer (mod $2$) $x=\varphi(\sum\delta_\ell)$ (dans le Cas 1), les $r_k=\varphi(\sum\mu_{k,\ell})$ (dans les Cas 1 et 3), et les $y_j=\varphi(\nu_{j,1}-\varepsilon_j\nu_{j,2})$ (dans les trois Cas).

\end{rem}

\subsection{ Les cup-produits, pour $p=2$, $\cup:H^1\tens H^1\to H^2$ }\label{sub:11=2,2}

\begin{theo} Pour $p=2$, \\
$\bullet$  Dans le Cas 1, les seuls
 cup-produits $\cup:H^1\tens H^1\to H^2$ sont
\begin{itemize}
\item $\boldsymbol{ \theta_i\cup\theta_j}$\\ 
- Pour les Types $o_i$, les cup-produits $\theta_i\cup\theta_j$ sont nuls sauf  $\theta_{2i}\cup\theta_{2i-1}=\beta$;\\
- Pour les Types $n_i$,  les cup-produits $\theta_i\cup\theta_j$ sont nuls sauf $\theta_i\cup\theta_i=\beta;$
\item  $\boldsymbol{ \theta_j\cup\alpha}$\\
- Pour tous les Types,  on a 
 $\theta_j\cup\alpha=\varphi_j$.
 \item $\boldsymbol{ \alpha\cup\alpha}$ \\
- Pour les Types $o_1$ et $n_1$, on a 
$\alpha\cup\alpha={c\over 2}\beta$;\\
-  Pour les Types $o_2$ et $n_2$, on a $\alpha\cup\alpha={c\over 2}\beta+\sum_{1\le j}\varphi_j$; \\
- Pour le Type  $n_3$, on a $\alpha\cup\alpha={c\over 2}\beta+\sum_{j>1}\varphi_j$;\\
- Pour le Type $n_4$, on a $\alpha\cup\alpha={c\over 2}\beta+\sum_{j>2}\varphi_j.$
\end{itemize}
$\bullet$  Dans le Cas 2, les  seuls cup-produits $\cup:H^1\tens H^1\to H^2$ sont 
\begin{itemize}
\item $\boldsymbol{ \theta_i\cup\theta_j},$\\
- Pour tous les Types, $\theta_i\cup\theta_j=0.$
\end{itemize}
$\bullet$  Dans le Cas 3, les seuls cup-produits $\cup:H^1\tens H^1\to H^2$ sont
\begin{itemize}
\item $\boldsymbol{ \theta_i\cup\theta_j}$,\\
- Pour tous les Types, on a $\theta_i\cup\theta_j=0;$
\item $\boldsymbol{ \theta_j\cup\alpha_k}.$\\
- Pour tous les Types, on a 
 $\theta_j\cup\alpha_k=0$,
\item $\boldsymbol{ \alpha_k\cup \alpha_i};$\\
- Pour tous les Types, on a
$\alpha_k\cup \alpha_i=\frac{a_0}{2}\sum_{0>\ell\le n-1}\beta_\ell+\delta_{k,\ell}\frac{a_k}{2}\beta_k.$
\end{itemize}
\end{theo}

\begin{proof}

{\bf Calcul de}  $\boldsymbol{ \theta_i\cup\theta_j}$

{\bf  Types $o_i$}
- Dans les trois Cas, 
 $\theta_j=[\hat t_j]$ et le relev\'e de $\theta_j$  est  $R\hat t_j=\hat t_j+\hat f_j+\hat S_\ell^\pm+\hat S_{\ell-1}^\pm+\hat e_\ell+\hat e_{\ell-1}$ avec $\ell=2j$ si $j$ impair et $\ell=2j-1$ si $j$ pair (voir D\'efinition (\ref{defi:R1}). 

\noindent $\bullet$ Les coefficients $r_k$ sont calcul\'es par la formule :
$$\begin{array}{ll}
r_k&=T^t\left(R(\hat t_i)\cup R(\hat t_j)\right)(\mu_k)= \left(R(\hat t_i)\cup R(\hat t_j)\right)T(\mu_k)\\
&=\left(R(\hat t_i)\cup R(\hat t_j)\right)(\sum_\ell\mu_{k,\ell})=\sum_\ell R(\hat t_i)(\mu_{k,\ell})_2R(\hat t_j)(\mu_{k,\ell})_0.\end{array}.$$
Les $r_k$ sont nuls quel que soit le signe de $b_k$ car la face $(\mu_{k,\ell})_2$ du 2-simplexe $\mu_{k,\ell}$ est un $p_{k,\ell}$, annul\' e par $R\hat t_j$ (voir la sous-section D\'ecoupage simplicial (\ref{coup}, 6)).

\noindent  $\bullet$ Les coefficients $y_\ell$ sont calcul\'es par la formule :
$$\begin{array}{rl}
y_\ell&=T^t\left(R(\hat t_i)\cup R(\hat t_j)\right)(\hat\nu_\ell)= \left(R(\hat t_i)\cup R(\hat t_j)\right)T(\nu_\ell)\\
&=\left(R(\hat t_i)\cup R(\hat t_j)\right)(\nu_{\ell,1}-\varepsilon_\ell\nu_{\ell,2})\\
&= R(\hat t_i)(\nu_{\ell,1})_2R(\hat t_j)(\nu_{\ell,1})_0-\varepsilon_\ell R(\hat t_i)(\nu_{\ell,2})_2R(\hat t_j)(\nu_{\ell,2})_0.\end{array}.$$ 
Les $y_\ell$ sont nuls car $R\hat t_j(\nu_{\ell,1})_0=R\hat t_j(h)=0$ et $R\hat t_j(\nu_{\ell,2})_2=R\hat t_j(h)=0$.
(Voir la sous-section D\'ecoupage simplicial (\ref{coup}, 3)).\\
Comme dans les Cas 2 et 3, $\beta$ n'est pas un   g\'en\'erateur, on a la conclusion:\\
{\sl Conclusion : Dans les Cas 2 et 3, pour le Type $o_i$, on a $\theta_i\cup\theta_j=0.$ }

\vspace{0.5cm}

\noindent  $\bullet$  Il reste  \`a calculer $x$ pour le Cas 1. 
On a, voir la sous-section D\'ecoupage simplicial (\ref{coup}, 4 (1)) :\\
\begin{tiny} $x=R\hat t_i(e_0)R\hat t_j(t_1)+R\hat t_i(e_1)R\hat t_j(t_2)+R\hat t_i(e_3)R\hat t_j(t_1)+R\hat t_i(e_4)R\hat t_j(t_2)+\ldots$\end{tiny}\\
 vaut $1$ si et seulement si $i$ impair et $j=i+1$ ou $i$ pair et $j=i-1$ (et vaut $0$ sinon).

{\sl Conclusion : Dans le Cas 1,  pour les Types $o_i$,  les cup-produits $\theta_i\cup\theta_j$ sont nuls sauf $\theta_{2i}\cup\theta_{2i-1}=\beta$.}

\vspace{0.5cm}

{\bf Types $n_i$},  $\theta_j=[\hat t_j]$ et le relev\'e de $\theta_j$  est $R\hat t_j=\hat t_j+\hat f_j+
\hat e_{\ell-1}+\hat S_{\ell-1}^\pm.$\\
 $\bullet$ Comme plus haut,  dans les trois Cas, les $r_k$ et les $y_\ell$ sont tous nuls  et $\beta$ n'est  un g\'en\'erateur que dans le Cas 1.\\
$\bullet$ Calcul de $x$ :
 $x=R\hat t_i(e_0)R\hat t_j(t_1)+R\hat t_i(e_1)R\hat t_j(t_1)+\ldots$ 
  vaut $1$ si et seulement si $i=j$ et vaut $0$ sinon. 

{\sl Conclusion : Pour les Types $n_i$, les cup-produits $\theta_i\cup\theta_j$ sont nuls sauf dans le  Cas 1 et si $i=j$, alors on a $\theta_i\cup\theta_i=\beta.$}

\begin{rem}\label{rem:ri} Dans la suite, (m\^eme pour $p>2$) pour tous les $\theta_j\cup\ldots$, les coefficients $r_k$ sont nuls.
\end{rem}

\vspace{1cm}

{\bf Calcul de $\boldsymbol{ \theta_i\cup\alpha}$}

{\bf  Pour tous les Types, } $\alpha=[\hat h+\sum b_j\hat q_j]$ n'est g\'en\'erateur que dans ce Cas 1.\\
Le relev\'e de $\alpha$ est \\
$R(\hat h+\sum b_j\hat q_j)=\hat h+\sum\hat f_j+\sum\hat q_k+\hat A^++\sum\hat S_\ell^+-\sum b_ka_k^{-1}Z_k-\sum V_k-{1\over a}[c_0(\hat e_{*+1}+\hat S_{*+1}^\pm)+(c_0+c_1)(\hat e_{*+2}+\hat S_{*+2}^\pm)+\ldots+(c_0+\ldots+c_{m-1})(\hat e_{*+m}+\hat S_{*+m}^\pm)].$\\

\noindent  $\bullet$ Nous avons d\'ej\`a que tous les $r_k$ sont nuls.

\noindent $\bullet$ Calcul du coefficient $x$.\\
Dans le d\'ecoupage simpliciale de $\delta$, l'indice des simplexes $\delta_u$ varie de $0$ \`a $4g+m$.\\
 Pour $u\leq 4g-1$, le simplexe $\delta_u$ est $\delta_u=(t_\cdot,e_\cdot,e_\cdot)$ et $R(\hat h-\sum b_ka_k^{-1}\hat q_k)$ (le relev\'e de $\alpha$) appliqu\'e \`a $(\delta_u)_0=t_.$ est nul. \\
Pour $u\geq 4g$, le simplexe $\delta_u$ est $\delta_u=(q_.,e_.,e_.)$ et $R(\hat t_j)$ le relev\'e de $\theta_j$ appliqu\'e \`a $(\delta_u)_2=e_.$ est nul. Le coefficient $x$ est nul.\\

 $\bullet$ Il ne reste plus qu'\`a calculer $y_i$. Or \\
 $$\begin{array}{rl}y_i=&R\hat t_j(\nu_{i,1})_2R(\hat h+\sum b_j\hat q_j)(\nu_{i,1})_0-\varepsilon_jR\hat t_j(\nu_{i,2})_2R(\hat h+\sum b_j\hat q_j)(\nu_{i,2})_0\\
 =&R\hat t_j(\nu_{i,1})_2,
\end{array} 
$$
 car \\
 $R(\hat h+\sum b_j\hat q_j)(\nu_{i,1})_0=R(\hat h+\sum b_j\hat q_j)(h)=1$ et $R\hat t_j(\nu_{i,2})_2=R\hat t_j(h)=0$. Comme de plus, $(\nu_{i,1})_2=t_i$ ou $f_i$, on obtient  $y_i=1$ si et seulement si $i=j$ et $y_i=0$ sinon.

{\sl Conclusion :  $\theta_j\cup\alpha$ n'intervient que dans le Cas 1. Pour tous les Types,  on a 
 $\theta_j\cup\alpha=\varphi_j$. }
 
 \vspace{1cm}

{\bf Calcul de $\boldsymbol{ \theta_j\cup\alpha_k}$}

{\bf Pour tous les Types}, les g\'en\'erateurs $\alpha_k=[\hat q_k-\hat q_0]$ n'interviennent que dans le Cas 3.\\
  $\bullet$ A nouveau, les $r_k$ sont nuls.\\
$\bullet$ Calcul des $y_\ell$\\
$y_\ell=R\hat t_j(\nu_{i,1})_2R(\hat q_k-\hat q_0)(\nu_{i,1})_0-\varepsilon_jR\hat t_j(\nu_{i,2})_2R(\hat q_k-\hat q_0])(\nu_{i,2})_0.$ \\
On a d\'ej\`a vu que $R\hat t_j(\nu_{i,2})_2=0$ et que $R\hat t_j(\nu_{i,1})_2=1$. Mais on a 
$$R(\hat q_k-\hat q_0)= Z_k-Z_0-\sum_{u=1}^k\left(\hat S^\pm_{\star+u}+\hat e_{\star+\ell}\right)$$ par cons\'equent $R(\hat q_k-\hat q_0)(\nu_{i,1})_0=R(\hat q_k-\hat q_0)(h)=0$ car $\hat h$ n'intervient pas dans les $Z_u$. Les coefficients $y_\ell$ sont nuls.

{\sl Conclusion : $\theta_j\cup\alpha_k$ n'intervient que dans le Cas 3.
Pour tous les Types, on a 
 $\theta_j\cup\alpha_k=0$. }

\vspace{1cm}
 
 {\bf Calcul de $\boldsymbol{ \alpha_i\cup\alpha_j}$}
 
{\bf Pour tous les Types}, les g\'en\'erateurs $\alpha_k=[\hat q_k-\hat q_0]$ n'interviennent que dans le Cas 3.
Nous devons calculer les coefficients $r_\ell,0\le \ell\le n-1$ et $y_i$.

$\bullet$ Calcul des coefficients $r_\ell=\sum_w R(\hat q_k-\hat q_0)(\mu_{\ell,w})_2
R(\hat q_j-\hat q_0)(\mu_{\ell,w})_0,$ avec $(\mu_{\ell,w})_0=x_{\ell,w}$ et $(\mu_{\ell,w})_2=p_{\ell,w}$\\ 
1) Dans $R(\hat q_j-\hat q_0)$ interviennent $Z_j$ et $Z_0$. Dans chaque  $Z_u,$
 on voit\\
  $\hat q_u-\sum_{s\ge 1}\hat p_{u,s}\sharp\{t\geq s \mid x_{u,t}=q_u\}$.\\
 -Si $u=j$,   $R(\hat q_j-\hat q_0)(\mu_{\ell,w})_0=R(\hat q_j-\hat q_0)(x_{\ell,w})=0$ sauf si $\ell=j$ et ceci pour tous les $a_j$ indices $w$ tels que $x_{j,w}=q_j$. Dans ces situations $R(\hat q_j-\hat q_0)(\mu_{j,w})_0=1$. \\
 - Si $u=0$, pour n'importe quel indice $j\neq 0$, $R(\hat q_j-\hat q_0)(\mu_{\ell,w})_0=0$ sauf si $\ell=0$ et ceci pour tous les $a_0$ indices $w$ tels que $x_{0,w}=q_0$. Dans ces situations $R(\hat q_j-\hat q_0)(\mu_{0,w})_0=1$. \\
 2) $R(\hat q_k-\hat q_0)(\mu_{\ell,w})_2=R(\hat q_k-\hat q_0)(p_{\ell,w})=0$ sauf si \\
 i) $\ell=j=k$ et pour tous les $w$ tels que $x_{k,w}=q_k$. Dans ces situations on a $R(\hat q_k-\hat q_0)(p_{k,w})=\sharp\{t\geq w \mid x_{k,t}=q_k\}$ ;\\
 ii) $\ell=0$ et pour tous les $w$ tels que $x_{0,w}=q_0$. Dans ces situations, pour n'importe quel indice $k\neq 0$, on a $R(\hat q_k-\hat q_0)(p_{0,w})=\sharp\{t\geq w \mid x_{0,t}=q_0\}$.
 
 Si $k=j,0\le k\le n-1$, nous avons obtenu 
 $$r_k=\sum_{\begin{array}{c}w\ge 1\\ x_{k,w}=q_k\end{array}}
 R(\hat q_k-\hat q_0)(\mu_{k,w})_2=\sum_{w=1}^{w=a_k-1}(a_k-w)= \frac{a_k(a_k-1)}{2}=\frac{a_k}{2}.$$
 Le dernier calcul est fait modulo 2.

$\bullet$ Calcul des coefficients\\
 $y_i=R(\hat q_k-\hat q_0)(\nu_{i,1})_2R(\hat q_j-\hat q_0)(\nu_{i,1})_0-\varepsilon_i R(\hat q_k-\hat q_0)(\nu_{i,2})_2R(\hat q_j-\hat q_0)(\nu_{i,2})_0. $\\
Nous avons $(\nu_{i,1})_2= t_i$ si $\varepsilon_i=1$ et $(\nu_{i,1})_2= f_i$ si $\varepsilon_i=-1$ ;
$(\nu_{i,1})_0=(\nu_{i,2})_2=h;  (\nu_{i,2})_0=t_i.$\\
 Aucun de ces \'el\'ements n'intervient dans $R(\hat q_u-\hat q_0)$. On a donc que pour tout $i, y_i=0.$

 Il reste \`a remarquer que les calculs pr\'ec\'edents sont valables pour tous les Types.

{\sl Conclusion : Les cup-produits $\alpha_i\cup\alpha_j$ n'interviennent que dans le Cas 3. Pour tous les Types, $\alpha_i\cup\alpha_j=\frac{a_0}{2}\beta_0+ \delta_{i,j}\frac{a_j}{2}\beta_j,$ 
 o\` u $\beta_0=\sum_{1\le k\le  n-1}\beta_k$ et $\delta_{i,j}$ est le symb\^ ole de Kronecker.}

\vspace{1cm}
 
 {\bf Calcul de $\boldsymbol{ \alpha\cup\alpha}$}
 
{\bf Pour tous les Types}, le g\'en\'erateur $\alpha$ n'intervient que dans le Cas 1. 
On rappelle que  $r$ est le nombre de $b_k$ pairs et on les a rang\'es entre $0$ et $r-1$.

\noindent  $\bullet$ Calcul du coefficient $x$.\\
Le calcul  se fait par la formule :
\begin{scriptsize}
$$T^t(R(\hat h+\sum b_j\hat q_j)\cup R(\hat h+\sum b_j\hat q_j))(\delta)=(R(\hat h+\sum b_j\hat q_j)\cup R(\hat h+\sum b_j\hat q_j))(T(\delta)).$$
\end{scriptsize}
Si $\ell\le *-1$, on a $R(\hat h+\sum b_j\hat q_j)(\delta_\ell)_0=0$.\\
Mais si $\ell=*+k,0\le k\le m-1$, on a \\
$$\begin{array}{rl}&(R(\hat h+\sum b_j\hat q_j)\cup R(\hat h+\sum b_j\hat q_j))(q_k,e_{*+k+1},e_{*+k})\\
&=R(\hat h+\sum b_j\hat q_j)(q_k)R(\hat h+\sum b_j\hat q_j)(e_{*+k}).
\end{array}
$$
On a d'abord $R(\hat h+\sum b_j\hat q_j)(q_k)=1$ si et seulement si $k\ge r-1$. \\
Ensuite on calcule  $R(\hat h+\sum b_j\hat q_j)(e_{*+k})=(c_0+c_1+\cdots +c_{k-1})$ o\`u $c_u=\frac{b_ua}{a_u}$ qui est non nul seulement si $u<r$. On obtient $(c_0+c_1+\cdots +c_{k-1})=1$ si et seulement si $k-r$ est impair.\\
 - {\bf Si $r$ impair,} $k$ doit \^etre pair et comme on est dans le Cas o\`u $c=0$, on doit avoir $m-r$ pair donc $m$ impair.  Le nombre de $k$ pairs entre $r+1$ pair  et $m-1$ pair, $r+1\le k\le m-1$, est $x=\frac{m-r}{2}$.\\
- {\bf  Si  $r$ est pair,} alors $k$ est impair et $m$ est pair. Le nombre de $k$ impairs entre $r+1$ impair  et $m-1$ impair, $r+1\le k\le m-1$, est $x=\frac{m-r}{2}$.\\
 On en d\'eduit que \\
$(R(\hat h+\sum b_j\hat q_j)\cup R(\hat h+\sum b_j\hat q_j))(\delta_\ell)=1$ si et seulement si $\ell=*+r+2i$ pour un $i>0$ et donc que  $x={m-r\over 2}={1\over 2}\sum_{r}^{m-1}1$.

 Le calcul de $x$ a \'et\'e fait pour les Types $o_i$. Pour les Types $n_i$, $T(\delta)=\sum\delta_\ell.$ Par cons\'equent ce calcul est valable aussi pour les Types $n_i$.

\noindent  $\bullet$ Calcul des coefficients $r_k$\\
On rappelle que $(\mu_{k,.})_0=x_{k,.}$ et $(\mu_{k,.})_2=p_{k,.}$ et 
$r_k=\sum_{\ell=0}^{m-1}R(\hat h+\sum b_j\hat q_j)(p_{k,\ell})R(\hat h+\sum b_j\hat q_j)
(x_{k,\ell})$.\\
{\bf Si } $b_k>0$
\begin{itemize}
\item si $0\le k\leq r-1$,  les $b_k$ sont pairs.  Le terme intervenant dans chaque 
$R(\hat h+\sum b_j\hat q_j)$ est $\sum_t V_t$, plus pr\'ecisement
 $$\sum_t\sum_{u\ge 1} \hat p_{t,u}\sharp\{i\ge u\mid x_{t,i}=h\}.$$
  Calculons $R(\hat h+\sum b_j\hat q_j)(x_{k,\ell})$.\\
- Si $\ell$ est tel que $x_{k,\ell}=h$, ce qui arrive pour $b_k$ d'entre eux, alors $R(\hat h+\sum b_j\hat q_j)(x_{k,\ell})=1$ et $r_k=\sum_{\ell=1}^{b_k-1}R(\hat h+\sum b_j\hat q_j)(p_{k,\ell}).$ On a donc \\
$r_k=\sum_{\ell=1, x_{k,\ell}=h}^{m-1}\sharp\{i\ge \ell\mid x_{\ell,i}=h\}= \sum_{\ell=1}^{b_k-1}(b_k-\ell)=\frac{b_k}{2}.$\\
La derni\`ere \'egalit\'e est calcul\'ee modulo 2.\\
- Si $\ell$ est tel $x_{k,\ell}=q_k$, alors $R(\hat h+\sum b_j\hat q_j)(x_{k,\ell})=0.$ \\
On conclut que, si $b_k>0$ et $0\le k\leq r-1$, on a $r_k=\frac{b_k}{2}.$
\item si $k\ge r$, les $b_k$ sont impairs, donc les $z_k=a_k+b_k$ sont pairs.  La somme  intervenant dans chaque 
$R(\hat h+\sum b_j\hat q_j)$ est $\hat h+ \sum_{u=r}^{m-1}Z_u+\sum V_t. $ Que $\ell$ soit tel que $x_{k,\ell}=h$ ou $x_{k,\ell}=q_k$, on  a $
R(\hat h+\sum b_j\hat q_j)(x_{k,\ell})=1$ d'o\`u
$$\begin{array}{rl}r_k=&\sum_\ell R(\hat h+\sum b_j\hat q_j)(p_{k,\ell})\\
=&\sum_{\ell=1, x_{k,\ell}=q_k}^{m-1}\sharp\{i\ge \ell\mid x_{\ell,i}=q_k\}+\sum_{\ell=1, x_{k,\ell}=h}^{m-1}\sharp\{i\ge \ell\mid x_{\ell,i}=h\}\\
=&\sum_{\ell=1}^{z_k-1}(z_k-\ell)={z_k\over 2}.
\end{array}$$
 La derni\`ere \'egalit\'e est calcul\'ee modulo 2.
\end{itemize}

\noindent  {\bf Si} $b_k\leq 0$ on trouve les m\^emes r\' esultats que pour $b_k>0$, puisque $r_k$ devient (mod $2$)
\begin{itemize}
\item si $k\leq r$, alors $z_k=1+b_k$ est impair. Ce qui change est l'expression de $V_t$ et ce qui nous int\'eresse  est maintenant\\
 $\sum_t\sum_{u\ge 1} \hat p_{t,u}(z_k-u+1).$
On a encore \\
$R(\hat h+\sum b_j\hat q_j)(x_{k,\ell})$ est \'egale \`a 1 si $x_{k,\ell}=h$ et \`a 0 sinon, d'o\`u 
$$r_k=\sum_{\ell=1,x_{k,\ell}=h}R(\hat h+\sum b_j\hat q_j)(p_{k,\ell})=\sum_{\ell=1}^{b_k-1}(z_k-\ell+1)=\frac{b_k}{2}={a_kb_k\over 2}.$$
Les  \'egalit\'es sont calcul\'ees modulo 2.
\item si $k\ge r$, alors $z_k$  est pair. On a en plus $Z_u=\hat q_u+\hat g_u$. Comme pr\'ec\'edemment,  que $\ell$ soit tel que $x_{k,\ell}=h$ ou $x_{k,\ell}=q_k$, on  a $
R(\hat h+\sum b_j\hat q_j)(x_{k,\ell})=1$ d'o\`u
$$r_k=\sum_{\ell=1}^{z_k}(z_k-\ell+1)
={z_k-2\over 2}={-b_k-1\over 2}={b_k+1\over 2}={a_k+b_k\over 2}={1+a_kb_k\over 2}.$$
 Les  \'egalit\'es sont calcul\'ees modulo 2.
\end{itemize}

Maintenant, on rappelle que dans le Cas 1, on a $\beta=[\hat \delta]=[\hat \mu_k]=\beta_k$. Le coefficient de $\beta$  est $(\sum r_k)-x= \sum_{0\le k\le r-1}{a_kb_k\over 2}+ {\sum_{m-1\ge k\ge r}(a_kb_k)\over 2}={1\over 2}\sum a_kb_k={c\over 2}.$

 Remarquons que ces calculs sont valables pour tous les Types.

\noindent $\bullet$ Calcul des coefficients $y_j$\\
On rappelle que
$$\begin{array}{rl} 
y_j&=R(\hat h+\sum b_j\hat q_j)(\nu_{j,1})_2R(\hat h+\sum b_j\hat q_j)(\nu_{j,1})_0\\
&-\varepsilon_jR(\hat h+\sum b_j\hat q_j)(\nu_{j,2})_2R(\hat h+\sum b_j\hat q_j)(\nu_{j,2})_0.
\end{array}.$$
 Pour tous les Types, on a $R(\hat h+\sum b_j\hat q_j)(\nu_{j,2})_0=0.$\\
Comme $R(\hat h+\sum b_j\hat q_j)((\nu_{i,1})_0)=R(\hat h+\sum b_j\hat q_j)(h)=1$, on a
 $y_j=R(\hat h+\sum b_j\hat q_j)(\nu_{j,1})_2$. Alors si $\varepsilon_j=1$, $y_j=R(\hat h+\sum b_j\hat q_j)(t_j)=0$, tandis que si $\varepsilon_j=-1$, $y_j=R(\hat h+\sum b_j\hat q_j)(f_j)=1$.
 
 {\sl Conclusion : $\alpha\cup\alpha$ n'existe que dans le Cas 1. Pour les Types $o_1$ et $n_1$, on a 
$\alpha\cup\alpha={c\over 2}\beta$. Pour les Types $o_2$ et $n_2$, on a $\alpha\cup\alpha={c\over 2}\beta+\sum_{j\ge 1}\varphi_j$. Pour le Type  $n_3$, on a $\alpha\cup\alpha={c\over 2}\beta+\sum_{j>1}\varphi_j$. Pour le Type $n_4$, on a $\alpha\cup\alpha={c\over 2}\beta+\sum_{j>2}\varphi_j$.}

\end{proof}

Pour les Types $o_1$ et $n_2$ ceci correspond bien au r\' esultat de \cite{bhzz1}, \cite{bhzz2}, \cite{bz}, puisque pour $a$ pair  
$\left(\begin{array}{c}a\\ 2\end{array}\right)$ est congru mod $2$ \`a $a/2$.

\subsection{ Les cup-produits, pour $p=2$, $\cup:H^1\tens H^2\to H^3$}\label{sub:12=3,2}

Dans cette section, nous utiliserons le proc\'ed\'e suivant. 

1) Pour $[\xi_1]$ un g\'en\'erateur du $H^1$ et $[\xi_2]$ un g\'en\'erateur du $H^2$, on choisit un repr\'esentant $\xi_1$ et $\xi_2$. Soient $R(\xi_1)$ et $R(\xi_2) $ les cocycles simpliciaux qui sont des sections de $T^t$ donn\'ees dans la section (\ref{section:relev}).

2) D'apr\`es la formule d'Alexander-Whitney, si $f$ est un 1-cocycle simpliciale,  $g$ un 2-cocycle simpliciale et $s$ un 3-simplexe de faces $s_0=(v_1,v_2,v_3)$, $s_1=(v_0,v_2,v_3)$, $s_2=(v_0,v_1,v_3)$ et $s_3=(v_0,v_1,v_2)$ alors $f\cup g(s)=f(v_0,v_1)g(s_0)$, et on trouve $(v_0,v_1)$ en prenant la derni\`ere ar\^ete de $s_2$ ou $s_3$, i.e. $(v_0,v_1)=(s_2)_2=(s_3)_2$. 

3) Quand la combinaison $C$ des 3-simplexes telle que $R(\xi_1)\cup R(\xi_2)= C $ a \'et\'e trouv\'ee, on obtient finalement $[\xi_1]\cup[\xi_2]= [T^t C].$
\begin{theo} Pour $p=2$, les seuls cup-produits $\cup:H^1\tens H^2\to H^3$ sont :\\
$\bullet$  Dans les trois Cas 
\begin{itemize}
 \item $\boldsymbol{\theta_i\cup\varphi_j}$,  \\
-   Pour les Types $o_i$,  les  $\theta_i\cup\varphi_j$ non nuls sont :\\
 si $j$ est impair $\theta_{j+1}\cup\varphi_j=\gamma$, si $j$ est pair $\theta_{j-1}\cup\varphi_j=\gamma.$\\
-  Pour les Types $n_i$,
  on a $\theta_j\cup\varphi_j=\gamma$ et 0 sinon.
  \end{itemize}
  Avec en plus,\\
  $\bullet$ dans le Cas 1
  \begin{itemize}
   \item  $\boldsymbol{\alpha\cup\varphi_j}$, \\
- Pour les Types $o_1$ et $n_1$, $\alpha\cup\varphi_j=0$, \\
- Pour les Types $o_2$ et $n_2$, $\alpha\cup\varphi_j=\gamma$,\\
- Pour le Type $n_3$, $\alpha\cup\varphi_j=\gamma $ si $j\neq 1$ et 0 sinon,\\  
- Pour le Type $n_4$, $\alpha\cup\varphi_j=\gamma $ si $j\neq 1, 2$ et 0 sinon,
\item $\boldsymbol{\theta_i\cup\beta}$, \\
- Pour tous les Types,  $\theta_i\cup\beta=0,$
\item  $\boldsymbol{\alpha\cup\beta }$, \\
- Pour tous les Types,  $\alpha\cup\beta=\gamma,$
\end{itemize}
 $\bullet$ dans le Cas 3
  \begin{itemize}
\item $\boldsymbol{\alpha_k\cup\varphi_j}$, \\
-  Pour tous les Types,  $\alpha_k\cup\varphi_j=0,$
\item $\boldsymbol{\alpha_k\cup\beta_k}$,  \\
- Pour tous les Types, $\alpha_k\cup\beta_k=\gamma$ et 0 sinon,
\item $\boldsymbol{\theta_i\cup\beta_k}$, \\
- Pour tous les Types,  $\theta_i\cup\beta_k=0.$
\end{itemize}
\end{theo}

\begin{proof}

{\bf Calcul de $\boldsymbol{ \theta_i\cup \varphi_j}$}

  Ce cup-produit intervient pour tous les Types, dans les trois Cas.\\
 $\theta_i=[\hat t_i]$ pour $1\le i\le g'$ et $\varphi_j=[\hat\nu_j]$ pour $1\le j\le g'$.

Le relev\'e de $\varphi_j$ contient $\hat\nu_{j,1}$. Les seuls 3-simplexes $s$ dont la  face $s_0$ est $\hat\nu_{j,1}$ sont uniquement $s=N_{j,1}$ ou $s=N'_{j,1}$. Pour les diff\'erents Types, les d\'ecoupages simpliciaux de $ N_{j,1}$  et $N'_{j,1}$ sont diff\'erents.\\
 Le relev\'e de $\theta_i$ est \\
$R(\hat t_i)=\hat t_i+\hat f_i+\hat e_\ell +\hat S^\pm_\ell+\hat e_{\ell-1}+\hat S^\pm_{\ell-1},$\\
- si $i$ est impair ,  $\ell=2i$  ce qui donne ici  $\ell= 2$ modulo 4,\\
- si $i$ est pair $\ell=2i-1$ ce qui donne ici  $\ell= 3$ modulo 4. 

{\bf Type $o_1$, $\varepsilon_j=1$}, on trouve, en comparant les valeurs modulo 4 et les parit\'es :
\begin{itemize}
\item pour $j$ impair, $(N_{j,1})_3=T^-_{2j-2}$ et $(N'_{j,1})_3=T^-_{2j}$,\\
-  comme $2j-2=0$ modulo 4, on a  $((N_{j,1})_3)_2=(T^-_{2j-2})_2=S^-_{2j-2}$,
\begin{itemize}
\item si $i$ est impair, $2i$ et $2i-1$   ne sont pas \'egaux modulo 4 \`a $2j-2$, on a $R(\hat t_i)((N_{j,1})_3)_2=0,$
\item si $i$ est pair, $2i-1$ et $2i-2$   ne sont pas \'egaux modulo 4 \`a $2j-2$, on a  $R(\hat t_i)((N_{j,1})_3)_2=0,$
\end{itemize}
- comme $2j=2$ modulo 4, on a  $((N'_{j,1})_3)_2=(T^-_{2j})_2=S^-_{2j+1}$,
\begin{itemize}
\item si $i$ est impair, comme $2j+1$ et $2i$ ne sont pas de la m\^eme parit\'e, on a $R(\hat t_i)((N'_{j,1})_3)_2=0,$
\item si $i$ est pair et $i=j+1$ alors $R(\hat t_{j+1})((N'_{j,1})_3)_2=1,$
\end{itemize}

\item pour $j$ pair, $(N_{j,1})_3=T^-_{2j-3}$ et $(N'_{j,1})_3=T^-_{2j-1}$,\\
 - comme $2j-3=1$ modulo 4, on a  $((N_{j,1})_3)_2=(T^-_{2j-3})_2=S^-_{2j-3}$,
\begin{itemize}
\item si $i$ est impair et $i=j-1$,  on a $R(\hat t_{j-1})((N_{j,1})_3)_2=1,$
\item si $i$ est pair,  $2i-1$ et $2i-2$   ne sont pas \'egaux modulo 4 \`a $2j-3$,  on a $R(\hat t_i)((N_{j,1})_3)_2=0,$
\end{itemize}
- comme $2j=2$ modulo 4, on a  $((N'_{j,1})_3)_2=(T^-_{2j-1})_2=S^-_{2j}$, 
\begin{itemize}
\item si $i$ est impair,  $2i$ et $2i-1$   ne sont pas \'egaux modulo 4 \`a $2j$, on a $R(\hat t_i)((N'_{j,1})_3)_2=0,$
\item si $i$ est pair, comme $2j$ et $2i-1$ ne sont pas de la m\^eme parit\'e, on a $R(\hat t_i)((N'_{j,1})_3)_2=0.$
\end{itemize}
\end{itemize}

De plus on a $T^t[N_{j,1}]=T^t[N'_{j,1}]=\epsilon$, d'o\`u  la conclusion :\\
{\sl Conclusion : Dans tous les Cas, pour le Type $o_1$, les seuls $\theta_i\cup\varphi_j$ non nuls sont :\\
- si $j$ est impair $\theta_{j+1}\cup\varphi_j=\gamma$,\\
- si $j$ est pair $\theta_{j-1}\cup\varphi_j=\gamma.$}

\vspace{0.5cm}

{\bf Type $o_2$, $\varepsilon_j=-1$}, on trouve
\begin{itemize}
\item  pour $j$ impair, $ (N_{j,1})_2= T^-_{2j-2}$  et $ (N'_{j,1})_2= T^+_{2j}$, \\
- comme $2j-2$ est \'egale \`a 0 modulo 4, on a $((N_{j,1})_3)_2=(T^-_{2j-2})_2=S^-_{2j-2}$,\\
- comme $2j$ est \'egale \`a 2 modulo 4, on a $((N'_{j,1})_3)_2=(T^+_{2j})_2=S^+_{2j+1}$, 
\item pour $j$ pair, $ (N_{j,1})_2= T^+_{2j-3}$  et $ (N'_{j,1})_2= T^-_{2j-1}$, \\
- comme $2j-3$ est \'egale \`a 1 modulo 4, on a $((N_{j,1})_3)_2=(T^+_{2j-3})_2=S^+_{2j-3}$,
- comme $2j-1$ est \'egale \`a 3 modulo 4, on a $((N'_{j,1})_3)_2=(T^-_{2j-1})_2=S^-_{2j}$,
 \end{itemize}

Peu importe qu'on applique $R(\hat t_i)$ \`a un $S^+_{..}$ ou \`a un $S^+_{..}$,  on obtient les m\^emes conditions sur les indices que pour le Type $o_1$.
On a encore $T^t[N_{j,1}]=T^t[N'_{j,1}]=\epsilon$, d'o\`u  la m\^eme conclusion que pour le Type $o_1$ :\\ 
{\sl Conclusion : Dans tous les Cas, pour le Type $o_2$, les seuls $\theta_i\cup\varphi_j$ non nuls sont :\\
- si $j$ est impair $\theta_{j+1}\cup\varphi_j=\gamma$,\\
- si $j$ est pair $\theta_{j-1}\cup\varphi_j=\gamma.$}

\vspace{0.5cm}

{\bf Type $n_1$, $\varepsilon_j=1$}, pour tous $j$, on a $(N_{j,1})_3=T^-_{2j-2}$ et $(N'_{j,1})_3=T^-_{2j-1}$.
Comme plus haut, en comparant les valeurs modulo 4 et les parit\'es, on trouve :
\begin{itemize}
\item pour $j$ impair, \\
- $((N_{j,1})_3)_2=(T^-_{2j-2})_2=S^-_{2j-2}$,
\begin{itemize}
\item pour $i$ impair,  on a $R(\hat t_i)((N_{j,1})_3)_2=0,$
\item pour $i$ pair, on a $R(\hat t_i)((N_{j,1})_3)_2=0,$
\end{itemize}
- $((N'_{j,1})_3)_2=(T^-_{2j-1})_2=S^-_{2j-1}$, 
\begin{itemize}
\item pour $i$ impair et $i=j$, on a $R(\hat t_j)((N'_{j,1})_3)_2=1,$
\item pour $i$ pair, on a $R(\hat t_i)((N'_{j,1})_3)_2=0,$
\end{itemize}
\item pour $j$ pair, \\
- $((N_{j,1})_3)_2=(T^-_{2j-2})_2=S^-_{2j-1}$,
\begin{itemize}
\item pour $i$ impair,  on a $R(\hat t_i)((N_{j,1})_3)_2=0,$
\item pour $i$ pair, on a $R(\hat t_i)((N_{j,1})_3)_2=0,$
\end{itemize}
- $((N'_{j,1})_3)_2=(T^-_{2j-1})_2=S^-_{2j}$, 
\begin{itemize}
\item pour $i$ impair et $i=j$, on a $R(\hat t_j)((N'_{j,1})_3)_2=0,$
\item pour $i$ pair, on a $R(\hat t_i)((N'_{j,1})_3)_2=1.$
\end{itemize}
\end{itemize}
{\sl Conclusion : Dans les trois Cas,  pour les Types $n_1$,  on a $\theta_j\cup\varphi_j=\gamma$ et 0 sinon.}

\vspace{0.5cm}

{\bf Type $n_2$, $\varepsilon_j=-1$} pour tous $j$, on a $((N_{j,1})_2)_2=(T^-_{2j-2})_2=S^-_{2j-2}$ et $((N'_{j,1})_2)=(T^+_{2j-1})_2=S^+_{2j-1}$.
Comme pour le Type $n_1$, en comparant les valeurs modulo 4 et les parit\'es, on trouve  la m\^eme conclusion puisque le signe $\pm$ de l'ar\^ete $S^\pm_{..}$ ne change rien au calcul.\\
{\sl Conclusion : Dans les trois Cas,  pour les Types $n_2$,  $\theta_j\cup\varphi_j=\gamma$ et 0 sinon.}

\vspace{1cm}

{\bf Type $n_3$, $n_4$} en utilisant les r\'esultats pr\'ec\'edents pour $\varepsilon_j=1$ et $\varepsilon_j=-1$,  on a  la conclusion :\\
{\sl Conclusion : Dans les trois Cas,  pour les Types $n_3$ et $n_4$,  on a $\theta_j\cup\varphi_j=\gamma$ et 0 sinon.}

\vspace{1cm}

{\bf Calcul de $\boldsymbol{ \alpha\cup\varphi_j}$}

Ce cup-produit intervient pour tous les Types mais seulement  dans le Cas 1.\\
$\alpha=[\hat h+\sum b_k\hat q_k]$ et $\varphi_j=[\hat \nu_j]$.\\
D'apr\`es l'\'etude pr\'ec\'edente, on sait d\'ej\`a que le relev\'e de $\varphi_j$ contient $\hat\nu_{j,1}$. Les seuls 3-simplexes $s$ dont la  face $s_0$ est $\hat\nu_{j,1}$ sont uniquement $s=N_{j,1}$ ou $s=N'_{j,1}$. On sait aussi que les ar\^etes $(s_3)_2$ ou $(s_2)_2$ sont $S^\pm_u$ avec 
$1\le u\le g'.$ \\
Dans le relev\'e de $\alpha$ n'interviennent que $S^+_u$. L'\'etude pr\'ec\'edente m\`ene \`a la conclusion :\\
{\sl Conclusion : Losque $\varepsilon_j=1$, $\alpha\cup\varphi_j=0$ et lorsque $\varepsilon_j=-1$, $\alpha\cup\varphi_j=\gamma.$}

\vspace{1cm}

{\bf Calcul de $\boldsymbol{ \alpha_k\cup\varphi_j}$ } 

Ce cup-produit intervient  pour tous les Types mais seulement dans le Cas 3.\\
$\alpha_k=[\hat q_k-\hat q_0], 1\leq k\leq n-1$ et $\varphi_j=[\hat \nu_j], 1\leq j\leq 2g.$\\
On sait d\'ej\`a que $R(\hat\nu_j)(s_0)$ est non nul pour $s=N_{j,1}$ et $s=N'_{j,1}$. On sait aussi que les ar\^etes $(s_3)_2$ ou $5s_2)_2$ sont $S^\pm_u$ avec 
$1\le u\le g'.$ Ces ar\^etes n'interviennent pas
dans le relev\'e de $\alpha_k$. On a donc la conclusion: \\
{\sl Conclusion : Pour tous les Types, dans le Cas 3, $\alpha_k\cup\varphi_j=0.$}

\vfill\eject

{\bf Calcul de $\boldsymbol{ \theta_i\cup\beta}$} 

Ce cup-produit intervient pour tous les  Types mais seulement  dans le Cas 1.\\
 On a $\theta_i=[\hat t_i] $ et $\beta=[\hat \delta]$.

On cherche un 3-simplexe $s$ tel que $R(\hat\delta)(s_0)\neq 0$ sachant que $R(\hat\delta)=\hat\delta_0+\hat T^{\pm}_0+\hat F_0.$ Le seul possible 3-simplexe est $s=D^\pm_0$ pour lequel $(s_3)_2=(E^\pm_0)_2= A^\pm$. Comme l'ar\^ete $A^\pm$ n'intervient pas dans le relev\'e de $\theta_i$, on a:\\
{\sl Conclusion : Dans le Cas 1, pour tous les Types, on a $\theta_i\cup\beta=0.$}

\vspace{1cm}

{\bf Calcul de $\boldsymbol{ \alpha\cup\beta}$}  

Ce cup-produit intervient pour tous les  Types mais seulement  dans le Cas 1.

Maintenant, pour tous les Types, dans le relev\'e de $\alpha$, il y a l'ar\^ete $A^+$. Comme $T^t(D^+_0)=\epsilon$, on a la conclusion:\\
{\sl Conclusion : Dans le Cas 1, pour tous les Types, on a toujours $\alpha\cup\beta=\gamma.$}

\vspace{1cm}

{\bf Calcul de $\boldsymbol{ \alpha_i\cup\beta_k}$ } 

Ce cup-produit intervient pour tous les  Types mais seulement dans le Cas 3.\\

 Le relev\'e de $\beta_k$ est $R(\hat\mu_k)=\hat\mu_{k,1}+\hat X_{k,1}+\hat G_k-\sum_{1\leq \ell\leq z_k-w_{k+1}}\hat P_{k,\ell}.$ Le seul 3-simplexe  $s$ tel que $R(\hat\mu_k)(s_0)=1$ est 
$s=M_{k,1}^\pm$. On a   $s_3=P^\pm_{k,1}$ et $(P^\pm_{k,1})_2=C_k^\pm$. Dans le relev\' e de $\alpha_i$ appara\^\i t seulement (via $Z_i$) $C_i^+$.

On v\'erifie que $T^t(M_{k,1}^\pm)=\zeta_k$ et on a $\gamma=[\hat\zeta_k]$.\\
{\sl Conclusion : Dans le Cas 3, pour tous les Types, $\alpha_k\cup\beta_k=\gamma$ et 0 sinon.
}

\vspace{1cm}

{\bf Calcul de $\boldsymbol{ \theta_i\cup\beta_k}$ } 

Ce cup-produit intervient pour tous les  Types mais seulement dans le Cas 3.\\

 Le relev\'e de $\beta_k$ est $R(\hat\mu_k)=\hat\mu_{k,1}+\hat X_{k,1}+\hat G_k-\sum_{1\leq \ell\leq z_k-w_{k+1}}\hat P_{k,\ell}.$ Le seul 3-simplexe  $s$ tel que $R(\hat\mu_k)(s_0)=1$ est 
$s=M_{k,1}^\pm$. On a   $s_3=P^\pm_{k,1}$ et $(P^\pm_{k,1})_2=C_k^\pm$. Le relev\' e de $\theta_i$ appliqu\'e \`a $(s_3)_2$ est nul.\\
{\sl Conclusion : Dans le Cas 3, pour tous les Types, $\theta_i\cup\beta_k=0.$
}
\end{proof}

\section{Calcul des cup-produits pour $p>2$}\label{section:p>2}

\bigskip

Ce calcul  est \`a la fois plus compliqu\' e ($1$ et $-1$ ne sont plus \' egaux, et les g\' en\' erateurs diff\`erent selon les Types) et plus simple (la plupart des cup-produits seront nuls).

\subsection{ Les cup-produits, pour $p>2$, $\cup:H^1\tens H^1\to H^2$}\label{sub:11=2,p}

On utilise la même méthode des coefficients en calculant maintenant modulo $p$.

\begin{theo} Pour $p>2$, \\
$\bullet$  Dans le Cas 1, les seuls  cup-produits $\cup:H^1\tens H^1\to H^2$ sont
\begin{itemize}
\item $\boldsymbol{ \theta_i\cup\theta_j}$\\ 
- Pour les Types $o_i$, les cup-produits $\theta_i\cup\theta_j$ sont nuls sauf\\
  $\theta_{2i-1}\cup\theta_{2i}=\beta,$\\
- Pour les Types $n_i$,  les cup-produits $\theta_i\cup\theta_j$ sont nuls,
\item  $\boldsymbol{ \theta_j\cup\alpha}$\\
- Pour  le Type $o_1$,  on a 
 $\theta_j\cup\alpha=\varphi_j$,\\
  - Pour  le Type $n_1$,  on a 
 $\theta_j\cup\alpha=\varphi_j, j>1$,\\
 \item $\boldsymbol{ \alpha\cup\alpha},$ \\
- Pour les Types $o_1$ et $n_1$, les cup-produits
$\alpha\cup\alpha$ sont nuls,\\
\end{itemize}
$\bullet$  Dans le Cas 2, les seuls cup-produits $\cup:H^1\tens H^1\to H^2$ sont 
\begin{itemize}
\item $\boldsymbol{ \theta_i\cup\theta_j}$,\\
- Pour  le Type $o_2$,  on a 
$\theta_{2i-1}\cup\theta_{2i}=\beta,$\\
- Pour tous les autres Types, $\theta_i\cup\theta_j=0,$
\item $\boldsymbol{\theta_i\cup\theta_j}$\\
- Pour le Type $n_1$, on a
 $\theta_j\cup\alpha=\varphi_j, j>1$,\\

\end{itemize}
$\bullet$  Dans le Cas 3, les seuls  cup-produits $\cup:H^1\tens H^1\to H^2$ sont
\begin{itemize}
\item $\boldsymbol{ \theta_i\cup\theta_j}$,\\
- Pour tous les Types, on a $\theta_i\cup\theta_j=0,$
\item $\boldsymbol{ \theta_j\cup\alpha_k}$,\\
- Pour tous les Types, on a 
 $\theta_j\cup\alpha_k=0$,
\item $\boldsymbol{ \alpha_k\cup \alpha_i}$,\\
- Pour tous les Types , on a
$\alpha_k\cup \alpha_i=0$ sont nuls,\\

\end{itemize}
\end{theo}

\begin{proof}

{\bf Calcul de $\boldsymbol{ \theta_i\cup\theta_j}$} 

{\bf - Dans le Cas 1} \\
{\bf Type $o_1, o_2$},  $\theta_j=[\hat t_j]$,\\
 $\bullet$ Comme les relev\'es sont les m\^emes que   pour $p=2$, les coefficients $r_k$ et $y_\ell$ sont encore nuls.\\
$\bullet$ Il y a seulement \`a calculer $x$ mais en prenant garde aux signes dans $T(\delta)$~:

 $T(\delta)=\sum_{i=0}^{g-1}(\delta_{4i}+\delta_{4i+1}-\delta_{4i+2}-\delta_{4i+3})+\sum_{\ell=4g}^{4g+m}\delta_\ell$ donne alors $x=1$ si $i$ impair et $j=i+1$, mais $x=-1$ si $i$ pair et $j=i-1$. \\
 {\sl Conclusion : Dans le Cas 1, pour les Types $o_1, o_2$, les cup-produits $\theta_i\cup\theta_j$ sont nuls sauf  $\theta_{2k-1}\cup\theta_{2k}=\beta, k>0$.} 
\vspace{0.5cm}

{\bf Type $n_i$, } $\theta_j=[\hat t_j-\hat t_1]$.\\
On a $R(\hat t_j-\hat t_1)= \hat t_j+\hat f_j-(\hat t_1+\hat f_1)-\sum \left(\hat S^\pm_\ell+\hat e_\ell\right).$\\
$\bullet$  Les coefficients $r_i$ et $y_\ell$ sont encore nuls.\\
$\bullet$ On ne calcule pas $x$ car $\beta$ n'est pas un g\'en\'erateur du $H^2$.\\
{\sl Conclusion : Dans le Cas 1, pour les Types $n_i$, les cup-produits $\theta_i\cup\theta_j$ sont tous nuls.} 
 \vspace{0.5cm} 
 
{\bf - Dans le Cas 2} \\
{\bf Type $o_1$, } $\theta_j=[\hat t_j]$.\\
 $\bullet$ Comme les relev\'es sont les m\^emes que   pour $p=2$, les coefficients $r_k$ et $y_\ell$ sont encore nuls.\\
 $\bullet$ On ne calcule pas $x$ car $\beta$ n'est pas un g\'en\'erateur du $H^2$.\\
 {\sl Conclusion : Dans le Cas 2, pour les Types $o_1$, les cup-produits $\theta_i\cup\theta_j$ sont tous nuls.} 
 \vspace{0.5cm} 
 
 {\bf Type $o_2$, } $\theta_j=[\hat t_j]$.\\
 $\bullet$ Comme les relev\'es sont les m\^emes que   pour $p=2$, les coefficients $r_k$ et $y_\ell$ sont encore nuls.\\ 
$\bullet$ Il y a seulement \`a calculer $x$ mais en prenant garde aux signes dans $T(\delta)$, comme dans le Cas 1 pour les Types $o_i$.\\
{\sl Conclusion : Dans le Cas 2, pour les Types $ o_2$, les cup-produits $\theta_i\cup\theta_j$ sont nuls sauf  $\theta_{2k-1}\cup\theta_{2k}=\beta, k>0$.} 
 \vspace{0.5cm}

{\bf Type $n_i$.} La situation est la m\^eme que dans le Cas 1.\\
{\sl Conclusion : Dans le Cas 2, pour les Types $n_i$, les cup-produits $\theta_i\cup\theta_j$ sont tous nuls.} 

\vspace{0.5cm}

{\bf - Dans le Cas 3} \\
{\bf Type $o_i,n_i$}\\
 $\bullet$ Comme les relev\'es sont les m\^emes que   pour $p=2$, les coefficients $r_k$ et $y_\ell$ sont encore nuls.\\
 $\bullet$ On ne calcule pas $x$ car $\beta$ n'est pas un g\'en\'erateur du $H^2$.\\
 {\sl Conclusion : Dans le Cas 3, pour les Types $o_i,n_i$, les cup-produits $\theta_i\cup\theta_j$ sont tous nuls.}

\vfill\eject 

{\bf Calcul de $\boldsymbol{ \theta_j\cup\alpha}$} 

\vspace{0.5cm}  

Ce cup-produit n'intervient que dans le Cas 1 pour le Type $o_1$  et dans les Cas 1 et 2 pour le Type $n_1$ .\\

{\bf - Dans le Cas 1}\\
{\bf Type $o_1$,} $\theta_j=[\hat t_j]$ et $\alpha= [\hat h-\sum b_ka_k^{-1}\hat q_k]$.\\
$\bullet$ Calcul du coefficient $x$.\\
Dans $T(\delta)$, l'indice des simplexes $\delta_u$ varie de $0$ \`a $4g+m$. Pour $u\leq 4g-1$, le simplexe $\delta_u$ est $\delta_u=(t_\cdot,e_\cdot,e_\cdot)$ et $R(\hat h-\sum b_ka_k^{-1}\hat q_k)$ le relev\'e de $\alpha$ appliqu\'e \`a $(\delta_u)_0=t_.$ est nul. Lorsque $u\geq 4g$, le simplexe $\delta_u$ est $\delta_u=(q_.,e_.,e_.)$ et $R(\hat t_j)$ le relev\'e de $\theta_j$ appliqu\'e \`a $(\delta_u)_2=e_.$ est nul. Le coefficient $x$ est nul.\\
 $\bullet$ Calcul du coefficient $y_j$.\\
 Pour le Type $o_1$, on a $T(\nu_j)=\nu_{j,1}+\nu_{j,2}$, puisque tous les $\varepsilon_j=1$ et $\nu_{j,1}=(h,f_j,t_j),\nu_{j,2}=(t_j,f_j,h)$. On voit que seulement
 $$R(\hat h-\sum b_ka_k^{-1}\hat q_k)((\nu_{j,1})_0)=1, R(\hat t_j)((\nu_{j,1})_2)=1.$$
 {\sl Conclusion.  Les cup-produits $\theta_j\cup\alpha$ sont nuls sauf pour le Type $o_1$, dans le Cas 1 et alors $\theta_j\cup\alpha=\varphi_j$.}
 
 \vspace{0.5cm}
 
{\bf Type $n_1$}, $\theta_j=[\hat t_j-\hat t_1]$ et $\alpha= [\hat h-\sum b_ka_k^{-1}\hat q_k]$.\\
  Dans le relev\'e de $\theta_j$ intervient $t_1+f_1$ mais ce terme ne donne pas de contribution car   il suffit de calculer $y_j$ pour $j>1$ puisque $\varphi_1$ n'est pas un g\'en\'erateur de $H^2$. Alors la conclusion est la m\^eme que pour le Type $o_1$ avec une restriction sur l'indice $j$.\\
 {\sl Conclusion.  Dans les Cas 1,  pour le Type $n_1$, $\theta_j\cup\alpha=\varphi_j, j>1$.}
 
  \vspace{0.5cm}
  
{\bf -Dans le Cas 2}\\
{\bf Type $n_1$,} $\theta_j=[\hat t_j-\hat t_1]$ et $\alpha= [\hat h-\sum b_ka_k^{-1}\hat q_k]$.\\
La conclusion est la m\^eme que dans le Cas 1.\\
 {\sl Conclusion.  Dans les Cas 2,  pour le Type $n_1$, $\theta_j\cup\alpha=\varphi_j, j>1$.}
 
\vspace{1cm}
 
 {\bf Calcul de $\boldsymbol{ \alpha\cup\alpha}$} 
 
 Ce cup-produit n'intervient que dans le Cas 1 pour le Type $o_1$  et dans les Cas 1 et 2 pour le Type $n_1$.
 
 {\bf - Dans le Cas 1}\\
{\bf Type $o_1$, } $\alpha= [\hat h-\sum b_ka_k^{-1}\hat q_k]$.\\
Le relev\'e de $\alpha$ est $R(\hat h-\sum b_ka_k^{-1}\hat q_k)=\hat h+\sum\hat f_j+\sum\hat g_k+\hat A^++\sum\hat S_\ell^+-\sum b_ka_k^{-1}Z_k-\sum V_k-{1\over a}[c_0(\hat e_{*+1}+\hat S_{*+1}^\pm)+(c_0+c_1)(\hat e_{*+2}+\hat S_{*+2}^\pm)+\ldots+(c_0+\ldots+c_{m-1})(\hat e_{*+m}+\hat S_{*+m}^\pm)]$.\\
$\bullet$ Calcul du coefficient $x$.\\
Comme $R(\hat h-\sum b_ka_k^{-1}\hat q_k)(t_j)=0$, on a 
$$x=\sum R(\hat h-\sum b_ka_k^{-1}\hat q_k)(e_{*+k})R(\hat h-\sum b_ka_k^{-1}\hat q_k)(q_k).$$
$\bullet$ Calcul des coefficients $r_k$.\\
\begin{itemize}
\item Si $b_k>0$, 
$$\begin{array}{rl}
r_k=&\sum R(\hat h-\sum b_ka_k^{-1}\hat q_k)(p_{k,\ell})R(\hat h-\sum b_ka_k^{-1}\hat q_k)(x_{k,\ell})\\
=&\sum_{x_{k,\ell}=q_k}(-b_k/a_k)[(b_k/a_k)\sharp\{i\geq\ell\tq x_{k,i}=q_k\}-\sharp\{i\geq\ell\tq x_{k,i}=h\}]\\
&+\sum_{x_{k,\ell}=h}[(b_k/a_k)\sharp\{i\geq\ell\tq x_{k,i}=q_k\}-\sharp\{i\geq\ell\tq x_{k,i}=h\}\\
=&{-1\over a_k^2}\sum_{i\geq\ell}s_is_\ell,
\end{array}$$
 avec $s_i=b_k$ si $x_{k,i}=q_k$ et $s_i=-a_k$ si $x_{k,i}=h$. Cet {\sl entier} $\sum_{i\geq\ell}s_is_\ell$ est \' egal \`a :\
  $\sum_{i\geq\ell}s_is_\ell={1\over 2}[\sum(s_\ell^2)-(\sum s_\ell)^2]={1\over 2}(a_kb_k^2+b_ka_k^2-0)={a_kb_k(a_k+b_k)\over 2}$,\\
   donc $r_k=-{b_k(a_k+b_k)\over 2a_k}.$
\item Si $b_k\leq 0$,\\
 $$\begin{array}{rl} r_k=&R(\hat h-\sum b_ka_k^{-1}\hat q_k)(q_k)\alpha(p_{k,1})-R(\hat h-\sum b_ka_k^{-1}\hat q_k(p_{k,3}+\\
 &\ldots+p_{k,1})R(\hat h-\sum b_ka_k^{-1}\hat q_k)(h)\\
 =&-\sum_{\ell>2}\alpha(p_{k,\ell})=-\sum_{\ell>2}(z_k-\ell+1)=-{z_k-2)(z_k-1)\over 2}, 
 \end{array}$$
 or $z_k=1-b_k$ et $a_k=1$ d'o\`u \\ 
 $r_k=-{b_k(a_k+b_k)\over 2a_k}$, comme dans le cas $b_k>0$.
\end{itemize}
$\bullet$ Calcul des coefficients $y_\ell$.\\
On a \\
\begin{tiny} $y_j=R(\hat h-\sum b_ka_k^{-1}\hat q_k)(t_j)R(\hat h-\sum b_ka_k^{-1}\hat q_k)(h)-R(\hat h-\sum b_ka_k^{-1}\hat q_k)(h)R(\hat h-\sum b_ka_k^{-1}\hat q_k)(t_j)=0$.\end{tiny}\\
Il reste donc $\alpha\cup\alpha=[x\hat\delta+\sum r_k\hat\mu_k]=N[\hat\delta]$, avec
$$N=x-\sum{r_k\over a_k}=\sum_{k=0}^m{b_k\over a_k}(\sum_{i<k}{b_i\over a_i})+\sum{b_k(a_k+b_k)\over 2a_k^2}={c(a+c)\over 2}a^{-2}. $$ Comme $c$ est divisible par $p>2$, $N$ est congru mod $p$ \`a $0$.\\
{\sl Conclusion : Dans le Cas 1, pour le Type $o_1$, $\alpha\cup\alpha=0$.}

\vspace {0.5cm}

{\bf Type $n_1$,} $\alpha=  [\frac{c}{2}\hat t_1]+\hat h-\sum b_ka_k^{-1}\hat q_k]$.\\
Le relev\'e de $\alpha$ est $R(\hat h-\sum b_ka_k^{-1}\hat q_k)=\hat h+\sum\hat f_j+\sum\hat g_k+\hat A^++\sum\hat S_\ell^+-\sum b_ka_k^{-1}Z_k-\sum V_k-{1\over a}[c_0(\hat e_{*+1}+\hat S_{*+1}^\pm)+(c_0+c_1)(\hat e_{*+2}+\hat S_{*+2}^\pm)+\ldots+(c_0+\ldots+c_{m-1})(\hat e_{*+m}+\hat S_{*+m}^\pm)]+ \frac{c}{2}a\left((\hat t_1+\hat f_1)-(\hat S^\pm_1+\hat e_1)-2(\hat S^\pm_0+\hat e_0)\right)$.\\
Comme $p$ divise $c$, le nouveau dernier facteur n'intervient pas dans le calcul des coefficients. De plus pour ce Type, on a $[\hat\delta]=0$.\\
{\sl Conclusion : Dans le Cas 1, pour le Type $n_1$, $\alpha\cup\alpha=0$.}
\vspace {0.5cm}

 {\bf - Dans le Cas 2}\\
 Pour le Type $n_1$, la situation est la m\^eme que dans le Cas 1.\\
{\sl Conclusion : Dans le Cas 2, pour le Type $n_1$, $\alpha\cup\alpha=0$.}

\vspace{1cm}

{\bf Calcul de $\boldsymbol{ \theta_j\cup\alpha_k}$}

- Ce cup-produit n'existe que dans le Cas 3. \\
Dans ce Cas, pour tous les Types, les coefficients $r_k$ et $y_\ell$ sont nuls pour les m\^emes raisons que plus haut et $\beta$ n'est pas un g\'en\'erateur du $H^2$ donc il n'y a pas de coefficient $x$.\\
 {\sl Conclusion : Pour les Types, les cup-produits $\theta_j\cup\alpha_k$ n'existent que dans le  Cas  3 et ils   sont tous nuls. }
  
\vspace{1cm}

\vspace{1cm}  
{\bf Calcul de $\boldsymbol{ \alpha_k\cup\alpha_j}$} 

Ce cup-produit intervient dans le Cas 3 pour tous les Types. Le d\'eroulement de la preuve est la m\^eme que pour $p=2$. Il faut maintenant tenir compte des signes $\pm$ et les divisibilit\'es par $p$.

{\bf Type $o_i$} Le g\'en\'erateur $\alpha_k$ est alors $\alpha_k=[\hat q_k-\hat q_0]$ pour $1\le k\le n-1.$ Nous devons calculer les coefficients $r_\ell$ et $y_i$.

\noindent $\bullet$ Calcul des coefficients $r_\ell=\sum_w R(\hat q_k-\hat q_0)(\mu_{\ell,w})_2
R(\hat q_j-\hat q_0)(\mu_{\ell,w})_0,$ avec $(\mu_{\ell,w})_0=x_{\ell,w}$ et $(\mu_{\ell,w})_2=p_{\ell,w}$.\\ 
1) Dans $R(\hat q_j-\hat q_0)$ interviennent $Z_j$ et $-Z_0$ et dans $Z_u,$
 on voit $\hat q_u-\sum_{s\ge 1}\hat p_{u,s}\sharp\{t\geq s \mid x_{u,t}=q_u\}$.\\
 -Si $u=j$,   $R(\hat q_j-\hat q_0)(\mu_{\ell,w})_0=R(\hat q_j-\hat q_0)(x_{\ell,w})=0$ sauf si $\ell=j$ et ceci pour tous les $a_j$ indices $w$ tels que $x_{j,w}=q_j$. Dans ces situations $R(\hat q_j-\hat q_0)(\mu_{j,w})_0=1$. \\
 - Si $u=0$, $R(\hat q_j-\hat q_0)(\mu_{\ell,w})_0=0$ sauf si $\ell=0$ et ceci pour tous les $a_0$ indices $w$ tels que $x_{0,w}=q_0$. Dans ces situations $R(\hat q_j-\hat q_0)(\mu_{0,w})_0=-1$. \\
 2) $R(\hat q_k-\hat q_0)(\mu_{\ell,w})_2=R(\hat q_k-\hat q_0)(p_{\ell,w})=0$ sauf si \\
 i) $\ell=j=k$ et pour tous les $w$ tels que $x_{k,w}=q_k$. Dans ces situations on a $R(\hat q_k-\hat q_0)(p_{k,w})=-\sharp\{t\geq w \mid x_{k,t}=q_k\}$ ;\\
 ii) $\ell=0$ et pour tous les $w$ tels que $x_{0,w}=q_0$. Dans ces situations on a $R(\hat q_k-\hat q_0)(p_{0,w})=-\sharp\{t\geq w \mid x_{0,t}=q_0\}$.
 
 Si $k=j$, nous avons obtenu 
 $$r_k=-\sum_{\begin{array}{c}w\ge 1\\ x_{k,w}=q_k\end{array}}
 R(\hat q_k-\hat q_0)(\mu_{k,w})_2=-\sum_{w=1}^{w=a_k-1}(a_k-w)= -\frac{a_k(a_k-1)}{2}.$$
Comme  $p\mid a_k$, tous les $r_k$ sont nuls. (Remarquons que ceci n'est pas vrai quand $p=2$.)

\noindent $\bullet$ Calcul des coefficients
$$y_i=R(\hat q_k-\hat q_0)(\nu_{i,1})_2R(\hat q_j-\hat q_0)(\nu_{i,1})_0-\varepsilon_i R(\hat q_k-\hat q_0)(\nu_{i,2})_2R(\hat q_j-\hat q_0)(\nu_{i,2})_0. $$
Nous avons\\
 $(\nu_{i,1})_2= t_i$ si $\varepsilon_i=1$ et $(\nu_{i,1})_2= f_i$ si $\varepsilon_i=-1$ ;\\
$(\nu_{i,1})_0=(\nu_{i,2})_2=h;  (\nu_{i,2})_0=t_i.$\\
 Aucun de ces \'el\'ements n'intervient dans $R(\hat q_u-\hat q_0)$. On a donc que pour tout $i, y_i=0.$

- Il reste \`a remarquer que les calculs pr\'ec\'edents sont valables pour les Types $o_1$ et $o_2$.\\
{\sl Conclusion. Pour les Types $o_i$, tous les cup-produits $\alpha_k\cup\alpha_j$ sont nuls.}

{\bf Pour les Types $n_i$.} Maintenant le g\'en\'erateur  est $\alpha=[ \hat q_k-\frac{1}{2}\hat t_g].$

\noindent $\bullet$ Calcul des $r_k$
La diff\'erence avec le paragraphe pr\'ec\'edent est que le terme $Z_0$ n'intervient pas.  Les calculs des $r_k$ restent les m\^emes et les $r_k$ sont nuls puisque $p\mid a_k$.

\noindent $\bullet$Dans les $R(\hat q_k-\frac{1}{2}\hat t_g)$ il y a maintenant $\frac{1}{2}t_g$ mais $R(\hat q_j-\frac{1}{2}\hat t_g)(\nu_{i,1})_0=R(\hat q_j-\frac{1}{2}\hat t_g)(h)=0$ et 
$R(\hat q_k-\frac{1}{2}\hat t_g)(\nu_{i,1})_2=R(\hat q_k-\frac{1}{2}\hat t_g)(h)=0.$ Ici aussi tous les $y_i$ sont nuls.\\
{\sl Conclusion. Pour les Types $n_i$, tous les cup-produits $\alpha_k\cup\alpha_j$ sont nuls.}

\end{proof}
 
 

\subsection{ Les cup-produits, pour $p>2$, $\cup:H^1\tens H^2\to H^3$}\label{sub:12=3,p}

Il suffit de consid\' erer les Types $o_1$ et $n_2$, car dans le cas non orientable, le $H^3$ \`a coefficients dans un anneau $A$ vaut $A/2A$  donc il est nul si $A=\Z/p\Z$.

On rappelle que si $f$ est un 1-cocycle et $g$ un 2-cocycle, et $s$ un 3-simplexe de faces $s_0=(v_1,v_2,v_3)$, $s_1=(v_0,v_2,v_3)$, $s_2=(v_0,v_1,v_3)$ et $s_3=(v_0,v_1,v_2)$ alors $f\cup g(s)=f(v_0,v_1)g(s_0)$, et on trouve $(v_0,v_1)$ en prenant la derni\`ere ar\^ete de $s_2$ ou $s_3$, i.e. $(v_0,v_1)=(s_2)_2=(s_3)_2$.

Mais aussi, si $g$ est un 2-cocycle, $f$ un 1-cocycle et $s=(s_0,s_1,s_2,s_3)$ un 3-simplexe de faces $s_0=(v_1,v_2,v_3)$, $s_2=(v_0,v_2,v_3)$ et $s_3=(v_0,v_1,v_2)$, o\`u les $v_i$ sont les sommets, alors $g\cup f(s)=g(s_3)f(v_2,v_3)$. On trouve $(v_2,v_3)$ en prenant la premi\`ere ar\^ete de $s_1$ ou de $s_0$, i.e. $(v_2,v_3)=(s_1)_0=(s_0)_0.$

Dans ces dimensions de cocycles, on a $f\cup g= (-1)^{1\times 2}g\cup f= g\cup f.$

\begin{theo} Pour $p>2$, les seuls cup-produits $\cup:H^1\tens H^2\to H^3$ sont :\\
$\bullet$  Dans les trois Cas 
\begin{itemize}
\item  $\boldsymbol{\theta_i\cup\varphi_j}$,\\
- Pour le Type $o_1$, 
 on a pour $j$ impair $\theta_{j+1}\cup\varphi_j=-\gamma$ et pour $j$ pair $\theta_{j-1}\cup\varphi_j=\gamma,$\\
- Pour le Type $n_2$,  on a $\theta_j\cup\varphi_j=\gamma,$ et 0 sinon.
\end{itemize}
Avec en plus\\
$\bullet$ dans le Cas 1, pour tous les Types
\begin{itemize}
\item   $\boldsymbol{\theta_i\cup\beta=0;}$
\item  $\boldsymbol{\alpha\cup\beta=\gamma ;}$
\item $\boldsymbol{\alpha\cup\varphi_j=0}.$
\end{itemize}
$\bullet$ dans le Cas 3
\begin{itemize} 
 \item $\boldsymbol{\alpha_i\cup\beta_k}$,\\
- Pour les Types $o_1$ et $n_2$,  on a 
$\alpha_i\cup\beta_k=0$ sauf si $i=k$ et dans cette situation on a 
$\alpha_k\cup\beta_k=b_k^{-1}\gamma.$
\item $\boldsymbol{\alpha_k\cup\varphi_j}$,\\
-  Pour le Type $o_1$,  on a $\alpha_k\cup\varphi_j=0.$ \\
\item $\boldsymbol{\alpha_k\cup\varphi_j}$,\\
- Pour le Type $n_2$,  
 pour tous indices $k$, on a
 $\alpha_k\cup\varphi_g=-\frac{1}{2}\gamma$.  
\end{itemize}

\end{theo}


\begin{proof}

{\bf Calcul de $\boldsymbol{\theta_i\cup \varphi_j}$} 

Ces cup-produits interviennent  dans les trois Cas.

{\bf Type $o_1$} ,  $\theta_i=[\hat t_i], 1\le i\le 2g$ et $\varphi_j=[\hat \nu_j], 1\le j\le 2g$.\\
Comme pour $p=2$, on trouve $R(\hat t_i)\cup R(\hat\nu_j)= 0$ sauf \\
- si $j$ est impair, $R(\hat t_{j+1})\cup R(\hat\nu_j)=N'_{j,2}$\\
- si $j$ est pair $R(\hat t_{j-1})\cup R(\hat\nu_j) =N_{j,2}$.\\

Pour $p>2$, on a\\ 
- pour $j$ impair,  $T^t(N'_{j,2})=-\epsilon$ d'o\`u  $\theta_{j+1}\cup\varphi_j=-\gamma$\\
-  pour $j$ pair, $T^t(N'_{j-1,2})=\epsilon$ d'o\`u $\theta_{j-1}\cup\varphi_j=\gamma.$\\
{\sl Conclusion : Dans les trois Cas, pour le Type $o_1$ , on a\\
- pour $j$ impair $\theta_{j+1}\cup\varphi_j=-\gamma$\\
- pour $j$ pair $\theta_{j-1}\cup\varphi_j=\gamma.$}

{\bf Type $n_2$}\\
On a $\theta_i=[\hat t_i -\hat t_1], i>1,$ et $\varphi_j=[\hat\nu_j], j>1$. Pour le Type $n_2$, tous les $\varepsilon_j$ sont \'egaux \`a $-1$. Le relev\'e de $\theta_i$ n'est plus le m\^eme que pour $p=2$, mais on a encore $R(\hat t_i -\hat t_1)\cup R(\hat\nu_j)=0$ sauf si $i=j$.
De plus on a $T^t(N'_{j,2})=\epsilon$, d'o\`u la conclusion :\\
{\sl Conclusion : Dans les trois Cas,  pour le Type $n_2$,  on a $\theta_j\cup\varphi_j=\gamma,$ et 0 sinon.}

\vspace{1cm}

{\bf Calcul de $\boldsymbol{\theta_i\cup\beta}$, $\boldsymbol{\alpha\cup\beta}$, $\boldsymbol{\alpha\cup\varphi_j}$} 

Ces cup-produits n'interviennent que pour le Type $o_1$ dans le Cas 1.

La preuve est exactement la m\^eme que pour $p=2$, d'o\`u la conclusion\\
{\sl Conclusion : On a toujours $\theta_i\cup\beta=0,$ $\alpha\cup\beta=\gamma,$ $\alpha\cup\varphi_j=0.$}

\vspace{1cm}

{\bf Calcul de $\boldsymbol{\alpha_i\cup\beta_k}$ } 

Ces cup-produits n'interviennent que dans le Cas 3.
Les calculs suivants ne d\'ependent pas des Types. 

Comme pour $p=2$, le seul 3-simplexe  $s$ tel que $R(\hat\mu_k)(s_0)=1$ est 
$s=M_{k,1}^\pm$. On a   $s_3=P^\pm_{k,1}$ et $(P^\pm_{k,1})_2=C_k^\pm$. Dans le relev\' e de $\alpha_i$ appara\^\i t seulement (via $Z_k$) $C_k^+$, affect\' e du coefficient $-v_k$, si $i=k$.

Du fait que que $a_ku_k-b_kv_k=1$ et que $p$ divise $a_k$, on $-v_k=b_k^{-1}$ dans $\Z_p$.
On v\'erifie que $T^t(M_{k,1}^\pm)=\zeta_k$ et on a $\gamma=[\hat\zeta_k]$.\\
{\sl Conclusion : Dans le Cas 3, pour les Types $o_1$ et $n_2$,
$\alpha_i\cup\beta_k=0$ sauf si $i=k$ et dans cette situation on a 
$\alpha_k\cup\beta_k=b_k^{-1}\gamma.$}

\vspace{1cm}

{\bf Calcul de $\boldsymbol{\alpha_k\cup\varphi_j}$ } 

Ces cup-produits n'interviennent que dans le Cas 3.

{\bf Type $o_1$,} $\alpha_k=[\hat q_k-\hat q_{0}], 1\leq k\leq m$ et $\varphi_j=[\hat \nu_j], 1\leq j\leq 2g.$\\

La preuve est exactement la m\^eme que pour $p=2$, d'o\`u la conclusion \\
{\sl Conclusion : Dans le Cas 3, pour le Type $o_1$, tous les cup-produits $\alpha_k\cup\varphi_j$ sont nuls.}

\vspace{0.5cm}

{\bf Type $n_2$,} $\alpha_k=[\hat q_k-\frac{1}{2}\hat t_g], 0\leq k\leq m$ et $\varphi_j=[\hat \nu_j], 1< j\leq g.$\\
Comme pour $p=2$, 
les seuls 3-simplexes $s$ tels que $R(\hat\nu_j)s_0\neq 0$ sont $s=N_{j,1}$ et  $s=N'_{j,1}$ car on a $(N'_{j,1})_0=(N_{j,1})_0=\nu_{j,1}.$ On a $((N_{j,1})_3)_2=(F_{2j-2})_2=S^-_{2j-2}$ et  $((N'_{j,1})_3)_2=(F_{2j-1})_2=S^+_{2j-1}$. \\
- Le relev\'e de $\alpha_k$ est maintenant\\
$R(\hat q_k-\frac{1}{2}\hat t_g)=Z_k-\sum_{\ell=0}^k (\hat e_{2g+\ell}+\hat S_{2g+\ell}^\pm)-\frac{1}{2}(\hat e_{2g-1}+\hat S_{2g-1}^\pm)$. Par cons\'equent, on obtient \\
  $R(\hat q_k-\frac{1}{2}\hat t_g)((N_{j,1})_3)_2=0$, $ R(\hat q_k-\frac{1}{2}\hat t_g)((N'_{j,1}))_3)_2=-\frac{1}{2}$ lorsque $j=g$ d'o\`u
$R(\hat q_k-\frac{1}{2}\hat t_g)\cup R(\hat\nu_j)=  -\frac{1}{2}T^t(N'_{g,1})$. Comme $T^t(N'_{g,1})=\epsilon$, on a la conclusion.\\
{\sl Conclusion : Dans le Cas 3, pour le Type $n_2$, pour tous indices $k$, on a $\alpha_k\cup\varphi_j=-\frac{1}{2}\gamma$.  }

\end{proof}

\begin{rem}
Les quelques diff\' erences de signe avec les r\'esultats obtenus pr\'ec\'edemment ( voir par exemple \cite{bz}) s'expliquent par le fait que les  g\'en\'erateurs not\'es $\alpha$ sont de signe oppos\'e. De plus les $\beta_k$ qui apparaissent naturellement ici sont des  multiplies des g\'en\'erateurs   not\'es $b_k$ dans  \cite{bz}, ce qui modifiera certains produits par ces facteurs. Pour $n_2$ nous avons choisi (pour \' eviter de distinguer inutilement les cas $n=0$ et $n>0$) des g\' en\' erateurs du $H^1$ diff\' erents, mais ces perturbations sont tu\' ees dans les produits.
\end{rem}

\vfill\eject

\section{Figures}\label{section:fig}

\vspace{1cm}

\ifx\JPicScale\undefined\def\JPicScale{1}\fi
\psset{unit=\JPicScale mm}
\psset{linewidth=0.3,dotsep=1,hatchwidth=0.3,hatchsep=1.5,shadowsize=1,dimen=middle}
\psset{dotsize=0.7 2.5,dotscale=1 1,fillcolor=black}
\psset{arrowsize=1 2,arrowlength=1,arrowinset=0.25,tbarsize=0.7 5,bracketlength=0.15,rbracketlength=0.15}
\begin{pspicture}(30,0)(140,31)
\psline[linestyle=dotted](120,25)(140,25)
\psline[linestyle=dotted](120,5)(140,5)
\psline{<-}(30,15)(30,5)
\psline{<-}(50,15)(50,5)
\psline{<-}(70,15)(70,5)
\psline{<-}(90,15)(90,5)
\psline{<-}(110,15)(110,5)
\psline{<-}(150,15)(150,5)
\psline{<-}(170,15)(170,5)
\psline{<-}(40,5)(30,5)
\psline{<-}(60,5)(50,5)
\psline{<-}(80,5)(90,5)
\psline{<-}(100,5)(110,5)
\psline{<-}(160,5)(150,5)
\psline{<-}(40,25)(30,25)
\psline{<-}(60,25)(50,25)
\psline{<-}(80,25)(90,25)
\psline{<-}(100,25)(110,25)
\psline{<-}(160,25)(150,25)
\psline(30,5)(20,5)
\psline(40,5)(50,5)
\psline(60,5)(80,5)
\psline(90,5)(100,5)
\psline(110,5)(120,5)
\psline(140,5)(150,5)
\psline(160,5)(180,5)
\psline(30,25)(20,25)
\psline(40,25)(50,25)
\psline(60,25)(80,25)
\psline(90,25)(100,25)
\psline(110,25)(120,25)
\psline(140,25)(150,25)
\psline(160,25)(180,25)
\psline(30,15)(30,25)
\psline(50,15)(50,25)
\psline(70,25)(70,15)
\psline(90,25)(90,15)
\psline(110,25)(110,15)
\psline(150,25)(150,15)
\psline(170,25)(170,15)
\rput[bl](27,15){\scriptsize $h$}
\rput[bl](47,15){\scriptsize $h$}
\rput[bl](67,15){\scriptsize $h$}
\rput[bl](87,15){\scriptsize $h$}
\rput[bl](113,15){\scriptsize $h$}
\rput[bl](147,15){\scriptsize $h$}
\rput[bl](167,15){\scriptsize $h$}
\rput(40,2){\scriptsize $t_1$}
\rput(60,2){\scriptsize $t_2$}
\rput(80,2){\scriptsize $t_1$}
\rput(100,2){\scriptsize $t_2$}
\rput(40,28){\scriptsize $t_1$}
\rput(60,28){\scriptsize $t_2$}
\rput(80,28){\scriptsize $t_1$}
\rput(100,28){\scriptsize $t_2$}
\rput(20,2){\scriptsize $q_m$}
\rput(20,28){\scriptsize $q_m$}
\rput(160,2){\scriptsize $q_0$}
\rput(160,28){\scriptsize $q_0$}
\rput(40,15){\bf{\large $\nu_1$}}
\rput(60,15){\bf{\large  $\nu_2$}}
\rput(80,15){\bf{\large  $-\nu_1$}}
\rput(100,15){\bf{\large  $-\nu_2$}}
\rput(20,15){\bf{\large  $\rho_m$}}
\rput(160,15){\bf{\large  $\rho_0$}}
\rput(87,0){\bf{\large $-\delta$}}
\rput(90,31){\bf{\large  $\delta$}}
\end{pspicture}

\vspace{0.3cm}

  \caption{D\'ecomposition cellulaire, Type $o_1$}
  
==================================
\vspace{0.5cm}

\ifx\JPicScale\undefined\def\JPicScale{1}\fi
\psset{unit=\JPicScale mm}
\psset{linewidth=0.3,dotsep=1,hatchwidth=0.3,hatchsep=1.5,shadowsize=1,dimen=middle}
\psset{dotsize=0.7 2.5,dotscale=1 1,fillcolor=black}
\psset{arrowsize=1 2,arrowlength=1,arrowinset=0.25,tbarsize=0.7 5,bracketlength=0.15,rbracketlength=0.15}
\begin{pspicture}(30,30)(180,90)
\psline[linestyle=dotted](120,70)(140,70)
\psline[linestyle=dotted](120,50)(140,50)
\psline{<-}(30,60)(30,50)
\psline{<-}(50,60)(50,50)
\psline{<-}(70,60)(70,50)
\psline{<-}(90,60)(90,50)
\psline{<-}(110,60)(110,50)
\psline{<-}(150,60)(150,50)
\psline{<-}(170,60)(170,50)
\psline{<-}(40,50)(30,50)
\psline{<-}(60,50)(50,50)
\psline{<-}(80,50)(90,50)
\psline{<-}(100,50)(110,50)
\psline{<-}(160,50)(150,50)
\psline{<-}(40,70)(30,70)
\psline{<-}(60,70)(50,70)
\psline{<-}(80,70)(90,70)
\psline{<-}(100,70)(110,70)
\psline{<-}(160,70)(150,70)
\psline(30,50)(20,50)
\psline(40,50)(50,50)
\psline(60,50)(80,50)
\psline(90,50)(100,50)
\psline(110,50)(120,50)
\psline(140,50)(150,50)
\psline(160,50)(180,50)
\psline(30,70)(20,70)
\psline(40,70)(50,70)
\psline(60,70)(80,70)
\psline(90,70)(100,70)
\psline(110,70)(120,70)
\psline(140,70)(150,70)
\psline(160,70)(180,70)
\psline(30,70)(30,60)
\psline(50,70)(50,60)
\psline(70,70)(70,60)
\psline(90,70)(90,60)
\psline(110,70)(110,60)
\psline(150,70)(150,60)
\psline(170,70)(170,60)
\rput[bl](27,60){\scriptsize $h$}
\rput[bl](47,60){\scriptsize $h$}
\rput[bl](67,60){\scriptsize $h$}
\rput[bl](87,60){\scriptsize $h$}
\rput[bl](113,60){\scriptsize $h$}
\rput[bl](147,60){\scriptsize $h$}
\rput[bl](167,60){\scriptsize $h$}
\rput(40,47){\scriptsize $t_1$}
\rput(60,47){\scriptsize $t_2$}
\rput(80,47){\scriptsize $t_1$}
\rput(100,47){\scriptsize $t_2$}
\rput(40,73){\scriptsize $t_1$}
\rput(60,73){\scriptsize $t_1$}
\rput(80,73){\scriptsize $t_1$}
\rput(100,73){\scriptsize $t_2$}
\rput(20,47){\scriptsize $q_m$}
\rput(20,73){\scriptsize $q_m$}
\rput(160,47){\scriptsize $q_0$}
\rput(160,73){\scriptsize $q_0$}
\rput(40,60){\bf{\large  $-\nu_1$}}
\rput(60,60){\bf{\large  $\nu_2$}}
\rput(80,60){\bf{\large  $-\nu_1$}}
\rput(100,60){\bf{\large  $\nu_2$}}
\rput(20,60){\bf{\large $\rho_m$}}
\rput(160,60){\bf{\large   $\rho_0$}}
\rput(90,30){\bf{\large  $-\delta$}}
\rput(90,90){\bf{\large  $\delta$}}
\end{pspicture}

\vspace{0.3cm}

\caption{D\'ecomposition cellulaire, Type $o_2$} 

==================================


\vspace{0.5cm}

\begin{centering}

\ifx\JPicScale\undefined\def\JPicScale{1.3}\fi
\psset{unit=\JPicScale mm}
\psset{linewidth=0.3,dotsep=1,hatchwidth=0.3,hatchsep=1.5,shadowsize=1,dimen=middle}
\psset{dotsize=0.7 2.5,dotscale=1 1,fillcolor=black}
\psset{arrowsize=1 2,arrowlength=1,arrowinset=0.25,tbarsize=0.7 5,bracketlength=0.15,rbracketlength=0.15}
\begin{pspicture}(20,0)(130,36)
\psline{<-}(20,30)(10,30)
\psline{<-}(20,10)(10,10)
\psline{<-}(30,20)(30,10)
\psline{<-}(50,20)(50,10)
\psline{<-}(70,20)(70,10)
\psline{<-}(100,20)(100,10)
\psline{<-}(110,10)(100,10)
\psline(80,30)(20,30)
\psline(30,30)(30,20)
\psline(20,10)(30,10)
\psline(30,10)(80,10)
\psline(50,30)(50,20)
\psline(70,30)(70,20)
\psline(100,30)(100,20)
\psline(100,30)(90,30)
\psline(100,10)(90,10)
\psline[linestyle=dotted](90,30)(80,30)
\psline[linestyle=dotted](90,10)(80,10)
\psline{<-}(110,30)(100,30)
\rput[bl](27,20){\scriptsize $h$}
\rput[bl](47,20){\scriptsize $h$}
\rput[bl](67,20){\scriptsize $h$}
\rput[bl](97,20){\scriptsize $h$}
\rput[bl](20,32){\scriptsize $q_m$}
\rput[bl](40,30){}
\rput[bl](105,32){\scriptsize $q_0$}
\rput[bl](20,7){\scriptsize $q_m$}
\rput[bl](110,7){\scriptsize $q_0$}
\rput[bl](14,20){\bf{\large  $\rho_m$}}
\rput[bl](36,18){\bf{\large $B_1$}}
\rput[bl](56,18){\bf{\large  $B_2$}}
\rput[bl](110,20){\bf{\large  $\rho_0$}}
\rput[bl](60,35){\bf{\large  $\delta$}}
\rput[bl](50,1){\bf{\large $-\delta$}}
\end{pspicture}
\caption{D\'ecomposition cellulaire, Type $n_i$}

\end{centering}
==================================


\vspace{0.5cm}

\ifx\JPicScale\undefined\def\JPicScale{1}\fi
\psset{unit=\JPicScale mm}
\psset{linewidth=0.3,dotsep=1,hatchwidth=0.3,hatchsep=1.5,shadowsize=1,dimen=middle}
\psset{dotsize=0.7 2.5,dotscale=1 1,fillcolor=black}
\psset{arrowsize=1 2,arrowlength=1,arrowinset=0.25,tbarsize=0.7 5,bracketlength=0.15,rbracketlength=0.15}
\begin{pspicture}(0,0)(110,60)
\psline{<-}(20,35)(10,35)
\psline{<-}(10,25)(10,15)
\psline{<-}(20,15)(10,15)
\psline{<-}(30,25)(30,15)
\psline{<-}(40,15)(30,15)
\psline{<-}(40,35)(30,35)
\psline{<-}(50,25)(50,15)
\rput[bl](27,40){$\varepsilon_j=1$}
\rput[bl](27,0){\bf{\large  $B_j$}}
\psline(10,35)(10,25)
\psline(20,35)(30,35)
\psline(30,35)(30,25)
\psline(20,15)(30,15)
\psline(40,35)(50,35)
\psline(50,35)(50,25)
\psline(50,15)(40,15)
\rput[bl](7,25){\scriptsize $h$}
\rput[bl](27,25){\scriptsize $h$}
\rput[bl](47,25){\scriptsize $h$}
\rput[bl](20,36){\scriptsize $t_j$}
\rput[bl](17,10){\scriptsize $t_j$}
\rput[bl](18,23){\bf{\large $\nu_j$}}
\rput[bl](38,23){\bf{\large $\nu_j$}}
\rput[bl](40,36){\scriptsize $t_j$}
\rput[bl](40,10){\scriptsize $t_j$}

\psline{<-}(80,35)(70,35)
\psline{<-}(70,25)(70,15)
\psline{<-}(80,15)(70,15)
\psline{<-}(100,15)(90,15)
\psline{<-}(100,35)(90,35)
\psline{<-}(110,25)(110,15)
\rput[bl](87,0){\bf{\large $B_j$}}
\psline(70,35)(70,25)
\psline(80,35)(90,35)
\psline(80,15)(90,15)
\psline(100,35)(110,35)
\psline(110,35)(110,25)
\psline(110,15)(100,15)
\rput[bl](67,25){\scriptsize $h$}
\rput[bl](86,25){\scriptsize $h$}
\rput[bl](107,25){\scriptsize $h$}
\rput[bl](80,11){\scriptsize $t_j$}
\rput[bl](71,22){\bf{\large $-\nu_j$}}
\rput[bl](98,22){\bf{\large$\nu_j$}}
\rput[bl](100,11){\scriptsize $t_j$}
\rput[bl](85,40){$\varepsilon_j=-1$}
\rput[bl](80,31){\scriptsize $t_j$}
\rput[bl](100,31){\scriptsize $t_j$}
\psline(90,15)(90,25)
\psline{<-}(90,25)(90,35)
\end{pspicture}

\vspace{0.5cm}

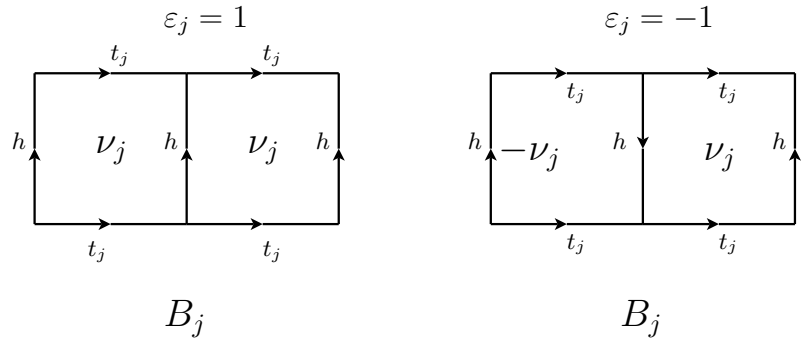
\caption{Description de $B_j$ pour les Types $n_1, n_2$}

==================================


\ifx\JPicScale\undefined\def\JPicScale{1}\fi
\psset{unit=\JPicScale mm}
\psset{linewidth=0.3,dotsep=1,hatchwidth=0.3,hatchsep=1.5,shadowsize=1,dimen=middle}
\psset{dotsize=0.7 2.5,dotscale=1 1,fillcolor=black}
\psset{arrowsize=1 2,arrowlength=1,arrowinset=0.25,tbarsize=0.7 5,bracketlength=0.15,rbracketlength=0.15}

\begin{pspicture}(-10,4)(100,39)
\psline[ArrowInside=->](40,5)(40,30)
\psline[ArrowInside=->](40,5)(65,5)
\psline[ArrowInside=->](65,5)(65,30)
\psline[ArrowInside=->](40,30)(65,30)
\psline[ArrowInside=->](40,5)(65,30)
\rput(67,18){\scriptsize $h$}
\rput(37,18){\scriptsize $h$}
\rput(57,18){\scriptsize $g_k$}
\rput(54,2){\scriptsize $q_k$}
\rput(54,33){\scriptsize $q_k$}
\rput(48,23){\bf{\large $\rho_{k,2}$}}
\rput(55,10){\bf{\large $\rho_{k,1}$}}
\end{pspicture}

\vspace{0.5cm}

\caption{D\'ecomposition simpliciale de $\rho_k$}

==================================
\vspace{0.5cm}

\begin{centering}
\ifx\JPicScale\undefined\def\JPicScale{1}\fi
\psset{unit=\JPicScale mm}
\psset{linewidth=0.3,dotsep=1,hatchwidth=0.3,hatchsep=1.5,shadowsize=1,dimen=middle}
\psset{dotsize=0.7 2.5,dotscale=1 1,fillcolor=black}
\psset{arrowsize=1 2,arrowlength=1,arrowinset=0.25,tbarsize=0.7 5,bracketlength=0.15,rbracketlength=0.15}

\begin{pspicture}(-3,0)(100,25)

\psline[ArrowInside=->](0,0)(25,0)
\psline[ArrowInside=->](25,0)(50,0)
\psline[ArrowInside=->](75,0)(50,0)
\psline[ArrowInside=->](100,0)(75,0)
\psline[ArrowInside=->](0,25)(25,25)
\psline[ArrowInside=->](25,25)(50,25)
\psline[ArrowInside=->](75,25)(50,25)
\psline[ArrowInside=->](100,25)(75,25)

\psline[ArrowInside=->](0,0)(0,25)
\psline[ArrowInside=->](25,0)(25,25)
\psline[ArrowInside=->](50,0)(50,25)
\psline[ArrowInside=->](75,0)(75,25)
\psline[ArrowInside=->](100,0)(100,25)

\psline[ArrowInside=->](0,0)(25,25)
\psline[ArrowInside=->](25,0)(50,25)
\psline[ArrowInside=->](75,0)(50,25)
\psline[ArrowInside=->](100,0)(75,25)

\rput(5,18){\bf{\large $\nu_{1,2}$}}
\rput(18,5){\bf{\large $\nu_{1,1}$}}
\rput(30,18){\bf{\small $\nu_{2,2}$}}
\rput(43,5){\bf{\large $\nu_{2,1}$}}
\rput(63,18){\bf{\large $-\nu_{1,2}$}}
\rput(60,5){\bf{\large $-\nu_{1,1}$}}
\rput(87,5){\bf{\large $-\nu_{2,1}$}}
\rput(90,18){\bf{\large $-\nu_{2,2}$}}
\rput(13,-4){$t_1$}
\rput(38,-4){$t_2$}
\rput(63,-4){$t_1$}
\rput(88,-4){$t_2$}
\rput(13,28){$t_1$}
\rput(38,28){$t_2$}
\rput(63,28){$t_1$}
\rput(88,28){$t_2$}
\rput(103,15){$h$}
\rput(-3,15){$h$}
\rput(22,15){$h$}
\rput(47,15){$h$}
\rput(72,15){$h$}
\rput(9,13){$f_1$}
\rput(33,13){$f_2$}
\rput(60,13){$f_1$}
\rput(83,13){$f_2$}
\end{pspicture}

\vspace{0.5cm}

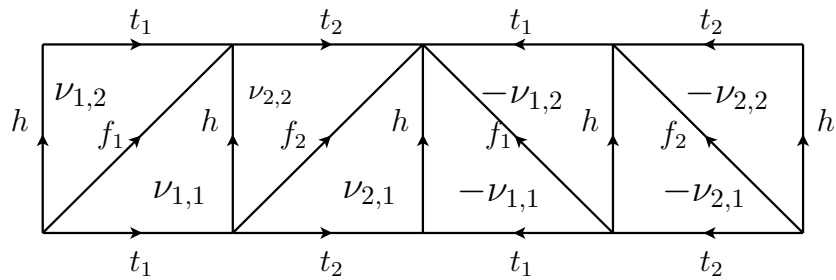
\caption{D\'ecomposition simpliciale de $\nu_1$, Type $o_1$}
\end{centering}

==================================
\vfill\eject

\vspace{0.9cm}

\ifx\JPicScale\undefined\def\JPicScale{1.5}\fi
\psset{unit=\JPicScale mm}
\psset{linewidth=0.3,dotsep=1,hatchwidth=0.3,hatchsep=1.5,shadowsize=1,dimen=middle}
\psset{dotsize=0.7 2.5,dotscale=1 1,fillcolor=black}
\psset{arrowsize=1 2,arrowlength=1,arrowinset=0.25,tbarsize=0.7 5,bracketlength=0.15,rbracketlength=0.15}
\begin{pspicture}(-10,0)(100,40)

\psline[ArrowInside=->](0,0)(25,0)
\psline[ArrowInside=->](25,0)(50,0)
\psline[ArrowInside=->](75,0)(50,0)
\psline[ArrowInside=->](100,0)(75,0)
\psline[ArrowInside=->](0,25)(25,25)
\psline[ArrowInside=->](25,25)(50,25)
\psline[ArrowInside=->](75,25)(50,25)
\psline[ArrowInside=->](100,25)(75,25)

\psline[ArrowInside=->](0,0)(0,25)
\psline[ArrowInside=->](25,25)(25,0)
\psline[ArrowInside=->](50,0)(50,25)
\psline[ArrowInside=->](75,25)(75,0)
\psline[ArrowInside=->](100,0)(100,25)

\psline[ArrowInside=->](0,0)(25,25)
\psline[ArrowInside=->](25,25)(50,0)
\psline[ArrowInside=->](50,0)(75,25)
\psline[ArrowInside=->](100,0)(75,25)

\rput(7,18){\bf{\large $-\nu_{1,2}$}}
\rput(18,5){\bf{\large $-\nu_{1,1}$}}
\rput(33,5){\bf{\large $-\nu_{2,2}$}}
\rput(41,18){\bf{\large $-\nu_{2,1}$}}
\rput(61,18){\bf{\large $-\nu_{1,2}$}}
\rput(67,5){\bf{\large $-\nu_{1,1}$}}
\rput(87,5){\bf{\large $-\nu_{2,1}$}}
\rput(90,18){\bf{\large $-\nu_{2,2}$}}
\rput(13,-4){$t_1$}
\rput(38,-4){$t_2$}
\rput(63,-4){$t_1$}
\rput(88,-4){$t_2$}
\rput(13,28){$t_1$}
\rput(38,28){$t_2$}
\rput(63,28){$t_1$}
\rput(88,28){$t_2$}

\rput(-3,15){$h$}
\rput(22,15){$h$}
\rput(47,15){$h$}
\rput(72,15){$h$}
\rput(103,15){$h$}
\rput(9,13){$f_1$}
\rput(33,13){$f_2$}
\rput(60,13){$f_1$}
\rput(83,13){$f_2$}
\end{pspicture}

\vspace{0.9cm}

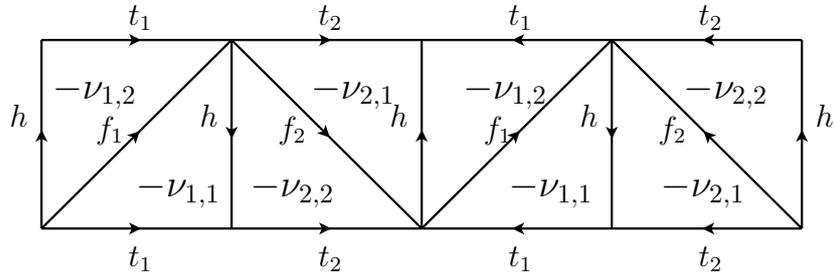
\caption{D\'ecomposition simpliciale de $-\nu_1$, Type $o_2$}

==================================


\vspace{0.9cm}

\ifx\JPicScale\undefined\def\JPicScale{1.8}\fi
\psset{unit=\JPicScale mm}
\psset{linewidth=0.3,dotsep=1,hatchwidth=0.3,hatchsep=1.5,shadowsize=1,dimen=middle}
\psset{dotsize=0.7 2.5,dotscale=1 1,fillcolor=black}
\psset{arrowsize=1 2,arrowlength=1,arrowinset=0.25,tbarsize=0.7 5,bracketlength=0.15,rbracketlength=0.15}
\begin{pspicture}(0,0)(130,60)
\psline [ArrowInside=->](7,10)(7,35)
\psline [ArrowInside=->](32,10)(32,35)
\psline [ArrowInside=->](57,10)(57,35)

\psline [ArrowInside=->](7,10)(32,10)
\psline [ArrowInside=->](32,10)(57,35)
\psline [ArrowInside=->](7,35)(32,35)
\psline [ArrowInside=->](7,10)(32,35)
\psline [ArrowInside=->](32,10)(57,10)
\psline [ArrowInside=->](32,35)(57,35)

\rput[bl](4,22){\scriptsize $h$}
\rput[bl](29,22){\scriptsize $h$}
\rput[bl](54,22){\scriptsize $h$}
\rput[bl](20,6){\scriptsize $t_j$}
\rput[bl](45,6){\scriptsize $t_j$}
\rput[bl](20,36){\scriptsize $t_j$}
\rput[bl](45,36){\scriptsize $t_j$}
\rput[bl](20,19){\scriptsize $f_j$}
\rput[bl](45,19){\scriptsize $f_j$}

\rput[bl](27,42){\bf\small$\varepsilon_j=1$}

\rput[bl](12,25){\bf{\large $\nu_{j,2}$}}
\rput[bl](37,25){\bf{\large $\nu_{j,2}$}}
\rput[bl](21,12){\bf{\large $\nu_{j,1}$}}
\rput[bl](46,12){\bf{\large $\nu_{j,1}$}}

\psline [ArrowInside=->](72,10)(72,35)
\psline [ArrowInside=->](97,35)(97,10)
\psline [ArrowInside=->](122,10)(122,35)

\psline [ArrowInside=->](72,10)(97,10)
\psline [ArrowInside=->](97,35)(122,10)
\psline [ArrowInside=->](72,35)(97,35)
\psline [ArrowInside=->](72,10)(97,35)
\psline [ArrowInside=->](97,10)(122,10)
\psline [ArrowInside=->](97,35)(122,35)

\rput[bl](69,22){\scriptsize $h$}
\rput[bl](94,22){\scriptsize $h$}
\rput[bl](123,22){\scriptsize $h$}
\rput[bl](85,6){\scriptsize $t_j$}
\rput[bl](110,6){\scriptsize $t_j$}
\rput[bl](85,36){\scriptsize $t_j$}
\rput[bl](110,36){\scriptsize $t_j$}
\rput[bl](85,19){\scriptsize $f_j$}
\rput[bl](113,19){\scriptsize $f_j$}

\rput[bl](92,42){\bf\small $\varepsilon_j=-1$}

\rput[bl](77,25){\bf{\large $\nu_{j,2}$}}
\rput[bl](84,12){\bf{\large $\nu_{j,1}$}}
\rput[bl](100,12){\bf{\large $-\nu_{j,2}$}}
\rput[bl](103,25){\bf{\large $-\nu_{j,1}$}}
\end{pspicture}

\vspace{0.9cm}

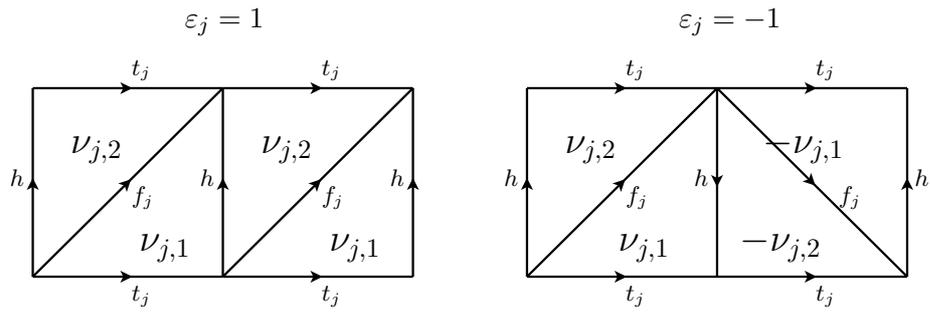
\caption{D\'ecomposition simpliciale de $-\nu_j$ quand $\varepsilon_j=1$ et quand $\varepsilon_j=-1$ }

==================================
\vfill\eject

\ifx\JPicScale\undefined\def\JPicScale{2}\fi
\psset{unit=\JPicScale mm}
\psset{linewidth=0.3,dotsep=1,hatchwidth=0.3,hatchsep=1.5,shadowsize=1,dimen=middle}
\psset{dotsize=0.7 2.5,dotscale=1 1,fillcolor=black}
\psset{arrowsize=1 2,arrowlength=1,arrowinset=0.25,tbarsize=0.7 5,bracketlength=0.15,rbracketlength=0.15}
\begin{pspicture}(-20,0)(66.12,70.4)
\rput{0}(35.02,39.3){\psellipse[linewidth=0.35,linestyle=dashed,dash=1 1](0,0)(31.1,31.1)}
\psline[ArrowInside=->](34,39.38)(5,39.38)
\psline[ArrowInside=->](34,40)(9.38,55.62)
\psline[ArrowInside=->](34,39)(9,22)
\psline[ArrowInside=->](34.38,39.38)(21,11)
\psline[ArrowInside=->](34.38,39.38)(39,9)
\psline[ArrowInside=->](34,40)(57,17)
\psline[ArrowInside=->](34.38,39.38)(63,53)
\psline[ArrowInside=->](34.38,39.38)(51,66)
\psline[ArrowInside=->](8.75,55.62)(4.38,39.38)
\psline[ArrowInside=->](4.38,39.38)(9.38,22.5)
\psline[ArrowInside=->](9.38,22.5)(21.25,11.25)
\psline[ArrowInside=->](39.38,8.75)(20.62,11.25)
\psline[ArrowInside=->](56.88,16.88)(39.38,8.75)
\psline[ArrowInside=->](63.12,53.12)(51.25,66.25)
\psbezier[linewidth=0.25,linestyle=dashed,dash=1 1](41.25,53.12)(34.25,56.19)(28.81,55.06)(23.12,49.38)
\psbezier[linewidth=0.25,linestyle=dashed,dash=1 1](48.75,45)(50.94,38.44)(50.38,33.94)(46.88,30)
\rput[bl](33.75,41.88){\scriptsize $a$}
\rput[bl](1,48){\scriptsize $q_m$}
\rput[bl](2,27.5){\scriptsize $t_1$}
\rput[bl](11.25,14.38){\scriptsize $t_2$}
\rput[bl](29.38,5.62){\scriptsize $t_1$}
\rput[bl](49.38,7.38){\scriptsize $t_2$}
\rput[bl](59.38,61.88){}
\rput[bl](59.38,61.88){\scriptsize $q_0$}
\rput[bl](18,51){\scriptsize $e_{4g+m}$}
\rput[bl](16.88,40){\scriptsize $e_0$}
\rput[bl](20,31.25){\scriptsize $e_1$}
\rput[bl](24,23.75){\scriptsize $e_2$}
\rput[bl](33.75,21.25){\scriptsize $e_3$}
\rput[bl](43,26.25){\scriptsize $e_4$}
\rput[bl](51.88,44.38){\scriptsize $e_{4g}$}
\rput[bl](41.88,56.25){\scriptsize $e_{4g+1}$}
\rput[bl](45,48){\bf{\large $\delta_{4g}$}}
\rput[bl](42.21,19){\bf{\large $\delta_3$}}
\rput[bl](29,17){\bf{\large $\delta_2$}}
\rput[bl](19,21.5){\bf{\large $\delta_1$}}
\rput[bl](13.75,31){\bf{\large $\delta_0$}}
\rput[bl](13.12,43){\bf{\large $\delta_{4g+m}$}}
\end{pspicture}

\vspace{1cm}

\caption{D\'ecomposition simpliciale de $\delta$, Type $o_i$}

==================================

\vspace{1cm}

\ifx\JPicScale\undefined\def\JPicScale{2}\fi
\psset{unit=\JPicScale mm}
\psset{linewidth=0.3,dotsep=1,hatchwidth=0.3,hatchsep=1.5,shadowsize=1,dimen=middle}
\psset{dotsize=0.7 2.5,dotscale=1 1,fillcolor=black}
\psset{arrowsize=1 2,arrowlength=1,arrowinset=0.25,tbarsize=0.7 5,bracketlength=0.15,rbracketlength=0.15}
\begin{pspicture}(-20,0)(66.12,70.4)
\rput{0}(35.02,39.3){\psellipse[linewidth=0.35,linestyle=dashed,dash=1 1](0,0)(31.1,31.1)}
\psline[ArrowInside=->](34,40.38)(5,39.38)
\psline[ArrowInside=->](34,40)(9.38,55.62)
\psline[ArrowInside=->](34,40)(9,22)
\psline[ArrowInside=->](34,40)(21,11)
\psline[ArrowInside=->](34,40)(63,53)
\psline[ArrowInside=->](34,40)(51,66)
\psline[ArrowInside=->](8.75,55.62)(4.38,39.38)
\psline[ArrowInside=->](4.38,39.38)(9.38,22.5)
\psline[ArrowInside=->](9.38,22.5)(21.25,11.25)
\psline[ArrowInside=->](63.12,53.12)(51.25,66.25)
\psbezier[linestyle=dashed,dash=1 1](30.62,26.25)(43.31,25.38)(48.75,30.62)(48.75,43.75)
\psbezier[linestyle=dashed,dash=1 1](41.25,54.38)(34.69,57.88)(29.06,56.94)(22.5,51.25)
\rput(33.12,43.12){\scriptsize $a$}
\rput(60.38,61.88){\scriptsize $q_0$}
\rput(2.38,48.12){\scriptsize $q_m$}
\rput(15.62,41.25){\scriptsize $e_0$}
\rput(17.5,30){\scriptsize $e_1$}
\rput(25,23.12){\scriptsize $e_2$}
\rput(12,33){\bf{\large $\delta_0$}}
\rput(19.38,23){\bf{\large $\delta_1$}}
\rput(14,47){\large $\delta_{2g+m}$}
\rput(3.38,30){\scriptsize $t_1$}
\rput(11,16){\scriptsize $t_1$}
\end{pspicture}

\vspace{1cm}

\caption{D\'ecomposition simpliciale de $\delta$, Type $n_i$}

==================================

\ifx\JPicScale\undefined\def\JPicScale{1.3}\fi
\psset{unit=\JPicScale mm}
\psset{linewidth=0.3,dotsep=1,hatchwidth=0.3,hatchsep=1.5,shadowsize=1,dimen=middle}
\psset{dotsize=0.7 2.5,dotscale=1 1,fillcolor=black}
\psset{arrowsize=1 2,arrowlength=1,arrowinset=0.25,tbarsize=0.7 5,bracketlength=0.15,rbracketlength=0.15}
\begin{pspicture}(0,0)(96,70.4)
\rput{0}(45.02,39.3){\psellipse[linestyle=dashed,dash=1 1](0,0)(31.1,31.1)}
\psline[ArrowInside=->](44,39.38)(15,39.38)
\psline[ArrowInside=->](44,40)(19.38,55.62)
\psline[ArrowInside=->](44,39)(19,22)
\psline[ArrowInside=->](44.38,39.38)(31,11)
\psline[ArrowInside=->](18.75,55.62)(14.38,39.38)
\psline[ArrowInside=->](14.38,39.38)(19.38,22.5)
\psline[ArrowInside=->](19.38,22.5)(31.25,11.25)
\rput{0}(44.05,42.46){\psellipticarc[linestyle=dashed,dash=1 1](0,0)(14,14){-101.53}{150.41}}
\rput(90,40){\bf{\large $b_k>0$}}
\rput(46.88,40){\scriptsize $c_k$}
\rput(26.25,53.12){\scriptsize $P_{k,z_{k}}$}
\rput(26.25,41.25){\scriptsize $P_{k,1}$}
\rput(27.5,30){\scriptsize $P_{k,2}$}
\rput(33.75,25){\scriptsize $P_{k,3}$}
\rput(11.88,48.12){\scriptsize $h$}
\rput(13.12,30){\scriptsize $q_{k}$}
\rput(22.5,15){\scriptsize $x_{k,2}$}
\rput(23,47){\bf{\large $\mu_{k,z_k}$}}
\rput(24.38,34){\bf{\large $\mu_{k,1}$}}
\rput(29.38,21.88){\bf{\large $\mu_{k,2}$}}
\end{pspicture}

\vspace{0.5cm}
\caption{D\'ecomposition simpliciale de $\mu_k$ pour $b_k>0$}

==================================

\vspace{0.5cm}

\ifx\JPicScale\undefined\def\JPicScale{0.5}\fi
\psset{unit=\JPicScale mm}
\psset{linewidth=0.3,dotsep=1,hatchwidth=0.3,hatchsep=1.5,shadowsize=1,dimen=middle}
\psset{dotsize=0.7 2.5,dotscale=1 1,fillcolor=black}
\psset{arrowsize=1 2,arrowlength=1,arrowinset=0.25,tbarsize=0.7 5,bracketlength=0.15,rbracketlength=0.15}

\begin{pspicture}(0,0)(101,72)
\rput{0}(50.02,44.3){\psellipse[ArrowInside=->](0,0)(31.1,31.1)}
\psline{->}(49.38,44.38)(42.19,30.1)
\psline(42.19,30.1)(36,16)
\psline{->}(35,5)(55,5)

\rput(98.62,46.25){\bf{\large $b_k=0$}}
\rput(40,60){\large $\mu_{k,1}$}
\rput(50,29){\scriptsize $P_{k,1}$}
\rput(49.38,46){\scriptsize $c_k$}
\rput(71.88,70.62){\scriptsize $q_k$}

\end{pspicture}

\vspace{0.5cm}
\caption{D\'ecomposition simpliciale de $\mu_k$ pour $b_k=0$}

==================================


\vfill\eject

 Les figures suivantes sont des projections des d\'ecompositions simpliciales de chacun des 3-simplexes. Les sommets sont de points carr\'es. Ils repr\'esentent la projection d'une ar\^ ete. Ci-dessous, nous donnons en exemple le codage des points carr\'es sur le 3-simplexe $D_0^+$.

\vspace{2cm}

\ifx\JPicScale\undefined\def\JPicScale{2.5}\fi
\psset{unit=\JPicScale mm}
\psset{linewidth=0.3,dotsep=1,hatchwidth=0.3,hatchsep=1.5,shadowsize=1,dimen=middle}
\psset{dotsize=0.7 2.5,dotscale=1 1,fillcolor=black}
\psset{arrowsize=1 2,arrowlength=1,arrowinset=0.25,tbarsize=0.7 5,bracketlength=0.15,rbracketlength=0.15}

\begin{pspicture}(10,40)(120,54.1)
\rput[Bl](63,60.62){\scriptsize $A^+$}
\rput[Bl](43,54){\small $E_0^+$}
\rput[Bl](55.62,48.12){\small  $E_1^+$}
\rput[Bl](23,34){\scriptsize $S_0^+$}
\rput[Bl](45,34){\scriptsize $S_1^+$}
\rput[Bl](35,34){\small  $T_0^+$}
\rput[bl](30,60){\bf{\Large $D_0^+$}}
\rput[bl](123,45){\scriptsize $b$}
\rput[Bl](123.12,60.62){\scriptsize $a$}
\rput[Bl](124,52){\scriptsize $A^+$}
\rput[Bl](109,37){\scriptsize $\sigma$}
\rput[Bl](116,40){\scriptsize $S_1^+$}
\rput[Bl](113,45){\scriptsize $S_0^+$}
\rput[Bl](98,37){\scriptsize $t_1$}
\psdots[dotstyle=square](27,39)
\psdots[dotstyle=square](49,39)
\psdots[dotstyle=square](62,61)
\pspolygon[fillstyle=vlines,hatchangle=0](27,39)(49,39)(62,61)
\pspolygon[fillstyle=vlines,hatchangle=0](87,39)(109,39)(122,61)
\psline[ArrowInside=->](122,45)(122,61)
\psline[ArrowInside=->](122,45)(109,39)
\psline[ArrowInside=->,linestyle=dashed](122,45)(87,39)

\end{pspicture}

\vspace{1.5cm}

\caption{Codage des points carr\'es sur le 3-simplexe $D_0^+$}

==================================


\vspace{0.5cm}

\ifx\JPicScale\undefined\def\JPicScale{3}\fi
\psset{unit=\JPicScale mm}
\psset{linewidth=0.3,dotsep=1,hatchwidth=0.3,hatchsep=1.5,shadowsize=1,dimen=middle}
\psset{dotsize=0.7 2.5,dotscale=1 1,fillcolor=black}
\psset{arrowsize=1 2,arrowlength=1,arrowinset=0.25,tbarsize=0.7 5,bracketlength=0.15,rbracketlength=0.15}

\begin{pspicture}(24,-5)(110,60)
\psline(25,20)(30.62,20)
\psdots[dotstyle=square](30.62,20)
\psline(30.62,20)(50.62,20)
\psdots[dotstyle=square](30.62,20)
(50.62,20)
\psline(50.62,20)(62.5,20)
\psdots[dotstyle=square](50.62,20)

\psline[linestyle=dashed,dash=1 1](62.5,20)(77.5,20)
\psdots[linestyle=dashed,dash=1 1,dotstyle=square](77.5,20)
\psline(77.5,20)(99.38,20)
\psdots[dotstyle=square]
(99.38,20)
\psline(24.38,40)(30,40)
\psdots[dotstyle=square]
(30,40)
\psline(30,40)(50,40)
\psdots[dotstyle=square](30,40)
(50,40)
\psline(50,40)(61.88,40)
\psdots[dotstyle=square](50,40)

\psline[linestyle=dashed,dash=1 1](61.88,40)(76.88,40)

\psline(76.88,40)(99.38,40)
\psdots[dotstyle=square](76.88,40)
(99.38,40)
\psline(99.38,40)(109.38,40)
\psdots[dotstyle=square](99.38,40)

\psline(60.62,59.38)(30,40)
\psdots[dotstyle=square](60.62,59.38)
(30,40)
\psline(60.62,59.38)(50,40)
\psdots[dotstyle=square](60.62,59.38)
(50,40)
\psline(60.62,59.38)(76.88,40)
\psdots[dotstyle=square](60.62,59.38)
(76.88,40)
\psline(60.62,59.38)(99.38,40)
\psdots[dotstyle=square](60.62,59.38)

\psline(60,0)(30.62,20)
\psdots[dotstyle=square](60,0)
(30.62,20)
\psline(60,0)(50.62,20)
\psdots[dotstyle=square](60,0)
(50.62,20)
\psline(60,0)(77.5,20)
\psdots[dotstyle=square](60,0)

\psline(60,0)(99.38,20)
\psline(30.62,20.62)(30,39.38)
\psline(50.62,20.62)(50,39.38)
\psline(77.5,20.62)(76.88,39.38)
\psline(99.38,20)(110,20)
\psline(99.38,20.62)(99.38,39.38)
\psline(78.12,20.62)(98.12,38.75)
\psline(31.25,21.25)(49.38,39.38)

\rput[Bl](62.12,60.62){\scriptsize $A^+$}
\rput[Bl](62.12,-2.5){\scriptsize $A^-$}
\rput[Bl](43.75,54){\small  $E_0^+$}
\rput[Bl](55.62,48.12){\small  $E_1^+$}
\rput[Bl](67,42.5){\scriptsize $S_{*+k}^+$}

\rput[Bl](100,41.88){\small  $T_{*+k}^+$}
\rput[Bl](25.62,41.88){\scriptsize $S_0^+$}
\rput[Bl](40.62,41.88){\small  $T_0^+$}
\rput[Bl](53.12,41.88){\scriptsize $S_1^+$}
\rput[Bl](24,28){\small  $H_0$}

\rput[Bl](52.5,28){\small  $H_1$}
\rput[Bl](40.62,28.75){\small  $F_0$}
\rput[Bl](31,33.8){\large $N_{1,2}$}
\rput[Bl](38,22){\large $N_{1,1}$}
\rput[Bl](67,28){\small  $H_{*+k}$}

\rput[Bl](78.12,15.12){\scriptsize $S_{*+k}^-$}
\rput[bl](81.25,33){\large $R_{k,2}$}
\rput[bl](90,24.38){\large $R_{k,1}$}
\rput[bl](44.88,43.88){\large $D_0^+$}
\rput[bl](28,15.88){\scriptsize $S_0^-$}

\rput[bl](38.62,15.12){\small  $T_0^-$}
\rput[bl](52.75,15.12){\scriptsize $S_1^-$}
\rput[bl](42.11,5){\small  $E_0^-$}
\rput[Bl](78,30){\small  $F_{*+k}$}
\end{pspicture}


\vspace{0.5cm}
\caption{Parties communes des d\'ecompositions simpliciales de $\epsilon$ pour tous les  Types}
==================================

\vfill\eject

Les quatre figures suivantes sont les d\'etails de la partie centrale de la figure ci-dessus, pour le d\'ebut de la longue relation.

\vspace{0.5cm}

\ifx\JPicScale\undefined\def\JPicScale{1.5}\fi
\psset{unit=\JPicScale mm}
\psset{linewidth=0.3,dotsep=1,hatchwidth=0.3,hatchsep=1.5,shadowsize=1,dimen=middle}
\psset{dotsize=0.7 2.5,dotscale=1 1,fillcolor=black}
\psset{arrowsize=1 2,arrowlength=1,arrowinset=0.25,tbarsize=0.7 5,bracketlength=0.15,rbracketlength=0.15}

\begin{pspicture}(14,10)(100.62,38)
\psline(20.62,10)(20.62,30)
\psdots[dotstyle=square](20.62,10)
(20.62,30)
\psline(20.62,30)(40,30)
\psdots[dotstyle=square](20.62,30)
(40,30)
\psline(40,30)(40,10)
\psdots[dotstyle=square](40,30)
(40,10)
\psline(40,10)(20.62,10)
\psdots[dotstyle=square](40,10)
(20.62,10)
\psline(60,30)(40,30)
\psdots[dotstyle=square](60,30)
(40,30)
\psline(60.62,10)(40,10)
\psdots[dotstyle=square](60.62,10)
(40,10)
\psline(80.62,30)(60,30)
\psdots[dotstyle=square](80.62,30)
(60,30)
\psline(80.62,10)(60.62,10)
\psdots[dotstyle=square](80.62,10)
(60.62,10)
\psline[linestyle=dashed,dash=1 1](100.62,10)(110,10)
\psline[linestyle=dashed,dash=1 1](100.62,30)(110,30)
\psline(100.62,30)(80.62,30)
\psdots[dotstyle=square](100.62,30)
(80.62,30)
\psline(100,10)(80.62,10)
\psdots[dotstyle=square](100,10)
(80.62,10)
\psline(60.62,10.62)(60,29.38)
\psline(80.62,10.62)(80.62,29.38)
\psline(100.62,10.62)(100.62,29.38)
\psline(21.25,10.62)(39.38,29.38)
\psline(40.62,10.62)(59.38,29.38)
\psline(60,30)(80,10.68)
\psline(81.25,29.38)(100,10.62)
\rput[bl](26.88,20){\small  $F_0$}
\rput[bl](46.88,20){\small  $F_1$}
\rput[bl](65,20){\small $F_2$}
\rput[bl](85,20){\small $F_3$}
\rput[bl](16,20){\small $H_0$}
\rput[bl](35,20){\small $H_1$}
\rput[bl](56,20){\small $H_2$}
\rput[bl](76,20){\small $H_3$}
\rput[bl](102,20){\small $H_4$}
\rput[bl](23.12,23.75){\large $N_{1,2}$}
\rput[bl](45,23.75){\large  $N_{2,2}$}
\rput[bl](65,23.38){\large  $N'_{1,2}$}
\rput[bl](85.40,23.38){\large $N'_{2,2}$}
\rput[bl](30,12){\large   $N_{1,1}$}
\rput[bl](50.62,12){\large  $N_{2,1}$}
\rput[bl](64,12){\large  $N'_{1,1}$}
\rput[bl](85.62,12){\large  $N'_{2,1}$}
\rput[bl](20,31){\scriptsize $S_0^+$}
\rput[bl](20,6){\scriptsize $S_0^-$}
\rput[bl](40,31){\scriptsize $S_1^+$}
\rput[bl](40,6){\scriptsize $S_1^-$}
\rput[bl](60,31){\scriptsize $S_2^+$}
\rput[bl](60,6){\scriptsize $S_2^-$}
\rput[bl](80,31){\scriptsize $S_3^+$}
\rput[bl](80,6){\scriptsize $S_3^-$}
\rput[bl](100,31){\scriptsize $S_4^+$}
\rput[bl](100,6){\scriptsize $S_4^-$}

\rput[bl](30,31){\small $T_0^+$}
\rput[bl](30,6){\small $T_0^-$}
\rput[bl](50,31){\small $T_1^+$}
\rput[bl](50,6){\small $T_1^-$}
\rput[bl](70,31){\small $T_2^+$}
\rput[bl](70,6){\small $T_2^-$}
\rput[bl](90,31){\small $T_3^+$}
\rput[bl](90,6){\small $T_3^-$}
\end{pspicture}

\vspace{1.5cm}

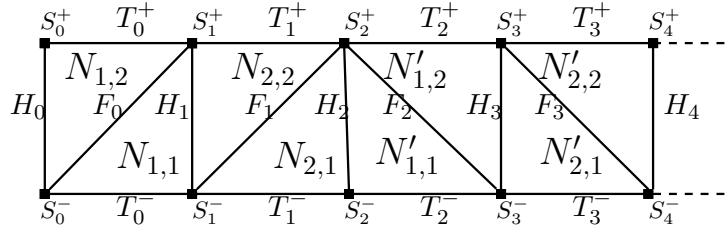
\caption{Partie centrale pour la d\'ecomposition simpliciale de $\epsilon$ pour  le  Type $o_1$ }

==================================



\ifx\JPicScale\undefined\def\JPicScale{1.5}\fi
\psset{unit=\JPicScale mm}
\psset{linewidth=0.3,dotsep=1,hatchwidth=0.3,hatchsep=1.5,shadowsize=1,dimen=middle}
\psset{dotsize=0.7 2.5,dotscale=1 1,fillcolor=black}
\psset{arrowsize=1 2,arrowlength=1,arrowinset=0.25,tbarsize=0.7 5,bracketlength=0.15,rbracketlength=0.15}

\begin{pspicture}(0,8)(90.62,35)
\psline(10.62,10)(10.62,30)
\psdots[dotstyle=square](10.62,10)
(10.62,30)
\psline(10.62,30)(30,30)
\psdots[dotstyle=square](10.62,30)
(30,30)
\psline(30,30)(30,10)
\psdots[dotstyle=square](30,30)
(30,10)
\psline(30,10)(10.62,10)
\psdots[dotstyle=square](30,10)
(10.62,10)
\psline(50,30)(30,30)
\psdots[dotstyle=square](50,30)
(30,30)
\psline(50.62,10)(30,10)
\psdots[dotstyle=square](50.62,10)
(30,10)
\psline(70.62,30)(50,30)
\psdots[dotstyle=square](70.62,30)
(50,30)
\psline(70.62,10)(50.62,10)
\psdots[dotstyle=square](70.62,10)
(50.62,10)
\psline(90.62,30)(70.62,30)
\psdots[dotstyle=square](90.62,30)
(70.62,30)
\psline(90,10)(70.62,30)
\psdots[dotstyle=square](90,10)
(70.62,10)
\psline(50.62,10.62)(50,29.38)
\psline(70,10)(90,10)
\psline(90.62,10.62)(90.62,29.38)
\psline(11.25,10.62)(29.38,29.38)
\psline(70,10)(70,30)

\psline[linestyle=dashed,dash=1 1](90,10.62)(98,10.62)
\psline[linestyle=dashed,dash=1 1](90,30)(98,30)

\rput[bl](16.88,19.38){\small $F_0$}
\rput[bl](34,20.62){\small $F_1$}
\rput[bl](56.88,20){\small $F_2$}
\rput[bl](81,19){\small $F_3$}
\rput[bl](13.12,23.75){\large $N_{1,2}$}
\rput[bl](36,23.75){\large  $N_{2,1}$}
\rput[bl](54.38,23.38){\large $N'_{1,1}$}
\rput[bl](75.62,23.38){\large  $N'_{2,2}$}
\rput[bl](20,14.38){\large  $N_{1,1}$}
\rput[bl](35,13.38){\large  $N_{2,2}$}
\rput[bl](60.62,14.38){\large  $N'_{1,2}$}
\rput[bl](73.62,13.38){\large  $N'_{2,1}$}
\psline(50,10.62)(30.62,29.38)
\psline(51.25,10.62)(70,29.38)
\end{pspicture}
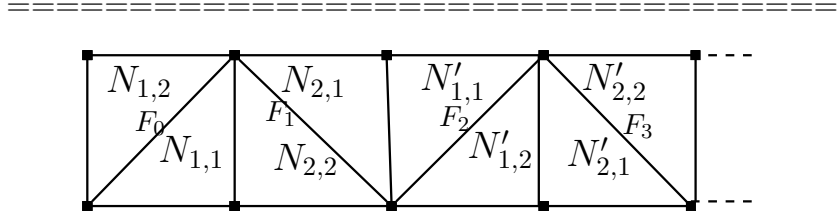
\caption{Partie centrale pour la d\'ecomposition simpliciale de $\epsilon$ pour  le  Type $o_2$}
==================================



\ifx\JPicScale\undefined\def\JPicScale{1.5}\fi
\psset{unit=\JPicScale mm}
\psset{linewidth=0.3,dotsep=1,hatchwidth=0.3,hatchsep=1.5,shadowsize=1,dimen=middle}
\psset{dotsize=0.7 2.5,dotscale=1 1,fillcolor=black}
\psset{arrowsize=1 2,arrowlength=1,arrowinset=0.25,tbarsize=0.7 5,bracketlength=0.15,rbracketlength=0.15}
\begin{pspicture}(-8,-8)(60,30.62)

\psline(0,0)(50,0)(50,25)(0,25)(0,0)
\psdots[dotstyle=square](0,0)
(25,0)(50,0)(50,25)(25,25)(0,25)
\psline(70,0)(120,0)(120,25)(95,25)(70,25)(70,00)
\psdots[dotstyle=square](70,0)
(95,0)(120,0)(120,25)(95,25)(70,25)
\psline(0,0)(25,25)
\psline(25,0)(50,25)
\psline(70,0)(95,25)
\psline(95,25)(120,0)
\psline(25,25)(25,0)
\psline(95,25)(95,0)

\rput[bl](10,12){\small $F_{2j-2}$}
\rput[bl](36,12){\small $F_{2j-1}$}
\rput[bl](81,12){\small $F_{2j-2}$}
\rput[bl](105,12){\small $F_{2j-1}$}

\rput[bl](-8,8){\small $H_{2j-2}$}
\rput[bl](21,8){\small $H_{2j-1}$}
\rput[bl](47,8){\small $H_{2j}$}
\rput[bl](62,8){\small $H_{2j-2}$}
\rput[bl](88,8){\small $H_{2j-1}$}
\rput[bl](120,8){\small $H_{2j}$}

\rput[bl](1,17.5){\bf{\large $N_{j,2}$}}
\rput[bl](26,16.5){\bf{\large $N'_{j,2}$}}
\rput[bl](71,17.5){\bf{\large $N_{j,2}$}}
\rput[bl](108,16.5){\bf{\large $N'_{j,1}$}}

\rput[bl](13,1){\bf{\large $N_{j,1}$}}
\rput[bl](38,1){\bf{\large $N'_{j,1}$}}
\rput[bl](83,1){\bf{\large $N_{j,1}$}}
\rput[bl](100,1){\bf{\large $N'_{j,2}$}}

\rput[bl](18,33){\bf{\large $\varepsilon_j=1$}}
\rput[bl](86,33){\bf{\large $\varepsilon_j=-1$}}


\bigskip


\rput[bl](7,26){\small $T_{2j-2}^ +$}
\rput[bl](35,26){\small  $T_{2j-1}^ +$}
\rput[bl](77,26){\small $T_{2j-2}^ +$}
\rput[bl](104,26){\small  $T_{2j-1}^ +$}
\rput[bl](7,-5){\small $T_{2j-2}^ -$}
\rput[bl](35,-5){\small $T_{2j-1}^ -$}
\rput[bl](77,-5){\small $T_{2j-2}^ -$}
\rput[bl](104,-5){\small $T_{2j-1}^ -$}

\end{pspicture}

\vspace{0.5cm}
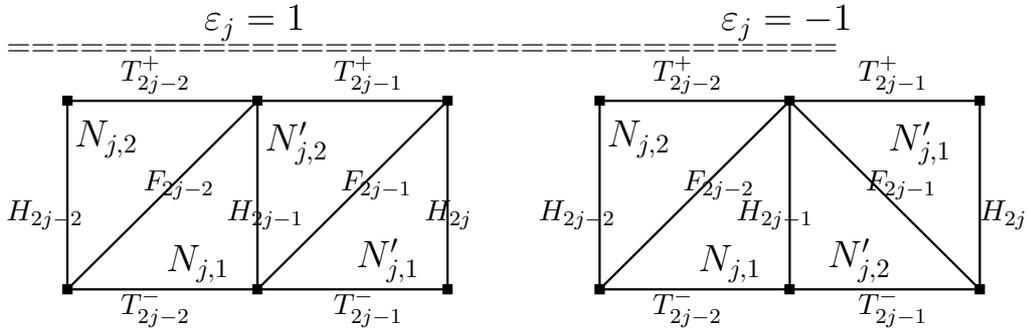
\caption{Partie centrale pour la d\'ecomposition simpliciale de $\epsilon$ pour  les  Types $n_i$}
==================================


\vspace{0.5cm}

\ifx\JPicScale\undefined\def\JPicScale{1}\fi
\psset{unit=\JPicScale mm}
\psset{linewidth=0.3,dotsep=1,hatchwidth=0.3,hatchsep=1.5,shadowsize=1,dimen=middle}
\psset{dotsize=0.7 2.5,dotscale=1 1,fillcolor=black}
\psset{arrowsize=1 2,arrowlength=1,arrowinset=0.25,tbarsize=0.7 5,bracketlength=0.15,rbracketlength=0.15}

\begin{pspicture}(0,-5)(90,69)
\psline(5.62,5)(5.62,60)
\psdots[dotstyle=square](5.62,5)
(5.62,60)
\psline(5.62,60)(65,60)
\psdots[dotstyle=square](5.62,60)
(65,60)
\psline(65,59.38)(65,5)
\psdots[dotstyle=square](65,59.38)
(65,5)
\psline(65,5)(5.62,5)
\psdots[dotstyle=square](65,5)
(5.62,5)
\psline(5.62,60)(25.62,40)
\psdots[dotstyle=square](5.62,60)
(25.62,40)
\psline(25.62,40)(65,60)
\psdots[dotstyle=square](25.62,40)
(65,60)
\psline(25.62,40)(5.62,5)
\psdots[dotstyle=square](25.62,40)
(5.62,5)
\psline(65,5.62)(43.75,16.88)
\psdots[dotstyle=square](65,5.62)
(43.75,16.88)
\psecurve(16.88,-3.75)(5.62,5)(16.88,-3.75)(74.38,-2.5)(83.12,18.12)(84.38,39.38)(80.62,54.38)(65,60)(65,59.38)(65,59.38)(65,59.38)

\rput(3.75,63.12){\scriptsize $S_{k,1}$}
\rput(32.5,62.5){\small $X_{k,1}$}
\rput(62.5,63.12){\scriptsize $S_{k,2}$}
\rput(18,50.62){\small $P_{k,1}^-$}
\rput(25.62,43.12){\scriptsize $C_k^-$}
\rput(41.88,50.62){\small $P_{k,2}^-$}
\rput(2,33.12){\small $X_{k,7}$}
\rput(13,23.75){\small $P_{k,7}^-$}
\rput(28.12,53.75){\bf{\large $M_{k,1}^-$}}
\rput(11.25,36.88){\bf{\large $M_{k,7}^-$}}
\rput(2,3.75){\scriptsize $S_{k,7}$}
\rput(34.38,3.12){\small $Q_k$}
\rput(66.88,1.88){\scriptsize $S_{k,0}$}
\rput(68.12,35){\small $H'_k$}
\psline(5.62,5)(23.75,18.12)
\psdots[dotstyle=square](5.62,5)
(23.75,18.12)
\psline(23.75,18.12)(35,26.88)
\psdots[dotstyle=square](23.75,18.12)
(35,26.88)

\psline(35,26.88)(44.38,36.88)
\psdots[dotstyle=square](35,26.88)
(44.38,36.88)
\psline(44.38,36.88)(52.5,45.62)
\psdots[dotstyle=square](44.38,36.88)
(52.5,45.62)
\psline(52.5,45.62)(65,59.38)
\psdots[dotstyle=square](52.5,45.62)
(65,59.38)
\psline(43.12,16.88)(5.62,5)
\psdots[dotstyle=square](43.12,16.88)
(5.62,5)
\psline(43.12,16.88)(23.75,18.12)
\psdots[dotstyle=square](43.12,16.88)
(23.75,18.12)
\psline(43.75,16.88)(35,26.88)
\psdots[dotstyle=square](43.75,16.88)
(35,26.88)
\psline(43.75,16.88)(44.38,36.88)
\psdots[dotstyle=square](43.75,16.88)
(44.38,36.88)
\psline(43.12,16.88)(52.5,45.62)
\psdots[dotstyle=square](43.12,16.88)
(52.5,45.62)

\rput(42.88,14){\scriptsize $C_k^+$}
\rput(27,9){\small $P_{k,3}^+$}
\rput(30.62,15.62){\small $P_{k,2}^+$}
\rput(19.38,12.5){\bf{\large $M_{k,2}^+$}}
\rput(39.38,24.38){\small $P_{k,1}^+$}
\rput(54.38,13.12){\small $P_{k,4}^+$}
\rput(32.5,21.25){\bf{\large $M_{k,1}^+$}}
\rput(42.88,9){\bf{\large $M_{k,3}^+$}}
\psline(43.75,16.88)(65,60)
\psdots[dotstyle=square](43.75,16.88)
(65,60)
\rput(53.12,31.88){\small $P_{k,5}^+$}
\rput(39.38,-3.75){\bf{\large $R'_{k,1}$}}
\rput(86.88,-2.5){\bf{\large $R'_{k,2}$}}
\rput(86,18.12){\small $G_k$}

\psline(26,39.5)(36,26)
\psline(25.5,39.5)(23.5,17.5)
\psline(26,40)(44,37)
\psline(26,40)(52,45.5)
\end{pspicture}

\vspace{0.5cm}

\caption{D\'ecomposition simpliciale de $\zeta_k$ pour $b_k>0$ et $a_k=5,b_k=2$, $w_{k,2}(q_k,h)=q^3_khq_k^2h=x_{k,1}\cdots x_{k,7}$}

==================================

\ifx\JPicScale\undefined\def\JPicScale{1}\fi
\psset{unit=\JPicScale mm}
\psset{linewidth=0.3,dotsep=1,hatchwidth=0.3,hatchsep=1.5,shadowsize=1,dimen=middle}
\psset{dotsize=0.7 2.5,dotscale=1 1,fillcolor=black}
\psset{arrowsize=1 2,arrowlength=1,arrowinset=0.25,tbarsize=0.7 5,bracketlength=0.15,rbracketlength=0.15}

\begin{pspicture}(0,0)(126.88,68.03)
\psline(10.62,5.62)(10.62,60)
\psline(10.62,60)(66.25,60)
\psline(66.25,60)(66.25,5.62)
\psline(10.62,5.62)(66.25,5.62)
\psdots[dotstyle=square](10.62,60)
(66.25,5.62)
\psdots[dotstyle=square](46.62,24.62)
(54.75,17.12)
\psdots[dotstyle=square](20.62,50)
(28.88,41.75)
\psline(52.12,51.25)(10.62,60)
\psline(52.12,51.25)(20.62,50)
\psline(52.12,51.25)(66.25,60)
\psdots[dotstyle=square](52.12,51.25)
(66.25,60)
\psline(52.12,51.25)(66.25,5.62)
\psline(21.25,18.75)(66.25,5.62)
(66.25,5.62)
\psline(21.25,18.75)(54.75,17.12)
\psline(21.25,18.75)(10.62,5.62)
\psdots[dotstyle=square](21.25,18.75)
(10.62,5.62)
\psline(20.5,49.5)(10.62,5)
\psline[linestyle=dotted](29.4,40.9)(47.3,24.6)
\psline(67,5)(47.3,24.6)
\psline(10.62,60)(29.4,40.9)
\psbezier[](10.62,5)(95,-22)(110,15)(66.25,60)
\rput[bl](100.88,35.62){\bf{\Large $b_k<0$}}
\rput[bl](8.12,61.25){\scriptsize $S_{k,1}$}
\rput[bl](35,61){\small $X_{k,1}$}
\rput[bl](63.12,61.88){\scriptsize  $S_{k,2}$}
\rput[bl](36.25,54){\small $P_{k,1}^-$}
\rput[bl](59,52){\small $P_{k,2}^-$}
\rput[bl](45.25,55){\bf{\large  $M_{k,1}^-$}}
\rput[bl](13.2,50){\scriptsize $S_{k,z_k}$}
\rput[bl](31.88,46){\small $P_{k,z_k}^-$}
\rput[bl](20.61,51){\bf{\large $M_{k,z_k}^-$}}
\rput[bl](46,45){\scriptsize  $C_k^-$}
\rput[bl](67.5,33.12){\small $X_{k,2}$}
\rput[bl](53.5,28.25){\small $P_{k,3}^-$}
\rput[bl](55,40.75){\bf{\large $M_{k,2}^-$}}
\rput[bl](17,20.62){\scriptsize  $C_k^+$}
\rput[bl](5,33.12){\small $H'_k$}
\rput[bl](46,9.5){\bf{\large $M_{k,2}^+$}}
\rput[bl](58,12){\small $X_{k,3}$}
\rput[bl](38.5,8.5){\small $P_{k,2}^+$}
\rput[bl](13.12,13.75){\small $P_{k,1}^+$}
\rput[bl](20,9){\bf{\large $M_{k,1}^+$}}
\rput[bl](5.62,3){\scriptsize  $S_{k,0}$}
\rput[bl](36.88,2.5){\small $Q_k$}
\rput[bl](68.12,4.38){\scriptsize $S_{k,3}$}
\rput[bl](90,17.5){\small $G_k$}
\rput[bl](63.12,-3.75){\bf{\large $R'_{k,1}$}}
\rput[bl](104.38,4.38){\bf{\large $R'_{k,2}$}}

\psline(20.5,18.7)(46.10,24.5)
\psline(21,18.7)(29,42)
\psline(21,18.7)(21,50)
\psline(29,42)(52,51)
\psline(21,50)(52,51)
\psline(46.5,24)(52,51)
\psline(55,17)(52,51)
\end{pspicture}

\vspace{0.5cm}

\caption{D\'ecomposition simpliciale de $\zeta_k$ pour $b_k<0$ }

==================================
\vspace{1.5cm}

\ifx\JPicScale\undefined\def\JPicScale{1}\fi
\psset{unit=\JPicScale mm}
\psset{linewidth=0.3,dotsep=1,hatchwidth=0.3,hatchsep=1.5,shadowsize=1,dimen=middle}
\psset{dotsize=0.7 2.5,dotscale=1 1,fillcolor=black}
\psset{arrowsize=1 2,arrowlength=1,arrowinset=0.25,tbarsize=0.7 5,bracketlength=0.15,rbracketlength=0.15}

\begin{pspicture}(0,-3)(90,57)
\rput{0}(46.08,42.98){\psellipse[](0,0)(11.11,11.11)}
\rput{0}(66.09,12.93){\psellipse[](0,0)(11.5,11.5)}
\psecurve(48.12,24.38)(64.38,24.38)(48.12,24.38)(44.38,10.62)(53.75,-1.25)(77.5,-1.88)(82.5,21.25)(73.12,30.62)(58.75,33.12)(49.38,32.5)(49.38,32.5)(49.38,32.5)
\psline(49.38,31.88)(64.38,24.38)
\psline(64.38,24.38)(66.25,12.5)
\psline(49.38,32.5)(45.62,43.12)
\rput(46.25,45){\scriptsize $C_k$}
\rput(31,45){\small $X_{k,1}$}
\rput(69,12){\scriptsize $C_k^+$}
\rput(84,2.5){\small $G_k$}
\rput(65.62,39.38){\bf{\large $R'_{k,2}$}}
\rput(49,15){\bf{\large $R'_{k,1}$}}
\rput(68,17){\small $P_{k,1}^+$}
\rput(61.88,6.25){\bf{\large $M_{k,1}^+$}}
\rput(44.38,37.5){\small $P_{k,1}^-$}
\rput(44.38,58){\bf{\large $M_{k,1}^-$}}
\rput(91.88,33.75){\bf{\large $b_k=0$}}
\rput(59,30){\bf{\small $H'_k$}}
\rput(80,17){\bf{\small $Q_k$}}
\psdots[dotstyle=square](49.38,31.88)(64.38,24.38)(66.25,12.5)(45.62,43.12)
\end{pspicture}

\vspace{1cm}

\caption{D\'ecomposition simpliciale de $\zeta_k$ pour $b_k=0$ }

\vspace{1cm}

\author{Anne Bauval}\\
\address{\small Institut de Math\'ematiques de Toulouse\\
Equipe Emile Picard, UMR 5580\\
Universit\'e Toulouse III\\
118 Route de Narbonne, 31400 Toulouse - France\\
e-mail: bauval@math.univ-toulouse.fr}

\author{Claude Hayat}\\
\address{\small Institut de Math\'ematiques de Toulouse\\
Equipe Emile Picard, UMR 5580\\
Universit\'e Toulouse III\\
118 Route de Narbonne, 31400 Toulouse - France\\
e-mail: hayat@math.univ-toulouse.fr}
\end{document}